\documentclass[reqno]{amsart}
\usepackage{amsthm}
\usepackage{amssymb}
\usepackage{amsfonts}
\usepackage{mathrsfs}
\usepackage[all,cmtip]{xy}
\usepackage{tikz}   
\usepackage{tikz-cd}  
\usepackage{multirow}
\usepackage{url}
\usepackage{graphicx}
\usepackage{epstopdf}
\usepackage{mathtools}
\usepackage{leftidx}
\usepackage{enumerate}
\usepackage{enumitem}
\usepackage{wrapfig}

\renewcommand{\epsilon}{\varepsilon}

\newcommand{\B}{\mathcal B}

\newcommand{\F}{\mathcal{F}}

\newcommand{\V}{\mathcal{V}}

\newcommand{\C}{\mathbb C}
\newcommand{\p}{\mathbb P}
\renewcommand{\P}{\mathbb P}
\newcommand{\A}{\mathbb A}
\newcommand{\Z}{\mathbb{Z}}
\newcommand{\R}{\mathbb{R}}

\newcommand{\bL}{\mathbb{L}}
\newcommand{\bH}{\mathbb H}
\newcommand{\bU}{\mathbb{U}}
\renewcommand{\l}{\ell_{\infty}}

\DeclareMathOperator{\Id}{Id}

\DeclareMathOperator{\Rep}{Rep}

\DeclareMathOperator{\GL}{GL}
\DeclareMathOperator{\PGL}{PGL}

\DeclareMathOperator{\Res}{Res}

\DeclareMathOperator{\Spec}{Spec}
\DeclareMathOperator{\Proj}{Proj}

\DeclareMathOperator{\Fun}{Fun}
\DeclareMathOperator{\rk}{rk}
\DeclareMathOperator{\Hom}{Hom}

\DeclareMathOperator{\End}{End}
\DeclareMathOperator{\Aut}{Aut}
\DeclareMathOperator{\Set}{Set}
\DeclareMathOperator{\Top}{Top}
\DeclareMathOperator{\Var}{Var}
\DeclareMathOperator{\Sch}{Sch}
\DeclareMathOperator{\Gpd}{Gpd}
\DeclareMathOperator{\Gr}{Gr}
\DeclareMathOperator{\Tot}{Tot}
\DeclareMathOperator{\Mat}{Mat}
\DeclareMathOperator{\Fl}{Fl}
\DeclareMathOperator{\Nat}{Nat}
\DeclareMathOperator{\modl}{mod}
\DeclareMathOperator{\gr}{gr}

\DeclareMathOperator\HOM{\mathcal{H}\mathit{om}}

\DeclareMathOperator{\tr}{tr}

\DeclareMathOperator{\Ext}{Ext}

\newcommand{\gl}{\mathfrak{gl}}
\renewcommand{\ltrans}[1]{\leftidx{^\mathrm{t}}{#1}{}}

\newcommand{\cA}{\mathcal A}

\newcommand{\cC}{\mathcal C}
\newcommand{\cD}{\mathcal D}
\newcommand{\cE}{\mathcal E}
\newcommand{\cF}{\mathcal F}
\newcommand{\cG}{\mathcal G}

\newcommand{\cI}{\mathcal I}

\newcommand{\cL}{\mathcal L}
\newcommand{\cM}{\mathcal M}

\newcommand{\cO}{\mathcal O}

\newcommand{\cQ}{\mathcal Q}
\newcommand{\cR}{\mathcal R}
\newcommand{\cT}{\mathcal T}
\newcommand{\cU}{\mathcal{U}}

\newcommand{\cX}{\mathcal X}

\newcommand\opcat[1]{{#1}^{\mathrm{op}}}
\newcommand\inv[1]{{#1}^{\text{-}1}}

\newcommand{\Bun}{\mathcal{B}un}

\newcommand{\fg}{\mathfrak g}

\newcommand{\fgl}{\mathfrak{gl}}
\newcommand{\fp}{\mathfrak{p}}
\newcommand{\fm}{\mathfrak{m}}
\newcommand{\fM}{\mathfrak{M}}

\renewcommand{\Im}{\mathop{\mathrm{im}}}
\newcommand{\coker}{\mathop{\mathrm{coker}}}

\begin{document}

\date{}
\title[Quiver Representations and Moduli Problems]{Lecture Notes on Quiver Representations and Moduli Problems in Algebraic Geometry}
\author{Alexander Soibelman}
\address{Alexander Soibelman,  Institut des Hautes \'{E}tudes Scientifiques, Bures-sur-Yvette, France}
\email{asoibel@ihes.fr}
\maketitle
\tableofcontents

\newtheorem{thm}{Theorem}[subsection]
\newtheorem{defn}[thm]{Definition}
\newtheorem{lmm}[thm]{Lemma}
\newtheorem{prp}[thm]{Proposition}
\newtheorem{conj}[thm]{Conjecture}
\newtheorem{cor}[thm]{Corollary}
\newtheorem{que}[thm]{Question}
\newtheorem{ack}[thm]{Acknowledgments}
\newtheorem{clm}[thm]{Claim}
\newtheorem*{Kac}{Kac's Theorem}
\newtheorem*{DS}{The Deligne-Simpson Problem}
\newtheorem*{aDS}{The Additive Deligne-Simpson Problem}
\newtheorem*{mDS}{The Multiplicative Deligne-Simpson Problem}

\theoremstyle{definition}
\newtheorem{rmk}[thm]{Remark}
\newtheorem{exa}[thm]{Example}

\section{Introduction}

These notes are based on my Spring 2019 graduate course at the Center for the Quantum Geometry of Moduli Spaces at Aarhus University. There are two main sources of motivation for this course: the representation theory of quivers and moduli problems in algebraic geometry.  
\subsection{Quiver representations}
 A quiver $Q$ is simply a directed graph where loops and multiple arrows between edges are allowed.  A representation of $Q$ over the field $k$ is an assignment of a $k$-vector space to each vertex of $Q$ and a linear mapping between the corresponding vector spaces to each arrow.  Quiver representations of $Q$ over $k$ form a $k$-linear abelian category $\Rep(Q,k)$, with morphisms being linear transformations between vector spaces assigned to the same vertex that commute with the arrows. 

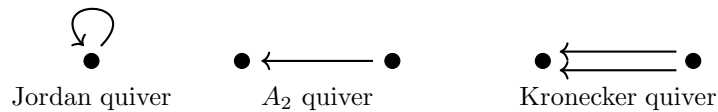
\begin{equation*}
\begin{tikzpicture}[scale=1.0, every node/.style={scale=1.0}]

\draw[fill] (1,1) circle [radius=0.1];
\node [above] (jor) at (1,.95) {};
\draw [thick, <-] (jor) to [out=130,in=45,looseness=10] node[above] {$$}(jor);
\node[] at (1,0.5) {Jordan quiver};

\draw[fill] (3,1) circle [radius=0.1];
\node [right] (a21) at (3.01,1) {};

\draw[fill] (5,1) circle [radius=0.1];
\node [left] (a22) at (4.99,1) {};

\node[] at (4,0.5) {$A_2$ quiver};

\draw [thick, <-] (a21)  -- node [above] {$$} (a22);

\draw[fill] (7,1) circle [radius=0.1];
\node [above] (k211) at (7.1,1) {};
\node [below] (k212) at (7.1,1) {};

\draw[fill] (9,1) circle [radius=0.1];
\node [above] (k221) at (8.9,1) {};
\node [below] (k222) at (8.9,1) {};

\node[] at (8,0.5) {Kronecker quiver};

\draw [thick, <-] (k211)  -- node [above] {$$} (k221);
\draw [thick, <-] (k212) -- node [below] {$$}(k222);
\end{tikzpicture}
\end{equation*}

Quiver representations are an effective combinatorial tool for organizing linear algebraic data.  Not only are they naturally related to many algebraic objects such as quantum groups, Kac-Moody algebras, and cluster algebras, but they have also been studied from the geometric point of view, often serving to bridge the gap between representation theory and algebraic geometry.  In this course we will focus on the geometric aspects of quivers, including moduli spaces, quiver varieties, and the relationships between quiver representations with different moduli spaces in algebraic geometry.

\subsection{Moduli problems}

Given a collection of algebra-geometric objects it is natural to try to classify these objects up to equivalence.  Naively, given a collection $\cA$ of such objects and an equivalence relation $\sim$ on $\cA$, one may ask whether there exists an algebraic variety $X$ whose points (over the base field $k$) correspond to equivalence classes in $\cA / \sim$.  This approach is flawed, since there may be many such varieties and some are better than others at retaining the relationships between the objects being classified.  If, for example, we are interested in classifying all lines through the origin in the complex plane $\C^2$ up to equality, we can view the corresponding equivalence classes as points of the complex projective line $\P^1$, but we can also see them as points in the disjoint union $\A_{\C}^1 \sqcup \{pt\}$ of a line and a point.  Ideally, the points of the variety solving a moduli problem should be configured to reflect the relationship between the geometric objects they parametrize. 

Thus, in more nuanced approach we try to describe equivalence classes of families of objects of $\cA$, rather than just of the objects themselves.  That is, we look at pairs $\cF, T$, consisting of a variety $T$ and a family $\pi: \cF \rightarrow X$ of objects of $\cA$ (the fibers $\inv{\pi}(t)$ are objects of $\cA$) subject to some additional conditions (e.g. that $\pi$ is flat).  Moreover, if the collection of families over a variety $T$ is denoted by $\cA_T$, then for any morphism $f: S \rightarrow T$, there should be a pullback operation assigning to any family $\cF \in \cA_T$ a family $f^*\cF \in \cA_S$.  Now, we can extend our moduli problem by introducing an equivalence relation $\sim_T$ on $\cA_T$ that is compatible with pullback and give us the starting equivalence relation $\sim$ when $T$ is $\Spec k$.      

The solution to such an extended moduli problem, called a \emph{fine moduli space}, consists of a variety $X$ whose $k$-points classify equivalence classes in $\cA/\sim$, together with a \emph{universal family} $\cU \in \cA_X$, which describes how these equivalence classes relate to each other.  More specifically, any family $\cF \in \cA_T$ over a variety $T$ is equivalent to the pullback $f^*\cU$ along a unique morphism $f: T \rightarrow X$.  In the special case that $T = \Spec k$, we obtain that the $k$-points of $X$ are in bijection with equivalence classes in $\cA$.  Furthermore, it turns out that the fine moduli space is unique up to isomorphism.

Considering once again the example of lines through the origin in $\C^2$, we see that a family of such lines over a variety $T$ may be thought of as a line subbundle $\cL \subset \C^2 \times T \rightarrow T$ of the trivial rank $2$ vector bundle over $T$.  The equivalence relation becomes an isomorphism of line subbundles of $\C^2 \times T$.  The solution to the corresponding moduli problem consists of the complex projective line $\P_{\C}^1$ together with the tautological line bundle $\cO_{\P_{\C}^1}(-1)$.  Indeed, if $\cL \subset \C^2 \times T$ is a subbundle, then its dual $\cL^{\vee}$ is generated by global sections (specifically, the images of the standard global sections with respect to the surjection $(\C^2)^{\vee} \times T \twoheadrightarrow \cL^{\vee}$).  This defines a unique morphism $f: T \rightarrow \P_{\C}^1$ such that $\cL \simeq f^*\cO(-1)$.

Unfortunately, it is often the case that, for a given class of objects $\cA$, a fine moduli space either does not exist or requires us to place restrictions on the kinds of objects in $\cA$ we wish to classify.  In order to avoid this, we can forget the universal family and look for a nice enough variety with points in bijection with equivalence classes in $\cA/\sim$.  Alternatively, we can allow for the solution of our moduli problem to no longer be a variety (or even a scheme).  The result is a more complicated object called a \emph{stack}.  

In these notes we will look at several different examples of moduli spaces and stacks of objects arising from quiver representations.

\subsection{Organization}

These notes consist of ten sections, an introduction, and a conclusion.  The sections are meant to be read sequentially, although Section 2 is independent of the rest, while Sections 3-5 and Section 6-11 form self-contained blocks within the larger text.

Section 2 concerns itself with the categorical framework used to define moduli functors, fine moduli spaces, coarse moduli spaces, and moduli stacks.  It is based on Section 2 of \cite{Hos2015} and the beginning of Chapter 2 of \cite{Vi2005}.

Section 3 deals with the construction of the Beilinson spectral sequence.  It contains some introductory material on spectral sequences as well as a proof of Beilinson's theorem.  The main references for this section are Chapter 5 of \cite{We1994} and Section II.3 of \cite{OSS1980}. 

Sections 4 and 5 consist of an application of the Beilinson spectral sequence to give description of vector bundles on $\P^1$ through representations of the Kronecker quiver and framed torsion-free sheaves on $\P^2$ in terms of representations of the ADHM quiver.  The former is a very concrete version of Beilinson's results from \cite{Be1978}, and may be found in \cite{ASo2014}.  The latter follows Chapter 2 of \cite{Na1999}, with additional details provided where deemed necessary.

Section 6 is a basic introduction to actions of algebraic groups on varieties and quotients of varieties by such actions.  It follows Section 3 of \cite{Hos2015}.  

Section 7 serves as a primer on geometric invariant theory.  It covers constructions of the affine GIT quotient and the GIT ``twisted'' by a group character.  The main references for this section are \cite{New1978} and \cite{Muk2003}.

Section 8 applies geometric invariant theory to quiver representations.  The main references for this section are \cite{Bri2012} for introductory material on quivers, \cite{DW2017} for descriptions of GIT quotients for quivers representations, and \cite{King1994} for construction of moduli spaces of stable and semistable quiver representations

Sections 9-11 constitute background material necessary to define Nakajima quiver varieties.  These sections deal with framed quiver representations, symplectic geometry for quiver representations, and quiver varieties themselves, respectively.  The material in these sections is based on Sections 3-5 of \cite{Gi2012}.

\subsection{Preliminaries}
These lecture notes assume some knowledge of algebraic geometry and $C^{\infty}$ symplectic manifolds.  Terminology from the basic theory of algebraic varieties and well-known results concerning sheaf cohomology will be referenced without citation.  All necessary definitions and theorems may be found in Chapters I-III of \cite{Ha1977}, in \cite{Lee2003}, and in \cite{McDS2017}.

Throughout the text the base field will be $\C$.  While in most situations it will be sufficient to assume the field is algebraically closed and/or has characteristic $0$, we will be working over the complex numbers in order to simplify exposition.  We use $\Sch$ to denote the category of schemes over $\C$ and $\Var$ to denote the category of varieties over $\C$ (i.e. reduced, separated schemes of finite type over $\C$).  

While we will be dealing mostly with varieties, we borrow certain notation from the theory of schemes (in fact, many of the results mentioned for varieties will have meaning in the broader context of schemes).  Namely, the one point variety will be denoted by $\Spec \C$ and the sheaf of regular functions on a variety $X$ by $\cO_X$ (omitting the subscript where no confusion arises).  Additionally, we will often use the notions of and notation for (algebraic) vector bundle and locally free sheaf interchangeably.  There is a correspondence between the two (see for example Proposition 1.8.1 in \cite{LP1997}) that defines an equivalence of categories.

\section{Moduli problems and moduli functors}

\subsection{Representable functors}
 We begin by recalling some basic terminology from category theory in order to establish notation and terminology.  A \emph{category} $\cC$ consists of a class of objects denoted by $Ob(\cC)$ and a set of morphisms $\Hom_{\cC}(X,Y)$ for any pair of objects $X,Y \in Ob(\cC)$ (we will drop the subscript $\cC$ and write $X,Y \in \cC$ if there is no ambiguity).  Morphisms come equipped with an associative composition operation for which there exists an identity element $\Id_X \in \Hom(X,X)$ for each $X \in \cC$.  We denote by $\opcat{\cC}$ the category where the objects are the same as for $\cC$, but $\Hom_{\opcat{\cC}}(X,Y) = \Hom_{\cC}(Y,X)$.

Given two categories $\cC$ and $\cD$, a \emph{covariant functor} $F: \cC \rightarrow \cD$ assigns to each object $X \in \cC$ an object $\F(X) \in \cD$ and to each morphism $f \in \Hom(X,Y)$ a morphism $F(f) \in \Hom(F(X), F(Y))$.  A \emph{contravariant functor} $G: \cC \rightarrow \cD$ is defined analogously on the objects, but to each morphism $f \in \Hom(X,Y)$, it assigns a morphism $G(f) \in \Hom(G(Y), G(X))$.  Note that a contravariant functor from $\cC$ to $\cD$ is a covariant functor from $\opcat{\cC}$ to $\cD$.  We will, therefore, refer to ``covariant functors'' just as ``functors''.

If $F,G: \cC \rightarrow \cD$ are two covariant functors, then a \emph{natural transformation} $\eta: F \rightarrow G$ associates to each object $X \in \cC$ a mapping of sets 
$\eta_X: F(X) \rightarrow G(X)$ such that for all $f \in \Hom(X,Y)$ the following diagram commutes:

\begin{equation*}
\begin{tikzcd}
F(X) \arrow[r, "\eta_X"] \arrow[d, swap, "F(f)"] & G(X) \arrow[d, "G(f)"]\\
F(Y) \arrow[r, swap, "\eta_Y"] & G(Y)
\end{tikzcd}
\end{equation*}

The definition is similar for contravariant functors, but the vertical arrows are reversed.  If $\eta_X$ is an isomorphism for all $X \in \cC$, then $\eta$ is called a natural isomorphism between $F$ and $G$.

Denote by $\Set$ the category of sets.  Consider the functor $h_X: \opcat{\cC} \rightarrow \Set$ defined by
\begin{align*}
& h_X  = \Hom(Y,X)\\
& h_X(f): \Hom(Y,X) \rightarrow \Hom(Z,X) \text{ for } Z \rightarrow Y\\
& h_X(f)(g) = g \circ f. 
\end{align*}
Since we will mainly be using $h_X$ in the context of algebraic geometry, we will refer to it as the \emph{functor of points}.

\begin{defn}
Let $\cC$ be a category.  A functor $F: \opcat{\cC} \rightarrow \Set$ is called \textbf{representable} (by $X \in \cC$) if it is naturally isomorphic to $h_X$.
\end{defn}

We will see that the right way of thinking about the solution to a given moduli problem is as a representable functor from the category of schemes to the category of sets.  

\begin{exa}
\label{repex}
Here are a few examples of representable functors. 
\begin{enumerate}
\item 
\begin{enumerate}
\item Let $P: \opcat{\Set} \rightarrow \Set$ be the functor that assigns to each set $S \in \Set$ the power set $P(S)$ (the set of all subsets of $S$). If $f \in \Hom(S,T)$, then $P(f)(T') = \inv{f}(T')$ for all $T' \subset T$.  Consider the two element set $X = \{0,1\}$.  We have that $h_X \cong P$.  Indeed, let the natural transformation $\eta: P \rightarrow h_X$ given by $\eta_S: P(S) \rightarrow \Hom(S,X)$, where $\eta_S(S') = \chi_{S'}$ is the characteristic function for $S'$ in $S$.  Since $\chi_{\inv{f}(T')} = \chi_{T} \circ F$, it follows that $\eta$ is well-defined.  It is easy to see that each $\eta_S$ is bijective, so $\eta$ is a natural isomorphism. 

\item Here is the same example but with added topology.  Let $\Top$ be be category of topological spaces, where the morphisms are continuous functions.  Consider the functor $F: \opcat{\Top} \rightarrow \Set$, which sends each topological space $S$ to $F(S)$, the set of all open subsets of $S$, with $F(f)$ as in the previous example.  Let $X = \{0,1\}$, as before, but impose on it the topology where the open sets are $\varnothing, \{1\}, \{0,1\}$ ($X$ together with this topology is called the Sierpinski space).  It follows that $g: S \rightarrow X$ is a mapping of topological spaces if and only if $\inv{g}(\{1\})$ is open.  Thus, the same argument as in the previous example shows that $F \cong h_X$.
\end{enumerate}
\item Here are a few examples more closely related to algebraic geometry.  In both of these, $\Var$ denotes the category of algebraic varieties over the field $\C$.
\begin{enumerate}
\item Let $F: \opcat{\Var} \rightarrow \Set$ be defined by associating to the variety $T$, the set $F(T)$ of regular functions $T \rightarrow \C$ (i.e. the global sections of the sheaf of functions on $T$).  If $f: S \rightarrow T$ is a morphism, then define $F(f)(g) := g \circ f$.  Interpreting the elements of $F(T)$ as regular mappings $T \rightarrow \A^1$ to the affine line, it is easy to see that $F \cong h_{\A^1}$.
\item  Let $F^{\times}: \opcat{\Var} \rightarrow \Set$ be defined as in the previous part, but $F^{\times}(T)$ consists only of invertible regular functions $T \rightarrow \C^{\times}$.  In this case, $F^{\times}$ can be represented by the punctured affine line $\A^1 - \{0\}$.
\end{enumerate}
\item Here is the example we saw earlier.  Let $F: \opcat{\Var} \rightarrow \Set$ be defined by $F(T) = \{\text{line subbundles } \cL \subset T \times \C^2 \}/\text{iso. line subbundles}$ and $F(f)(\cL) = f^*\cL$ for $f: S \rightarrow T$.  As mentioned in the introduction, dualizing the inclusion of $\cL$ into the trivial bundle gets us the surjection $q:  T \times (\C^2)^{\vee} \twoheadrightarrow \cL^{\vee}$.  If $\epsilon_0, \epsilon_1$ are the global sections of $T \times (\C^2)^{\vee} \rightarrow T$ corresponding to the standard basis vectors, then $t \mapsto [q(\epsilon_0(t)):q(\epsilon_1(t))]$ defines a morphism $T \rightarrow \P^1$.  This morphism is unique with respect to the property that $f^*\cO(-1) \cong \cL$ (this is a well-known fact concerning morphisms into projective space).  The functor $F$ is therefore represented by the complex projective line $\P^1$.
\end{enumerate}
\end{exa}

\subsection{The Yoneda lemma}
We would like to define moduli problems categorically, formulating them in terms of contravariant functors into $\Set$.  Under this description, the moduli problems that can be solved correspond to representable functors.  Indeed, representable functors give us the two important properties we have already seen in the introduction.  Namely:
\begin{itemize}
\item  A moduli space should have a universal family.
\item A moduli space should be unique.
\end{itemize}

Both of these properties are consequences of the following important categorical result.

\begin{lmm}[The Yoneda Lemma]
Let $F: \opcat{\cC} \rightarrow \Set$ be a functor and let $X \in \cC$.  The mapping $\phi$ between the set $\Nat(h_X,F)$ of natural transformations from $h_X$ to $F$ and $F(X)$ given by $\phi: \eta \rightarrow \eta_X(\Id_X)$ is a bijection.
\end{lmm}
\begin{proof}
Let $f \in \Hom(Y,X)$ be a morphism.  The proof follows from the following commutative diagram. 

\begin{equation*}
\begin{tikzcd}[row sep=20pt, column sep=10pt, every label/.append style={font=\small}]
\Hom(X,X) \arrow[rrr,"\eta_X"] \arrow[ddd,swap,"h_X(f)"] & & &
F(X) \arrow[ddd,"F(f)"]\\
& g \arrow[r,mapsto] \arrow[d,swap, mapsto] & \eta_X(g) \arrow[d,mapsto] &\\
& g \circ f \arrow[r,mapsto] &  \eta_Y(g \circ f) = F(f)(\eta_X(g)) &\\
\Hom(Y,X) \arrow[rrr, swap, "\eta_Y"] & & & F(Y) 
\end{tikzcd}
\end{equation*}

Specifically, if  $g = \Id_X$, we get that $\eta_Y(f) = F(f)(\eta_X(\Id_X))$, so $\eta_Y$ is uniquely determined by $\eta_X(\Id_X)$, making $\phi$ injective.  For any $x \in F(X)$ define $\eta_Y(f) = F(f)(x)$.  It is easy to see that $\eta_X(\Id_X) = x$, so $\phi$ is also surjective. 
\end{proof}

Here are several important consequences of the Yoneda Lemma.  

\subsubsection{Yoneda embedding}
Let $\Fun(\opcat{\cC}, \Set)$ be the category of contravariant functors from $\cC$ to $\Set$.  The objects of this category are contravariant functors and the morphisms are natural transformations.  We can define the following functor:
\begin{align*}
& h: \cC \rightarrow \Fun(\opcat{\cC}, \Set)\\
& h(X) = h_X\\
& (h(f)_Z)(g) = f \circ g \text{ for } f \in \Hom(X,Y), Z \in \cC, \text{ and } g \in \Hom(Z,X). 
\end{align*}

The Yoneda Lemma implies the following:
\begin{cor}[The Yoneda embedding]
The functor $h: \cC \rightarrow \Fun(\opcat{\cC}, \Set)$ is bijective on the corresponding sets of morphisms.  That is, $h$ is fully faithful.
\end{cor}
\begin{proof}
To show $h: \Hom(X,Y) \rightarrow \Hom(h_X, h_Y)$ is a bijection, let $F = h_Y$ in the Yoneda Lemma, noting that $h_Y(X) = \Hom(X,Y)$. 
\end{proof}
The Yoneda embedding allows us to think of objects in a category $\cC$ as functors.

\subsubsection{Uniqueness of objects}

It is not hard to see that if $F: \cC \rightarrow \cD$ is a functor, then for $X,Y \in \cC$ we have $F(X) \cong F(Y)$ if $X \cong Y$.  For fully faithful functors, the converse is also true.

\begin{prp}
Let $F: \cC \rightarrow \cD$ be a fully faithful functor.  For any $X, Y \in \cC$, if $F(X) \cong F(Y)$, then $X \cong Y$.
 \end{prp}  
\begin{proof}
Let $h \in \Hom(F(X), F(Y))$ be an isomorphism.  We have $F(f) = h$ for some $f \in \Hom(X, Y)$. Similarly there is a $g \in \Hom(Y,X)$ such that $F(g) = \inv{h}$.  It follows that $F(f \circ g) = F(f) \circ F(g) = \Id_{F(Y)}$.  Since $F(\Id_Y) = \Id_{F(Y)}$, then $f \circ g = \Id_Y$.  Similarly $g \circ f = \Id_X$.  Consequently, $f$ is an isomorphism between $X$ and $Y$.
\end{proof}

Combining this proposition with the Yoneda embedding results in the following corollary.
\begin{cor}
\label{yonedacor2}
Let $X,Y \in \cC$.  If $F: \opcat{\cC} \rightarrow \Set$ is a functor represented by $X$ and also represented by $Y$, then $X \cong Y$.
\end{cor}

Note it is similarly true that if $F,G: \opcat{\cC} \rightarrow \Set$ are two functors represented by the same object $X \in \cC$, then $F \cong G$.

\subsubsection{Universal objects}

We start off by formalizing the concept of a universal family, appearing in the introduction, for arbitrary contravariant functors.

\begin{defn}
Let $F \in \Fun(\opcat{\cC}, \Set)$.  A \textbf{universal object} for $F$ is a pair $(X,u)$ where $X \in \cC$ and $u \in F(X)$ such that for all $Y \in \cC$ and $y \in F(Y)$ there is a unique morphism $f \in \Hom(Y,X)$ such that $F(f)(u) = y$.
\end{defn}

A useful consequence of the Yoneda Lemma is that universal objects exist exactly when the corresponding functors are representable.

\begin{cor}
\label{yonedacor3}
A functor $F \in \Fun(\opcat{\cC, \Set})$ is representable if and only if it has a universal object $(X,u)$.  In this case, $F$ is represented by $X$.
\end{cor}
\begin{proof}
If $h_X \cong F$ for some object $X \in \cC$, then there exists a natural isomorphism $\eta: h_X \rightarrow F$.  Let $u = \eta_X(\Id_X)$.  By the proof of the Yoneda Lemma, any $y \in F(Y)$ can be written as $y = F(f)(u)$ where $f$ is chosen so that $y = \eta_Y(f)$ (possible because $\eta_Y$ is an isomorphism).  Thus $(X,u)$ is a universal object for $F$.

Conversely, if $(X,u)$ is a universal object of $F$, it suffices to show $h_X \cong F$.  Let $\eta: h_X \rightarrow F$ be the natural transformation corresponding to $u$ in the Yoneda Lemma.  This means $u = \eta_X(\Id_X)$.  Let  $\nu: F \rightarrow h_X$ be the natural transformation given by  $\nu_Y(y) = f$, where $y = F(f)(u)$.  This mapping is well-defined because $f$ is unique.  We have $(\eta_Y \circ \nu_Y)(y) = \nu_Y(f) = F(f)(u) = y$.  Similarly, $(\nu_Y \circ \eta_Y)(f) = \nu_Y(F(f)(u)) = f$. 
\end{proof}

Let us look at what the universal objects are for the representable functors we have considered in Example \ref{repex}.

\begin{enumerate}
\item In both parts of this example the pair consisting of $X  = \{0,1\}$ and $u = \{1\}$ is the universal object.  Indeed, if $S'$ is in $P(S)$ (resp. in $F(S)$), then the characteristic function $\chi_{S'}$ defines a mapping $S \rightarrow X$ such that the pullback of the subset (resp. open subset) $\{1\}$ is $S'$.

\item In part (a), it is not hard to check that the universal object $(\A^1, \Id_{\A^1})$, where the identity function $\Id_{\A^1}$ is the regular function that sends each point of $\A^1$ to the corresponding point of the base field $k$.  Similarly, the universal object in part (b) is $(\A^1, \Id_{\A^1 - \{0\}})$. 

\item As we have already seen, the universal object is $(\P^1, \cO(-1))$.
\end{enumerate}

Let us end this section by giving an example of a functor that is not representable.  We will produce several more examples and discuss concrete problems with representability after we introduce moduli functors in the next section.

\begin{exa}
Let $F: \opcat{\Var} \rightarrow \Set$ be the functor that assigns to each $T \in \Var$ the set $F(T)$ of all open subsets (in the Zariski topology) of $T$.  On morphisms the functor is defined as in Example \ref{repex} (b).  However, unlike in that example, the $F$ is not representable.  Indeed, suppose this were true.  According to Corollary \ref{yonedacor3},  there is a universal object $(X,U)$ for this functor.  Note that the open subset $U \subset X$ is itself a variety over $\C$.  Given two morphisms of varieties,  $f,g: T \rightarrow U$, we can extend these to $f', g': T \rightarrow X$ by composing with the inclusion.  These morphisms both have the property that $\inv{(f')}(U) = \inv{(g')}(U) = T \in F(T)$.  By uniqueness, $f' = g'$ and therefore $f = g$.  Thus, there is only one morphism in $\Hom(T,U)$, making it naturally isomorphic to the functor of points for a variety consisting of a single point.  By Corollary \ref{yonedacor2} the variety $U$ is isomorphic to $\Spec \C$.  This makes $U \subset X$ a $\C$-rational point of $X$, and therefore closed.  Since $(X,U)$ is a universal object, any open subset $V \subset T$ of any variety $T$ is the inverse image of $U$ under some morphism $T \rightarrow X$.  Consequently, every open subset of every variety is simultaneously open and closed.  This is untrue in the Zariski topology (for example $\A^1 - \{0\} \subset \A^1$ is open but not closed). 
\end{exa}

\subsection{Fine moduli spaces}

We are finally ready to put all of the pieces together and define moduli functors.  Let $\cA$  be a set of algebra-geometric objects with an equivalence relation $\sim$, and let $\cC$ be the category $\Var$ of algebraic varieties or $\Sch$, the category of schemes.
Recall from the introduction that for any $T \in \cC$ we denote by $\cA_T$ be the set of (flat) families $\pi: \cF \rightarrow T$ over $T$ of objects in $\cA$.  Furthermore, suppose there exists a pullback operation such that for every $f: S \rightarrow T$ and family $\cF \in \cA_T$, there exists a family $f^*\cF \in \cA_S$ and an equivalence relation $\sim_T$ on $\cA_T$ which satisfy the following conditions:
\begin{itemize}
\item $\Id_T^*\cF  = \cF$ for any family $\cF \in \cA_T$,
\item given a morphism $f: S \rightarrow T$ and families $\cF, \cG \in \cA_T$ such that $\cF \sim_T \cG$, we have $f^*\cF \sim_S f^*\cG$,
\item given a pair of morphisms $f: S \rightarrow T$ and $g: R \rightarrow S$ and a family $\cF \in \cA_T$, we have $(f \circ g)^*\cF \sim_R f^*g^*\cF$,
\item if $T$ is the variety consisting of a single point, then $\cA_T = \cA$ and $\sim_T = \sim$.
\end{itemize} 

\begin{defn}
The \textbf{moduli functor} corresponding to $(\cA, \sim)$ is the functor $\cM: \opcat{\cC} \rightarrow \Set$ defined by $\cM(T) = \cA_T/\sim_T$ and $\cM(f) = f^*$.  If this functor $\cM$ is representable with universal object $(X, \cU)$, then $X$ is called the \textbf{fine moduli space} of objects in $\cA$, and $\cU$ is called a \textbf{universal family} of objects in $\cA$.
\end{defn}

Note that the definition of moduli functor depends not only on $\cA$ and $\sim$, but also on the kinds of families we consider and the way we extend the equivalence relation $\sim$ to $\sim_T$.  Let us look at some examples of moduli functors, fine moduli spaces, and universal families.

\begin{exa}[Projective space]
\label{projfunct}
Projective space can be interpreted as a fine moduli space for the following two moduli functors:
\begin{enumerate}
\item Let $\cM: \opcat{\Var} \rightarrow \Set$ be the functor defined by 
\begin{align*}
& \cM(T) = \{\textrm{line subbundles } \cL \subset \C^{n+1} \times T \} / \textrm{isomorphism of subbundles}\\
& \cM(f)(\cL) = f^*\cL.
\end{align*} 
Note this is just a generalization of the functor from Example \ref{repex}.  As in that example, it is not hard to see that $\cM$ is represented by $\P^n$, and the corresponding universal family is just the tautological line bundle $\cO_{\P^n}(-1)$. 
\item  Here is a different perspective on the same example.  Rather than considering families of lines $\ell$ through the origin in $\C^{n+1}$, we look instead at the dual concept of surjections $(\C^{\vee})^{n+1} \twoheadrightarrow \ell^{\vee}$.  This allows us to define following moduli functor:
\begin{align*}
& \cM: \opcat{\Var} \rightarrow \Set \\
& \cM(T) = \left\{ 
     \begin{tabular}{lll} $\cL$ is a line bundle generated by global sections over T\\ $s_0, \dots, s_n \in H^0(T, \cL)$ is a basis of global sections \end{tabular}
   \right\}  \bigg/  \textrm{iso.}\\
& \cM(f)(\cL, s_0, \dots, s_n) = (f^*\cL, f^*s_0, \dots, f^*s_n),
\end{align*}
where isomorphisms are just isomorphisms of line bundles that send the corresponding $(n+1)$-tuples of global sections to each other.
Note an element in the set $\cM(T)$ is the equivalent to an isomorphism class of the quotient $\C^{n+1} \times T \rightarrow \cL$.
An argument dual to the one in the previous part shows that $\cM$ is still represented by $\P^n$, but the universal family is the hyperplane bundle $\cO_{\P^n}(1)$. 
\end{enumerate}
\end{exa}

\begin{exa}[Quadruples of points in $\P^1$]
Consider the classification problem where $\cA = \{p = (p_1, p_2, p_3, p_4)| p_i \in \P^1 \textrm{ and } p_i \neq p_j \textrm{ for } i \neq j\}$ and $p \sim p'$ if there is an element of $\Aut{\P^1} = \PGL(2,\C)$ mapping $p$ to $p'$.  

Recall that for each $p \in \cA$ there is a unique $\varphi_p \in \Aut(\P^1)$ such that 
\[
\varphi_p(p_1) =0, \varphi_p(p_2) = 1, \varphi_p(p_3) = \infty. 
\]
The image $\varphi_p(p_4)$ is called the \emph{cross-ratio} of $p$, and it is denoted by $\lambda(p) \in \P^1$. If the four points are $p_1 = [x_1:1], p_2 = [x_2:1], p_3 = [x_3:1], p_4 = [x_4:1]$ (i.e. all distinct from $[1:0]$), then the cross-ratio is given by the formula:
\[
\lambda(p) = \frac{(x_2-x_3)(x_4-x_1)}{(x_2-x_1)(x_4-x_3)}.
\]
  It follows that $p \sim (0,1,\infty, \lambda(p_4))$.  Now let 
\[
\cA_T = \{T \times \P^1 \rightarrow T \textrm{ trivial bundle}, \sigma_1, \sigma_2, \sigma_3, \sigma_4 \textrm{ disjoint sections}\},
\]
where ''disjoint sections`` means $\sigma_i(t) \neq \sigma_j(t)$ for all $t \in T$ and $i \neq j$.  Let the equivalence relation be $\sim_T$ defined by $(T \times \P^1, \sigma_1, \sigma_2, \sigma_3, \sigma_4) \sim_T (T \times \P^1, \sigma_1', \sigma_2', \sigma_3', \sigma_4')$ given by isomorphisms $f: T \times \P^1 \rightarrow T \times \P^1$ such that $f \circ \sigma_i = \sigma_i'$.  We can see that over a point $\sim_S$ becomes the equivalence relation $\sim$.

We should technically consider families over $T$ to be $\cF \xrightarrow{\pi} T$ where $\cF$ is a variety, $\pi$ is a flat, proper morphism such that $\inv{\pi}(t) \cong \P^1$ for $ t \in T$.  However, if $\cF$ admits $3$ or more disjoint sections, then $\cF \cong T \times \P^1$ as families over $T$.  This means we may interpret $\cA_T$ are the set of families of elements of $\cA$ over $T$. 

Let $\cM_{0,4}: \opcat{Var} \rightarrow \Set$ be defined on objects by $\cM_{0,4}(T) = \cA_T/\sim_T$.  For $f: S \rightarrow T$ the morphism $\cM_{0,4}(f)$ is given by pullback.  Now, denote $M_{0,4} = \P^1 - \{0,1,\infty\}$ and $\cU  = (M_{0,4} \times \P^1, \tau_1, \tau_2, \tau_3, \tau_4)$, where the four sections of $M_{0,4} \times \P^1$ are $\tau_1(x) = (x,0), \tau_2(x) = (x,1), \tau_3(x) = (x,\infty), \tau_4(x) = (x,x)$.

We can check that the pair $(M_{0,4}, \cU)$ is a universal object for the moduli functor $\cM_{0,4}$.  Indeed, let $\cF  = (T \times \P^1, \sigma_1, \dots, \sigma_4) \in \cA_T$.  Define $f: T \rightarrow M_{0,4}$ by $f = \lambda \circ \sigma$, where $\sigma = (\sigma_1, \cdots, \sigma_4)$ and $\lambda$ sends a quadruple of points to their cross-ratio.  Let $\varphi: T \times \P^1 \rightarrow T \times \P^1$ be the mapping defined by $\varphi(t,x) = (t,\varphi_{\sigma(t)})$.  It is not hard to check that $\varphi$ is a well-defined isomorphism of varieties that gives us $f^*\cU \sim_T \cF$.  Since $f$ is defined in terms of the cross-ratio it is the unique morphism $T \rightarrow M_{0,4}$ with this property. 
\end{exa}

\begin{exa}[The Grassmannian]
\label{grassmannfunct}
Let us consider the following generalization of Example \ref{projfunct}:  
\begin{align*}
& \cG r(d,n): \opcat{\Var} \rightarrow \Set \\
& \cG r(d,n)(T) = \{E \textrm{ is a rank $d$ vector subbundle of } T \times \C^n\} /  \textrm{iso. of subbundles}\\
& \cG r(d,n)(f)(E) = f^*E.  
\end{align*}
This functor is represented by the Grassmannian $\Gr(d,n)$, which parametrizes $d$-dimensional subspaces of the vector space $\C^n$.  The universal family is the tautological bundle $\cT = \{(V,x)| x \in V \} \subset \Gr(d,n) \times \C^n$ where the fiber over each point of $\Gr(d,n)$ is the corresponding subspace in $\C^n$.
 
To show that $(\Gr(d,n))$ is a universal object we need to construct for a given rank $d$ subbundle $E \subset T \times \C^n$ a unique morphism $f: T \rightarrow \Gr(d,n)$ such that $E \simeq  f^*\cT$ as subbundles of $T \times \C^n$.  Let $\{ U_i \rightarrow T\}$  be a covering of $T$ that trivializes $E$ as a subbundle.  That is, $E|_{U_i} \cong U_i \times \C^d \subset U_i \times \C^n = T \times \C^n|_{U_i}$.  This defines an $n \times d$ matrix with entries in $\cO_T(U_i)$.  Therefore, we obtain a regular mapping  $\varphi_i: U_i \rightarrow \Mat_d(n \times d, \C)$ into rank $d$ matrices of size $n \times d$.  

Let $g: \Mat_d(n \times d, \C) \rightarrow \Gr(d,n)$ be the morphism defined by sending each matrix to the point corresponding to its column space.  Let $f_i = g_i \circ \varphi_i$.  The transition functions for $E$ act on each $\phi_i|_{U_i \cap U_j}$ by change of basis.  Therefore, $f_i|_{U_i \cap U_j} = f_j|_{U_i \cap U_j}$.  This means the morphisms $f_i$ glue together to define a morphism $f: T \rightarrow \Gr(d,n)$.  Note that by definition $f(t)$ is just the point of $\Gr(d,n)$ corresponding to the fiber $E_t$ over $t \in T$ of $E$.  This means $f^*\cT = \{(t,x)| x \in E_t \} \subset T \times \C^n$, so $f^*\cT = E$ as a subbundle of $T \times \C^n$.  It is not hard to show that $f$ is unique with this property.  Note that as with the moduli functor represented by $\P^n$, there is a dual construction of $\cG r(d,n)$ that involves quotient bundles rather than subbundles.
\end{exa}

\begin{rmk}
The moduli functor in Example \ref{grassmannfunct} can be generalized even further as follows:
\begin{align*}
& \cF l (d_0, \dots, d_l): \opcat{\Var} \rightarrow \Set \\
& \cF l(d_0, \dots, d_l)(T) = \left\{ 
     \begin{tabular}{lll} a filtration by subbundles\\ $0 = E_{l+1} \subset \cdots \subset E_0 = T \times \C^n$\\ $E_i$ is a subbundle of rank $d_i$ \end{tabular}
   \right\}  \Bigg/  \textrm{iso. of subbundles}\\\\
& \cF l(d_0, \dots, d_l)(f)(E) = f^*E \textrm{ for } f: S \rightarrow T.  
\end{align*}
This functor is represented by the flag variety $\Fl(d_0, \dots, d_l)$ parametrizing flags of subspaces $0  = V_l \subset \cdots \subset V_0 = \C^n$.  The universal bundle is the tautological bundle that assigns to every point of $\Fl(d_0, \dots, d_l)$ the corresponding flag in $\C^n$.
\end{rmk}

\subsection{Problems with representability}
Unfortunately, the circumstances in the previous several examples are uncommon, and, in general, fine moduli spaces do not exist.  Let us consider several examples that demonstrate problems with defining representable functors.

\begin{exa}
\label{vspacenonrep}
\leavevmode
\begin{enumerate}
\item Let us consider the classification problem where $\cA = \{\C\textrm{-vector spaces}\}$ and $\sim$ is vector space isomorphisms.  This can be naturally extended to a moduli functor (see below).  However, even before considering families, we can already see that the equivalence classes in $\cA/\sim$ do not correspond to the points of an algebraic variety.  Indeed, for any $n \in \Z_{\ge 0}$, there is a unique equivalence class in $\cA/\sim$ of vector spaces of dimension $n$. Therefore, any variety classifying $\cA/\sim$ would have to have an open affine subvariety $X$ whose points form a countably infinite subset of $\Z_{\ge 0}$.  By the Noether normalization lemma, this implies there is a finite surjection from $X$ to $\A^d$ for some $d \ge 0$.  Since this morphism is finite, then $d \ge 1$.  However, this is impossible since in this case $\A^d$ is uncountable.

\item Even if we consider a modified version of the classification problem above, and let $\cA$ consist only of $n$-dimensional vector spaces, we will have trouble constructing a fine moduli space.  Indeed, consider the following moduli functor:
\begin{align*}
& \cM(T) = \{E \textrm{ vector bundle over } T \textrm{ of rank } n\}/\textrm{iso.}\\
& \cM(f)(E) = f^*E \textrm{ for } f: S \rightarrow T.
\end{align*}
Since there is only one $n$-dimensional vector space up to isomorphism, the fine moduli space representing this functor should consist of a single point.  Therefore, the universal family should just be the trivial rank $n$ vector bundle over a point.  However, this implies any rank $n$ vector bundle on any variety has to be isomorphic to the trivial one, which is clearly untrue.
\end{enumerate}
\end{exa}

The problem illustrated in the second part of this example can be described as follows.  For any family $\cF \in \cA_T$, the fiber $\cF_t$ is just the pullback along the morphism $g_t: pt \rightarrow T$ from the one point variety that sends the point to $t \in T$.  A family of objects $\cF \in \cA_T$ is called \emph{isotrivial} if any two fibers $\cF_{t_1} \sim \cF_{t_2}$ are equivalent.  

If a moduli functor corresponding to a classification problem is representable, then any isotrivial families in $\cA_T$ must be equivalent to a trivial family (i.e. the pullback of the universal family by a constant morphism) under $\sim_T$.  

Indeed, let $X$ be the fine moduli space, $\cU$ the universal family, and $\cF \in \cA_T$ an isotrivial family.  This means $\cF \sim_T f^*\cU$ for some $f: T \rightarrow X$.  Let $g_t: \Spec \C \rightarrow T$ be the morphism sending the single point of $\Spec \C$ to $t \in T$.  We have $(f\circ g_t)^* \cU \sim \cU_{f(t)}$. Furthermore, $(f \circ g_t)^*\cU \sim g_t^*f^*\cU \sim g_t^*\cF$ for all $t \in T$.  Since $\cF$ is isotrivial, this implies $\cU_{f(t)} \sim \cU_{f(t')}$ for any $t,t'\in T$.  This is only possible if $f$ is constant.

The existence of a nontrivial isotrivial family of objects in $\cA_T$ for some variety $T$ means the moduli functor is not representable.  This is often the case when the objects of $\cA$ being classified have nontrivial automorphisms.  For vector spaces of fixed dimension seen Example \ref{vspacenonrep}, nontrivial automorphisms are responsible for the existence of nontrivial isotrivial families, namely: nontrivial vector bundles.    

The next example demonstrates an even more serious problem with representability.
\begin{exa}
\label{jorquivex}
Let us consider the following moduli problem to classify $n$-dimen\-sion\-al Jordan quiver (see Section 1.1) representations over $\C$.  That is, we are interested in the following moduli functor:
\begin{align*}
& \cM: \opcat{\Var} \rightarrow \Set \\
& \cM(T) = \left\{ 
     \begin{tabular}{ll} rank $n$ vector bundles $E$ on $T$\\ endomorphisms $\varphi: E \rightarrow E$ \end{tabular}
   \right\}  \bigg/ 
     \begin{tabular}{ll}vector bundle iso. commu-\\ting with endomorphisms\end{tabular}\\
& \cM(f)(E) = f^*E \textrm{ for } f: S \rightarrow T.  
\end{align*}
This functor is not representable even in the case when $n = 2$.  Indeed, consider the family $\cF = \left(\A^1 \times \C^2, \begin{pmatrix}1 & z\\ 0 & 1 \end{pmatrix}\right)$, where $z$ is the coordinate on $\A^1$.  Note that for $z \neq 0$ all the fibers of this family are isomorphic, since all of the corresponding endomorphisms have the same Jordan normal form.  However, if $z =  0$, then we get the identity endomorphism on the fiber.  This means all of the fibers of $\cF$ are equivalent except for the fiber over $z = 0$.

Now, supposed $\cM$ is representable with universal object $(X, \cU)$.  Let $f: \A^1 \rightarrow X$ such that $\cF \sim_{\A^1} f^*\cU$.  For morphisms $g: \Spec \C \rightarrow \A^1$ such that $g(\Spec \C) \in \A^1 - \{0\}$ the pullbacks $(g \circ f)^* \cU$ are all equivalent.  As we have already seen in the discussion of isotrivial families, this means $f|_{\A^1 - \{0\}}$ is constant.  Let $f(\A^1 - \{0\}) = x \in X$.  We see that $\inv{f}(x)$ is closed in $\A^1$ and also contains $\A^1 - \{0\}$.  Therefore, we have $\inv{f}(x) = \A^1$, making $f$ constant.  This means $f^* \cU$ is the trivial family, but this is impossible since $\cF_0$ is not equivalent to the other fibers.
\end{exa}

The problem demonstrated in the above example may be summarized as follows.  We say a moduli problem exhibits the \emph{jump phenomenon} if there exists a family $\cF$ over $\A^1$ such that $\cF_{z_1} \sim \cF_{z_2}$ for all $z_1, z_2 \in \A^1 - \{0\}$, but $\cF_0$ is not equivalent to the other fibers.  We can see that the jump phenomenon prevents the existence of a fine moduli space of vector bundles of fixed rank and degree in the following example.

\begin{exa}
Let us consider the moduli problem to classify vector bundles of rank $2$ and degree $0$ on $\P^1$.  Namely, consider the moduli functor
\begin{align*}
& \cM: \opcat{\Var} \rightarrow \Set \\
& \cM(T) = \left\{ 
     \begin{tabular}{lll} vector bundles $E$ on $T \times \P^1$\\ $\rk E = 2$, $\deg E|_{\{t\} \times \P^1} = 0$ \end{tabular}
   \right\}  \bigg/  \textrm{vector bundle iso.}\\
& \cM(f)(E) = f^*E \textrm{ for } f: S \rightarrow T.  
\end{align*}
Recall that the elements of $\Ext^1(\cO_{\P^1}(1), \cO_{\P^1}(-1))$ are in one-to-one correspondence with equivalence classes of extensions.  We can compute 
\[
\Ext^1(\cO_{\P^1}(1), \cO_{\P^1}(-1)) \cong \Hom(\cO_{\P^1}(-1), \cO_{\P^1}(-1))^{\vee} \cong \End(\cO_{\P^1}(-1))^{\vee} \cong \C
\]
using Serre duality. We can also see this by looking at the transition functions of these extensions, which all look like $\begin{pmatrix}\Id & c \\ 0 & \Id \end{pmatrix}$ for some $c \in \C$. 

Identifying $\C$ with $\A^1$ we can obtain a rank $2$ bundle $E$ on $\A^1 \times \P^1$ with $E \in \cM(\A^1)$ such that $E|_{\{z\} \times \P^1}$ is the extension in $\Ext^1(\cO_{\P^1}(1), \cO_{\P^1}(-1))$ corresponding to $z$.  This bundle $E$ is just the universal extensions obtained by pulling back the bundles $\cO_{\P^1}(1)$ and $\cO_{\P^1}(-1)$ to $\A^1 \times \P^1$ and considering their extension corresponding to the transition function $\begin{pmatrix}\Id & z\\ 0 & \Id \end{pmatrix}$.

Note that for any $z \in \A^1 - \{0\}$, the fiber $E_z$ of $E$ considered as a bundle over $\A^1$ is isomorphic to $\cO_{\P^1} \oplus \cO_{\P^1}$.  However, for $z = 0$, the fiber $E_0$ is $\cO_{\P^1}(-1)\oplus \cO_{\P^1}(1)$.  Thus the moduli problem exhibits the jump phenomenon, so no fine moduli space exists. 
\end{exa}
 
\subsection{Solving problems with representability}
Unfortunately, the pathologies seen in the previous section are all too common, and the existence of a fine moduli space representing a given moduli functor is the exception rather than the rule.  There are several potential fixes we can apply to get some kind of solution to a moduli problem in the case that the corresponding moduli functor is not representable.  

\subsubsection{Rigidification}

We have seen previously that fine moduli spaces can fail to exist because the objects being classified have automorphisms.  To remedy this we instead attempt to classify these objects together with some additional structure.  We call this additional structure \emph{rigidity}.  The specific procedure used to obtain a representable moduli functor heavily depends on the starting moduli problem.  Here is an example based on classifying vector spaces (see Example \ref{vspacenonrep}).

\begin{exa}
Recall there is no fine moduli space classifying families of $n$-dimensional vector spaces up to isomorphism.  We can modify the corresponding moduli functor as follows:
\begin{align*}
& \cM: \opcat{\Var} \rightarrow \Set\\
& \cM(T) = \left\{ 
     \begin{tabular}{lll} vector bundles $E$ on $T$ of rank $n$,\\ isomorphism $r: T \times \C \rightarrow E$ \end{tabular}
   \right\} \bigg/  
     \begin{tabular}{ll}vector bundle iso.\\ commuting with $r$\end{tabular}\\
& \cM(f)(E) = f^*E \quad \textrm{ such that } f: S \rightarrow T. 		
\end{align*}
\end{exa}
Let $X  = \Spec \C$.  The universal object for the moduli functor $\cM$ is $(X, \Id: X \times \C^n \rightarrow X \times \C^n)$.  Indeed, let $(E,r)$ be a pair such that $E$ is a rank $n$ vector bundle and $r: T \times \C^n \rightarrow E$ is a vector bundle isomorphism.  If $f: T \rightarrow X$ is the only possible (constant) morphism, we see that $(X \times \C^n, \Id)$ pulls back to the trivial bundle $(T \times \C^n, \Id)$ along $f$.  The isomorphism $r$ defines an equivalence between $(T \times \C^n, \Id)$ and $(E, r)$.

\subsubsection{Coarse Moduli Spaces}

In order for a moduli functor to be representable we need to have not only a fine moduli space parametrizing equivalence classes of objects but also a universal family classifying equivalence classes of families.  Even if a universal family does not exist, we can still try and look for a variety $X$ such that the points of $X$ are in bijection with equivalence classes of objects.  This leads to the following definition.

\begin{defn}
Let $\cM: \opcat{Var} \rightarrow \Set$ be a moduli functor.  A \textbf{coarse moduli space} for $\cM$ is a pair $(X, \eta)$ such that $X$ is a variety and $\eta: \cM \rightarrow h_X$ is a natural transformation satisfying
\begin{enumerate}
\item $\eta_T: \cM(T) \rightarrow h_X(T)$ is a bijection for $T  = \Spec \C$
\item For any $Y \in \Var$ and natural transformation $\nu: \cM \rightarrow h_Y$ there exists a unique morphism $f: X \rightarrow Y$ such that $\nu = h(f) \circ \eta$ ($h$ is the Yoneda embedding).
\end{enumerate}
\end{defn}

Unpacking this definition, we see the first property is exactly the requirement that the points of $X$ are in bijection with equivalence classes of objects.  The second property is the universal property.  It implies that $X$ is unique up to unique isomorphism.

\begin{exa}
\label{coarsemodex}
Going back to the moduli functor $\cM$ for $n$-dimensional vector spaces up to isomorphism from Example \ref{vspacenonrep}, we can see that even though $\cM$ is not representable, the pair $(X, \eta)$, where $X = \Spec \C$ and $\eta_T: \cM(T) \rightarrow \Hom(T,X)$ is the constant map, is a coarse moduli space for $\cM$.

Indeed, it is clear that $\eta_X: \cM(X) \rightarrow \Hom(X,X)$ is a bijection, since there is only one $n$-dimensional vector space up to isomorphism.  Now, let $Y$ be a variety and $\nu: \cM \rightarrow h_Y$ a natural transformation.  For all $g: X \rightarrow T$ and $E \in \cM(T)$ we get 
\[
\nu_X(\cM(g)(E)) = \nu_X(X \times \C^n) = \nu_T(E) \circ g.
\]
Let $f = \nu_X(X \times \C^n)$.  The above relation gives us $\nu_T(E) = \nu_X(X \times \C^n) \circ \eta_T(E)$.  Consequently, we have the universal property.
\end{exa} 

Unfortunately, in certain situations the moduli functor behaves so badly that not even a course moduli space exists.  This is case when the jump phenomenon occurs (a slightly modified version of the argument in Example \ref{jorquivex} proves this).

\subsubsection{Stacks}

The last approach to fixing moduli functors we will be discussing is by far the most complicated.  Therefore, we will only give a brief informal description of stacks, and a more rigorous treatment can be found in \cite{Vi2005} or \cite{Ol2016}. 

In all of the examples we have considered so far, moduli spaces have existed (or did not exist) in the category of varieties.  If we relax the condition that a moduli space has to be a variety (or even that it has to be a scheme) we can try to find another type of object to represent our moduli functor.  One way of doing this in a meaningful manner is based on the following analogy.

Let $X$ be a topological space, and let $\cC_X$ be the category where the objects are open subsets of $X$ and morphisms are inclusions of open sets.  A presheaf of sets on $X$ may be viewed as a functor $F: \opcat{\cC}_X \rightarrow \Set$.  The presheaf $F$ is a sheaf if the following two conditions hold.
\begin{enumerate}
\item For any cover $\{U_i\}$ be an open cover of $U \subset X$ is an open subset and any $s,t \in F(U)$ such that $s|_{U_i \cap U_j} = t|_{U_i \cap U_j}$ we have $s = t$.  Here $s|_{V} = F(f)(s) \in F(V)$ where $f: V \rightarrow U$.
 
\item For any open cover $\{U_i\}$ of $U \subset X$ and $s_i \in F(U_i)$ such that $s|_{U_i \cap U_j} = s|_{U_i \cap U_j}$ there exists a $s \in F(U)$ such that $s|_{U_i} = s_i$ for all $i$.
\end{enumerate}

Another way of restating the sheaf conditions is that the inclusions $U_i \rightarrow U$ induce a bijective mapping between $\{s \in F(U)\}$ and $\{(s_i \in F(U_i))| s|_{U_i \cap U_j} = s|_{U_i \cap U_j}\}$.

We can apply the above interpretation to representable functors.  Namely, working in the category of schemes, a presheaf in the Zariski topology is a functor $F: \opcat{\Sch} \rightarrow \Set$.  An open cover of a scheme $U$ is just a covering of $U$  by a collection $\{U_i\}$ of open subschemes.  We say that $F$ is a \emph{sheaf in the Zariski topology} if $F$ satisfies the sheaf conditions listed above for any open cover of any scheme $U$.  It is not hard to see that the ability to glue morphisms in the Zariski topology makes the functor of points $h_X$ a sheaf.  This implies that every representable functor is a sheaf in the Zariski topology.

Note that the converse to this is not true.  It is not enough for $F$ to be a sheaf in the Zariski topology for it to be representable.  An additional condition that amounts to $F$ being ``covered'' by subfunctors represented by affine schemes (see Section 25.15 in \cite{Sta2019}) is necessary for $F$ to be representable.

This idea that representable functors are just sheaves with values in $\Set$ can be generalized by replacing the Zariski topology with a more exotic one.  The key idea is to think of an open cover of $U$ not as a collection of subsets $\{U_i\}$ but instead as a collection of morphisms $\{U_i \rightarrow U\}$ satisfying certain properties. Formalizing this notion leads to the definition of a \emph{Grothendieck topology} on a category (see \cite{Vi2005}).

In this interpretation the Zariski topology may be viewed as a Grothendieck topology on $\Sch$ which assigns to each scheme $U \in \Sch$ an open cover $\{U_i \rightarrow U\}$, where the morphisms are open embeddings (of schemes).

Other important Grothendieck topologies comes from replacing open embeddings with various types of flat morphisms.  For example, the \emph{\'{e}tale topology} considers coverings by \'{e}tale morphisms, the \emph{fppf topology} uses coverings by morphisms that are faithfully flat and locally of finite presentation, and the \emph{fpqc topology} works with coverings by morphisms that are faithfully flat and quasicompact (open quasicompact subsets are images of open quasicompact subsets).

By replacing the Zariski topology on $\Sch$ with one of these (finer) topologies, we can get potentially more interesting sheaves.  For example a sheaf with respect to the \'{e}tale topology on $\Sch$
\[
F: \opcat{\Sch}_{\textrm{ \'{e}tale}} \rightarrow \Set
\] 
together with a ``covering'' condition similar to the one for schemes is called an \emph{algebraic space} (see Chapter 5 in \cite{Ol2016}).

Unfortunately, simply changing the topology on $\Sch$ is often not sufficient if we are trying to solve a moduli problem where the objects have nontrivial automorphisms.  To account for these, we replace the target category $\Sch$ with the category of groupoids $\Gpd$.  The objects of $\Gpd$ are categories in which all of the morphisms are isomorphisms.  A functor
\[
F: \opcat{\Sch}_{\textrm{ fppf}} \rightarrow \Gpd\\
\]
is called a \emph{stack} if it is a sheaf in the fppf-topology.  The gluing procedure necessary for $F$ to be a sheaf is called \emph{decent}.  An additional ``covering'' condition can be used to obtain an \emph{algebraic stack}.  Note that as $\Gpd$ is a category of categories, then $F$ is actually a \emph{pseudofunctor} (also called a \emph{weak $2$-functor}).

It is possible to think of a stack $\cX$ as a \emph{category fibered in groupoids} rather than as a pseudofunctor (\cite{Vi2005}).  In the analogy with sheaves on a topological space this is akin to taking the ``total space''.  The corresponding functor $F$ acts like a section of this space.  

If we are interested in stacks that solve a given moduli problem, we define a moduli functor as follows:
\begin{align*}
& \cM: \opcat{\Sch}_{\textrm{ fppf}} \rightarrow \Gpd \\
& \cM(T) = \langle \textrm{families over $T$} \rangle,
\end{align*}
where the angled brackets $\langle \quad \rangle$ denote a groupoid.  Note that since the groupoid structure accounts for isomorphisms between families, we do not quotient out by the equivalence relation.  

\begin{exa}
\label{stackex}
\leavevmode
\begin{enumerate}
\item Let $X$ be a smooth, projective curve.  The moduli stack of rank $r$ and degree $d$ vector bundles on $X$ is defined by the functor
\begin{align*}
& \Bun_{r,d}(X): \opcat{\Sch}_{\textrm{ fppf}} \rightarrow \Gpd\\\ 
& \Bun_{r,d}(X)(T) = \left \langle \begin{tabular}{lll} vector bundles $E \rightarrow T \times X$,\\ $\rk E = r$ and $\deg E|_{\{t\} \times X} = d$ for all $t \in T$ \end{tabular} \right\rangle
\end{align*}
\item Let $G$ be an affine algebraic group acting on an algebraic variety $X$ (see Section 6 for details).  The \emph{quotient stack} corresponding to this action is defined as follows
\begin{align*}
& [X/G]: \opcat{\Sch}_{\textrm{ fppf}} \rightarrow \Gpd\\\ 
& [X/G](T) = \left \langle \begin{tabular}{llll} $\pi: E \rightarrow T$ principal $G$-bundles on $T$,\\ equivariant morphisms $p: E \rightarrow X$\end{tabular} \right\rangle
\end{align*}
The intuition is that the ``points'' of $[X/G]$ are orbits with respect to the $G$-action and the automorphism groups of these points are just the stabilizers.

To see this is plausible, let us assume that the quotient $X/G$ exists as a variety (in the sense that $X/G$ has an algebraic variety structure and $X \rightarrow X/G$ is a principal $G$-bundle).  In this case, the points of $X/G$ over $T$ can be seen as morphisms $\Hom(T, X/G)$.  Given a morphism $f: T \rightarrow X/G$, the pullback of the $G$-bundle $f^*X$ yields an object in $[X/G](T)$.  Conversely, given an object in $[X/G](T)$ we can construct $f: T \rightarrow X/G$ via gluing using a local trivialization (in the \'{e}tale topology) of the principal $G$-bundle $E$.

In the case that $X = \Spec \C$, the mapping $p$ in the definition of $[X/G](T)$ is irrelevant, and we are left with just $G$-bundles over $T$.  If $G = \GL(n, \C)$, we get
\[
[X/G](T) = \langle \textrm{vector bundles } E \rightarrow T \textrm{ of rank } n \rangle.
\]
This can be seen as the stacky version of the moduli functor for dimension $n$ vector spaces that we have already discussed in Example \ref{vspacenonrep}.  From Remark \ref{stackyquotient} it follows this stack has a single point corresponding to the unique isomorphism class of $n$-dimensional vector spaces. 
\end{enumerate}
\end{exa}

\section{The Beilinson spectral sequence}
\subsection{Spectral sequences}

Before proceeding with further examples of moduli spaces, we would like to prove an important computational result that will be used in the next two sections to identify coherent sheaves with quiver representations.  In order to do this, we begin with a short introduction to spectral sequences, focusing on the spectral sequences used to compute right hyper-derived functors. 

Throughout this section let $\cC$ be an abelian category with enough injectives.  Let $F: \cC \rightarrow \cC$ be a left exact functor.  Recall the construction of the right derived functor of $F$.  Namely, for any $\cE \in \cC$ consider the injective resolution 
\[
0 \rightarrow \cE \rightarrow \cI^0 \rightarrow \cI^1 \rightarrow \cdots.
\]
Applying $F$ we get the chain complex 
\[
0 \rightarrow F(\cI^0) \xrightarrow{d_0} F(\cI^1) \xrightarrow{d_1} \cdots.
\]
The \emph{right derived functor} $R^iF$ is defined by $R^iF(\cE) = \frac{\ker d_i}{\Im d_{i-1}}$.  Note that if $G: \opcat{\cC} \rightarrow \cC$ is a left exact functor, we can compute $R^iG(\cE)$ as the cohomology of the complex obtained from a projective resolution of $\cE$.
\begin{exa}
\leavevmode
\begin{enumerate}
\item  If $f: X \rightarrow Y$ is a morphism of ringed spaces and $F = f_*$ is the direct image functor for $\cO_X$-module, then $R^if_*$ is the higher direct image functor.  If $Y$ consists of a single point, then $R^if_* = H^i$ is the cohomology functor.
\item If $X$ is a ringed space, and $F = \Hom(\cE, -)$ is the $\Hom$ functor for some $\cO_X$-module $\cE$, then $R^iF(\cG) = \Ext^i(\cE, \cG)$ for any $\cO_X$-moduli $\cG$.  Similarly, if $G = \Hom(-, \cE)$, then $R^i(\cG) = \Ext^i(\cG,\cE)$. 
\end{enumerate}
\end{exa}

We would like to be able to obtain an analogous definition of a right derived functor for complexes of objects in $\cC$.  To do this, we must first be able to construct injective resolutions of complexes.  This can be done with the help of the following result (see Section III.7 of \cite{GeMa2003} for proof):

\begin{thm}[Cartan-Eilenberg resolution]
Let $\cC$ be an abelian category with enough injectives, and let $\cE^{\bullet}$ be complex of objects in $\cC$.  There exists a double complex of objects of $\cC$

\begin{equation*}
\begin{tikzcd}
& \vdots & \vdots & \vdots & \\
\cdots \arrow[r] & \cI^{p, 0} \arrow[u, "d^{p, 1}"] \arrow[r, "d^{p, 0}"] & \cI^{p+1, 0} \arrow[u, "d^{p+1, 1}"] \arrow[r, "d^{p+1, 0}"] & \cI^{p+2, 0} \arrow[u, "d^{p+2, 1}"] \arrow[r, "d^{p+2, 0}"] & \cdots\\
\cdots \arrow[r] & \cE^p \arrow[u] \arrow[r] & \cE^{p+1} \arrow[u] \arrow[r] & \cE^{p+2} \arrow[u] \arrow[r] & \cdots
\end{tikzcd}
\end{equation*}

$I^{\bullet \bullet}$ such that $I^{pq} = 0$ for $q <0$ and the entries in the column over any $\cE^i = 0$ are equal to $0$, and furthermore we have
\begin{enumerate}
\item the complex $\ker(d^{p\bullet})$ forms an injective resolution of $\ker(\cE^p \rightarrow \cE^{p+1})$,
\item the complex $\Im(d^{p-1\bullet})$ forms an injective resolution of $\Im(\cE^{p-1} \rightarrow \cE^{p})$.
\item the complex $H^p(\cI^{p \bullet})$ forms an injective resolution of $H^p(\cE^{\bullet})$.
\end{enumerate}
\end{thm}
The double complex $I^{\bullet \bullet}$ is called an \emph{injective Cartan-Eilenberg resolution} of $\cE^{\bullet}$.  If $F: \cC \rightarrow \cC$ is a left exact functor, then the \emph{right hyper-derived functor} of $F$ is defined as
\[
\R^iF(\cE^{\bullet}) := H^i(\Tot F(\cI^{\bullet \bullet})),
\] 
the $i$-th cohomology of the total complex associated to the double complex $\cI^{\bullet \bullet}$.  If $f: X \rightarrow Y$ is a morphism of ringed spaces, then analogous to the higher direct image functor $R^if_*$, we may define $\R^if_*$.  In the case when $Y$ is a point, $\R^if_* = \bH^i$ is called the \emph{hypercohomology} functor.

\begin{rmk}
There is a dual construction of the left hyper-derived functor $\bL^iF$ associated with a right exact functor $F$, which involves a projective Cartan-Eilenberg resolution.  It is entirely analogous to the construction of the left derived functor of $F$ and can been seen in full detail in Chapter 5 of \cite{We1994}.
\end{rmk}

We will need to compute $\R^if_*$ in the category of sheaves.  In order to do this, we will be using  
spectral sequences, which we will briefly introduce below.

\begin{defn}
Let $\cC$ be an abelian category.  A \textbf{cohomology spectral sequence} consists of the data
\begin{enumerate}
\item A family of objects $\{E^{p, q}_r\}$ for $p,q,r \in Z$ and $r \ge 0$.
\item Differentials $d^{p, q}_r: E^{p, q}_r \rightarrow E^{p+r, q-r+1}$ satisfying 
\begin{enumerate}
\item $d^{p+r, q-r+1}_r \circ d^{p, q}_r = 0$
\item $E^{p, q}_{r+1} \cong \ker(d^{p, q}_r)/\Im(d^{p-r, q+r-1}_r)$.
\end{enumerate}
\end{enumerate}
\end{defn}

Cohomological spectral sequences can be visualized as a sequence of ``sheets'' $E_r$ of objects arranged at the vertices of the lattice $\Z^2$.  The first few sheets appear as follows.

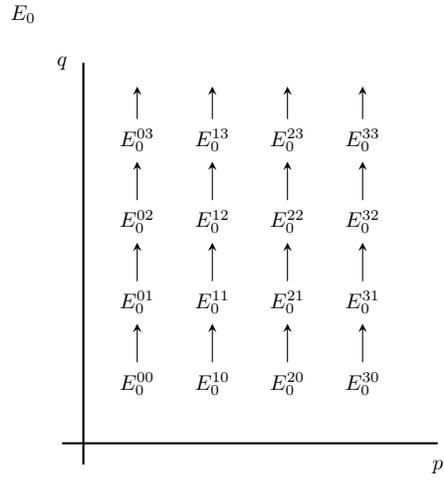
\begin{equation*}
\begin{tikzpicture}[scale=0.8, every node/.style={scale=0.8}]
	\matrix (m) [matrix of math nodes,
    nodes in empty cells,nodes={minimum width=7ex,
    minimum height=7ex,outer sep=-5pt},
    column sep=1ex,row sep=1ex]{
          q     &      &     &     &    & \\
               & E_0^{03}  & E_0^{13} &  E_0^{23}  & E_0^{33}&\\
               &  E_0^{02}  & E_0^{12} &  E_0^{22}  & E_0^{32}&\\   
               &  E_0^{01}  & E_0^{11} &  E_0^{21}  & E_0^{31}&\\
               &  E_0^{00}  & E_0^{10} &  E_0^{20}  & E_0^{30}&\\
    \quad\strut &    &    &    &  & p\strut \\};
    \draw[-stealth] (m-5-2.north) -- (m-4-2.south);
		\draw[-stealth] (m-5-3.north) -- (m-4-3.south);
		\draw[-stealth] (m-5-4.north) -- (m-4-4.south);
		\draw[-stealth] (m-5-5.north) -- (m-4-5.south);
		
		\draw[-stealth] (m-4-2.north) -- (m-3-2.south);
		\draw[-stealth] (m-4-3.north) -- (m-3-3.south);
		\draw[-stealth] (m-4-4.north) -- (m-3-4.south);
		\draw[-stealth] (m-4-5.north) -- (m-3-5.south);
		
		\draw[-stealth] (m-3-2.north) -- (m-2-2.south);
		\draw[-stealth] (m-3-3.north) -- (m-2-3.south);
		\draw[-stealth] (m-3-4.north) -- (m-2-4.south);
		\draw[-stealth] (m-3-5.north) -- (m-2-5.south);
		
		\draw[-stealth] (m-2-2.north) -- (m-1-2.south);
		\draw[-stealth] (m-2-3.north) -- (m-1-3.south);
		\draw[-stealth] (m-2-4.north) -- (m-1-4.south);
		\draw[-stealth] (m-2-5.north) -- (m-1-5.south);
\draw[thick] (m-1-1.east) -- (m-6-1.east) ;
\draw[thick] (m-6-1.north) -- (m-6-6.north) ;

\node [above=5pt, align=flush left,text width=1cm] at (m-1-1.north west)
        {
            $E_0$
        };
\end{tikzpicture}
\end{equation*}

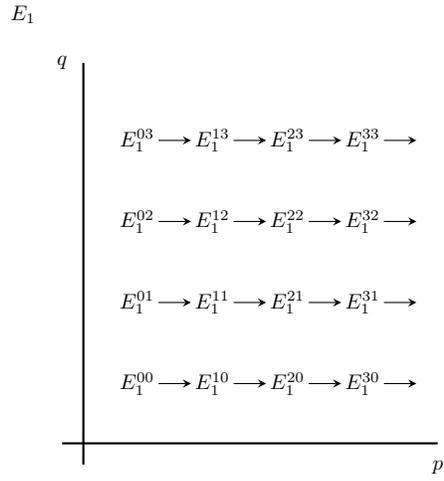
\vspace{40pt}

\begin{equation*}
\begin{tikzpicture}[scale=0.8, every node/.style={scale=0.8}]
	\matrix (m) [matrix of math nodes,
    nodes in empty cells,nodes={minimum width=7ex,
    minimum height=7ex,outer sep=-5pt},
    column sep=1ex,row sep=1ex]{
          q     &      &     &     &    & \\
               & E_1^{03}  & E_1^{13} &  E_1^{23}  & E_1^{33}&\\
               &  E_1^{02}  & E_1^{12} &  E_1^{22}  & E_1^{32}&\\   
               &  E_1^{01}  & E_1^{11} &  E_1^{21}  & E_1^{31}&\\
               &  E_1^{00}  & E_1^{10} &  E_1^{20}  & E_1^{30}&\\
    \quad\strut &    &    &    &  & p\strut \\};
    \draw[-stealth] (m-5-2.east) -- (m-5-3.west);
		\draw[-stealth] (m-5-3.east) -- (m-5-4.west);
		\draw[-stealth] (m-5-4.east) -- (m-5-5.west);
		\draw[-stealth] (m-5-5.east) -- (m-5-6.west |- m-5-5.east);
		
		 \draw[-stealth] (m-4-2.east) -- (m-4-3.west);
		\draw[-stealth] (m-4-3.east) -- (m-4-4.west);
		\draw[-stealth] (m-4-4.east) -- (m-4-5.west);
		\draw[-stealth] (m-4-5.east) -- (m-4-6.west |- m-4-5.east);
		
		\draw[-stealth] (m-3-2.east) -- (m-3-3.west);
		\draw[-stealth] (m-3-3.east) -- (m-3-4.west);
		\draw[-stealth] (m-3-4.east) -- (m-3-5.west);
		\draw[-stealth] (m-3-5.east) -- (m-3-6.west |- m-3-5.east);
		
		\draw[-stealth] (m-2-2.east) -- (m-2-3.west);
		\draw[-stealth] (m-2-3.east) -- (m-2-4.west);
		\draw[-stealth] (m-2-4.east) -- (m-2-5.west);
		\draw[-stealth] (m-2-5.east) -- (m-2-6.west |- m-2-5.east);
		
\draw[thick] (m-1-1.east) -- (m-6-1.east) ;
\draw[thick] (m-6-1.north) -- (m-6-6.north) ;

\node [above=5pt, align=flush left,text width=1cm] at (m-1-1.north west)
        {
            $E_1$
        };
\end{tikzpicture}
\end{equation*}

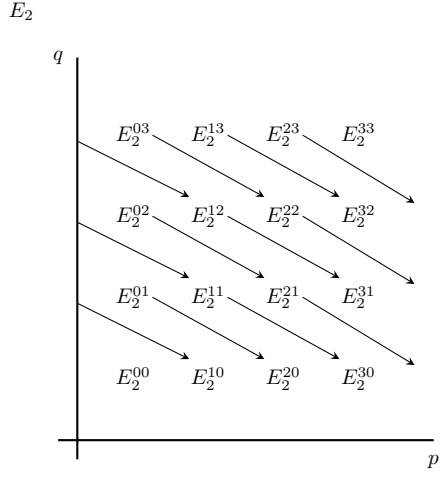
\begin{equation*}
\begin{tikzpicture}[scale=0.8, every node/.style={scale=0.8}]
	\matrix (m) [matrix of math nodes,
    nodes in empty cells,nodes={minimum width=7ex,
    minimum height=7ex,outer sep=-6pt},
    column sep=1ex,row sep=1ex]{
          q     &      &     &     &    & \\
               & E_2^{03}  & E_2^{13} &  E_2^{23}  & E_2^{33}&\\
               &  E_2^{02}  & E_2^{12} &  E_2^{22}  & E_2^{32}&\\   
               &  E_2^{01}  & E_2^{11} &  E_2^{21}  & E_2^{31}&\\
               &  E_2^{00}  & E_2^{10} &  E_2^{20}  & E_2^{30}&\\
    \quad\strut &    &    &    &  & p\strut \\};
 
		\draw[-stealth] (m-4-1.east) -- (m-5-3.north west);
		 \draw[-stealth] (m-4-2.east) -- (m-5-4.north west);
		\draw[-stealth] (m-4-3.east) -- (m-5-5.north west);
		\draw[-stealth] (m-4-4.east) -- (m-5-6.north west);
		
		\draw[-stealth] (m-3-1.east) -- (m-4-3.north west);
		\draw[-stealth] (m-3-2.east) -- (m-4-4.north west);
		\draw[-stealth] (m-3-3.east) -- (m-4-5.north west);
		\draw[-stealth] (m-3-4.east) -- (m-4-6.north west);
		
		\draw[-stealth] (m-2-1.east) -- (m-3-3.north west);
		\draw[-stealth] (m-2-2.east) -- (m-3-4.north west);
		\draw[-stealth] (m-2-3.east) -- (m-3-5.north west);
		\draw[-stealth] (m-2-4.east) -- (m-3-6.north west);
		
\draw[thick] (m-1-1.east) -- (m-6-1.east) ;
\draw[thick] (m-6-1.north) -- (m-6-6.north) ;

\node [above=5pt, align=flush left,text width=1cm] at (m-1-1.north west)
        {
            $E_2$
        };
\end{tikzpicture}
\end{equation*}

A spectral sequence is called \emph{bounded} if for  each $(p, q) \in \Z^2$ there exists a $r_0$ such that $d^{pq}_r = d^{p-r, q+r-1}_r = 0$ for $r \ge r_0$.  This means $E^{p q}_r \cong E^{p q}_{r_0}$ for $r \ge r_0$.  We write $E^{pg}_{\infty}$ for this object.

A bounded spectral sequence \emph{converges} to $H^\bullet$ if for each $H^n$ there exist a decreasing finite filtration
\[
  H^n = F^{M}H^n \supset \cdots \supset F^{p}H^n \supset F^{p+1}H^n \supset \cdots \supset F^{m}H^n =0
\]
such that $E^{pq}_r \cong F^pH^{p+q}/F^{p+1}H^{p+q}$.  We write $E^{pq}_r \Rightarrow H^{p+q}$.

A spectral sequence \emph{collapses} at $E_r$ if there is exactly one nonzero row or column in $\{E^{p q}_r\}$.  If such a spectral sequence converges to $H^{\bullet}$, then we can obtain $H^n$ as the unique nonzero $E^{pq}_r$ with $p+q = n$.

A \emph{filtration} of a cochain complex $C^{\bullet}$ is an ordered family $F^\bullet C^\bullet$ of chain subcomplexes
\[
 C^{\bullet} \supset \cdots \supset F^{l-1} C^{\bullet} \supset F^l C^{\bullet} \supset \cdots
\]
Such a filtration is called \emph{bounded} if for each $n$ there are $m < M$ such that $F^mC^n = 0$ and $F^MC^n = C^n$.  We will need the following convergence theorem for spectral sequence (see Theorem 5.5.1 in \cite{We1994} for proof).

\begin{thm}
A filtration $F^\bullet C^\bullet$ of a cochain complex $C^\bullet$ naturally determines a spectral sequence $\{E^{pq}_r\}$ such that $E^{pq}_0 = F^pC^{p+q}/F^{p+1}C^{p+q}$ and $E^{pq}_1 = H^{p+q}(E^{p\bullet}_0)$.  If this filtration is bounded, then the spectral sequence is bounded and converges to  the cohomology $H^{\bullet}(C^\bullet)$.  That is, $E^{pq}_r \Rightarrow H^{p+q}(C^\bullet)$.
\end{thm}

The two spectral sequences we are interested in are associated to the filtrations of a double complex.

\begin{exa}[Filtrations of a double complex]
If $C^{\bullet \bullet}$ is a double complex, then it has two natural filtrations: one by rows and one by columns.  More precisely, there is a filtration on the total complex $\Tot(C^{\bullet \bullet})$ of $C^{\bullet \bullet}$ where $F^n\Tot(C^{\bullet \bullet})$ is given by $\Tot(D^{\bullet \bullet}_n)$ , where 
\[
D^{pq}_n = \begin{cases} C^{pq} & \textrm{ for } p \ge n\\ 0 & \textrm{ for } p < n. \end{cases}
\]
This results in an $E_0$ that looks as follows (the connecting maps are just the vertical differentials of $C^{\bullet \bullet}$).

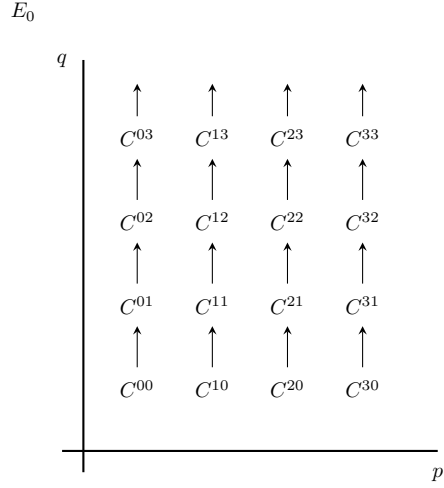
\begin{equation*}
\begin{tikzpicture}[scale=0.8, every node/.style={scale=0.8}]
	\matrix (m) [matrix of math nodes,
    nodes in empty cells,nodes={minimum width=7ex,
    minimum height=7ex,outer sep=-5pt},
    column sep=1ex,row sep=1ex]{
          q     &      &     &     &    & \\
               & C^{03}  & C^{13} &  C^{23}  & C^{33}&\\
               &  C^{02}  & C^{12} &  C^{22}  & C^{32}&\\   
               &  C^{01}  & C^{11} &  C^{21}  & C^{31}&\\
               &  C^{00}  & C^{10} &  C^{20}  & C^{30}&\\
    \quad\strut &    &    &    &  & p\strut \\};
    \draw[-stealth] (m-5-2.north) -- (m-4-2.south);
		\draw[-stealth] (m-5-3.north) -- (m-4-3.south);
		\draw[-stealth] (m-5-4.north) -- (m-4-4.south);
		\draw[-stealth] (m-5-5.north) -- (m-4-5.south);
		
		\draw[-stealth] (m-4-2.north) -- (m-3-2.south);
		\draw[-stealth] (m-4-3.north) -- (m-3-3.south);
		\draw[-stealth] (m-4-4.north) -- (m-3-4.south);
		\draw[-stealth] (m-4-5.north) -- (m-3-5.south);
		
		\draw[-stealth] (m-3-2.north) -- (m-2-2.south);
		\draw[-stealth] (m-3-3.north) -- (m-2-3.south);
		\draw[-stealth] (m-3-4.north) -- (m-2-4.south);
		\draw[-stealth] (m-3-5.north) -- (m-2-5.south);
		
		\draw[-stealth] (m-2-2.north) -- (m-1-2.south);
		\draw[-stealth] (m-2-3.north) -- (m-1-3.south);
		\draw[-stealth] (m-2-4.north) -- (m-1-4.south);
		\draw[-stealth] (m-2-5.north) -- (m-1-5.south);
\draw[thick] (m-1-1.east) -- (m-6-1.east) ;
\draw[thick] (m-6-1.north) -- (m-6-6.north) ;

\node [above=5pt, align=flush left,text width=1cm] at (m-1-1.north west)
        {
            $E_0$
        };
\end{tikzpicture}
\end{equation*}

The corresponding $E_1$ comes from maps induced on cohomology by the horizontal differentials.

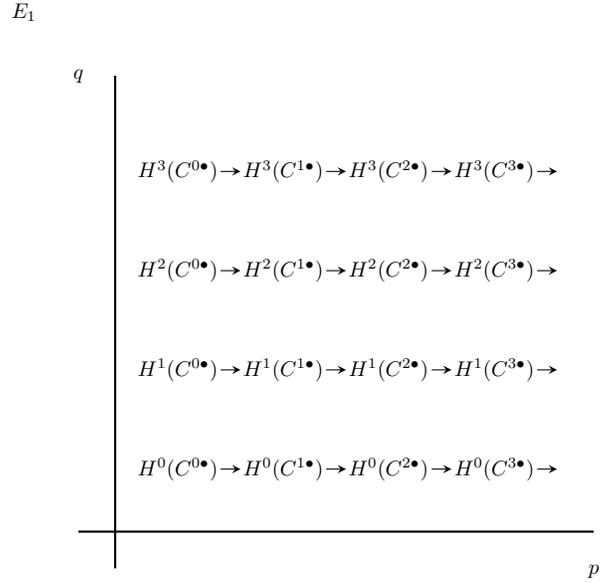
\begin{equation*}
\begin{tikzpicture}[scale=0.8, every node/.style={scale=0.8}]
	\matrix (m) [matrix of math nodes,
    nodes in empty cells,nodes={minimum width=9ex,
    minimum height=9ex,outer sep=-2pt},
    column sep=1ex,row sep=1ex]{
          q     &      &     &     &    & \\
               & H^3(C^{0\bullet})  & H^3(C^{1\bullet}) &  H^3(C^{2\bullet})  & H^3(C^{3\bullet}) &\\
               &  H^2(C^{0\bullet})  & H^2(C^{1\bullet}) &  H^2(C^{2\bullet})  & H^2(C^{3\bullet}) &\\   
               &  H^1(C^{0\bullet})  & H^1(C^{1\bullet}) &  H^1(C^{2\bullet})  & H^1(C^{3\bullet}) &\\
               &  H^0(C^{0\bullet})  & H^0(C^{1\bullet}) &  H^0(C^{2\bullet})  & H^0(C^{3\bullet}) &\\
    \quad\strut &    &    &    &  & p\strut \\};
    \draw[-stealth] (m-5-2.east) -- (m-5-3.west);
		\draw[-stealth] (m-5-3.east) -- (m-5-4.west);
		\draw[-stealth] (m-5-4.east) -- (m-5-5.west);
		\draw[-stealth] (m-5-5.east) -- (m-5-6.west |- m-5-5.east);
		
		 \draw[-stealth] (m-4-2.east) -- (m-4-3.west);
		\draw[-stealth] (m-4-3.east) -- (m-4-4.west);
		\draw[-stealth] (m-4-4.east) -- (m-4-5.west);
		\draw[-stealth] (m-4-5.east) -- (m-4-6.west |- m-4-5.east);
		
		\draw[-stealth] (m-3-2.east) -- (m-3-3.west);
		\draw[-stealth] (m-3-3.east) -- (m-3-4.west);
		\draw[-stealth] (m-3-4.east) -- (m-3-5.west);
		\draw[-stealth] (m-3-5.east) -- (m-3-6.west |- m-3-5.east);
		
		\draw[-stealth] (m-2-2.east) -- (m-2-3.west);
		\draw[-stealth] (m-2-3.east) -- (m-2-4.west);
		\draw[-stealth] (m-2-4.east) -- (m-2-5.west);
		\draw[-stealth] (m-2-5.east) -- (m-2-6.west |- m-2-5.east);
		
\draw[thick] (m-1-1.east) -- (m-6-1.east) ;
\draw[thick] (m-6-1.north) -- (m-6-6.north) ;

\node [above=5pt, align=flush left,text width=1cm] at (m-1-1.north west)
        {
            $E_1$
        };
\end{tikzpicture}
\end{equation*}

If the original filtration was bounded, then we get 
\[
E^{pq}_2 = H_h^pH^q_v(C^{\bullet \bullet}) \Rightarrow H^{p+q}(\Tot C^{\bullet \bullet}),
\]
where cohomology is computed using the vertical differentials first, and then using the horizontal differentials.

Similarly, we can define a filtration on the double complex by rows.  Namely, $F^n\Tot(C^{\bullet \bullet})$ is given by $\Tot(D^{\bullet \bullet}_n)$, where
\[
D^{pq}_n = \begin{cases} C^{pq} & \textrm{ for } q \ge n\\ 0 & \textrm{ for } q < n. \end{cases}
\]
This gives us an $E_0$ that looks as follows (the connecting maps are the horizontal differentials).

\begin{equation*}
\begin{tikzpicture}[scale=0.8, every node/.style={scale=0.8}]
	\matrix (m) [matrix of math nodes,
    nodes in empty cells,nodes={minimum width=7ex,
    minimum height=7ex,outer sep=-5pt},
    column sep=1ex,row sep=1ex]{
          q     &      &     &     &    & \\
               & C^{30}  & C^{31} &  C^{32}  & C^{33}&\\
               &  C^{20}  & C^{21} &  C^{22}  & C^{23}&\\   
               &  C^{10}  & C^{11} &  C^{12}  & C^{13}&\\
               &  C^{00}  & C^{01} &  C^{02}  & C^{03}&\\
    \quad\strut &    &    &    &  & p\strut \\};
    \draw[-stealth] (m-5-2.north) -- (m-4-2.south);
		\draw[-stealth] (m-5-3.north) -- (m-4-3.south);
		\draw[-stealth] (m-5-4.north) -- (m-4-4.south);
		\draw[-stealth] (m-5-5.north) -- (m-4-5.south);
		
		\draw[-stealth] (m-4-2.north) -- (m-3-2.south);
		\draw[-stealth] (m-4-3.north) -- (m-3-3.south);
		\draw[-stealth] (m-4-4.north) -- (m-3-4.south);
		\draw[-stealth] (m-4-5.north) -- (m-3-5.south);
		
		\draw[-stealth] (m-3-2.north) -- (m-2-2.south);
		\draw[-stealth] (m-3-3.north) -- (m-2-3.south);
		\draw[-stealth] (m-3-4.north) -- (m-2-4.south);
		\draw[-stealth] (m-3-5.north) -- (m-2-5.south);
		
		\draw[-stealth] (m-2-2.north) -- (m-1-2.south);
		\draw[-stealth] (m-2-3.north) -- (m-1-3.south);
		\draw[-stealth] (m-2-4.north) -- (m-1-4.south);
		\draw[-stealth] (m-2-5.north) -- (m-1-5.south);
\draw[thick] (m-1-1.east) -- (m-6-1.east) ;
\draw[thick] (m-6-1.north) -- (m-6-6.north) ;

\node [above=5pt, align=flush left,text width=1cm] at (m-1-1.north west)
        {
            $E_0$
        };
\end{tikzpicture}
\end{equation*}

The corresponding $E_1$ comes from maps induced on cohomology by the vertical differentials.

\begin{equation*}
\begin{tikzpicture}[scale=0.8, every node/.style={scale=0.8}]
	\matrix (m) [matrix of math nodes,
    nodes in empty cells,nodes={minimum width=9ex,
    minimum height=9ex,outer sep=-2pt},
    column sep=1ex,row sep=1ex]{
          q     &      &     &     &    & \\
               & H^3(C^{\bullet 0})  & H^3(C^{\bullet 1}) &  H^3(C^{\bullet 2})  & H^3(C^{\bullet 3}) &\\
               &  H^2(C^{0\bullet})  & H^2(C^{1\bullet}) &  H^2(C^{2\bullet})  & H^2(C^{3\bullet}) &\\   
               &  H^1(C^{\bullet 0})  & H^1(C^{\bullet 1}) &  H^1(C^{\bullet 2})  & H^1(C^{\bullet 3}) &\\
               &  H^0(C^{\bullet 0})  & H^0(C^{\bullet 1}) &  H^0(C^{\bullet 2})  & H^0(C^{\bullet 3}) &\\
    \quad\strut &    &    &    &  & p\strut \\};
    \draw[-stealth] (m-5-2.east) -- (m-5-3.west);
		\draw[-stealth] (m-5-3.east) -- (m-5-4.west);
		\draw[-stealth] (m-5-4.east) -- (m-5-5.west);
		\draw[-stealth] (m-5-5.east) -- (m-5-6.west |- m-5-5.east);
		
		 \draw[-stealth] (m-4-2.east) -- (m-4-3.west);
		\draw[-stealth] (m-4-3.east) -- (m-4-4.west);
		\draw[-stealth] (m-4-4.east) -- (m-4-5.west);
		\draw[-stealth] (m-4-5.east) -- (m-4-6.west |- m-4-5.east);
		
		\draw[-stealth] (m-3-2.east) -- (m-3-3.west);
		\draw[-stealth] (m-3-3.east) -- (m-3-4.west);
		\draw[-stealth] (m-3-4.east) -- (m-3-5.west);
		\draw[-stealth] (m-3-5.east) -- (m-3-6.west |- m-3-5.east);
		
		\draw[-stealth] (m-2-2.east) -- (m-2-3.west);
		\draw[-stealth] (m-2-3.east) -- (m-2-4.west);
		\draw[-stealth] (m-2-4.east) -- (m-2-5.west);
		\draw[-stealth] (m-2-5.east) -- (m-2-6.west |- m-2-5.east);
		
\draw[thick] (m-1-1.east) -- (m-6-1.east) ;
\draw[thick] (m-6-1.north) -- (m-6-6.north) ;

\node [above=5pt, align=flush left,text width=1cm] at (m-1-1.north west)
        {
            $E_1$
        };
\end{tikzpicture}
\end{equation*}

If the original filtration was bounded, then we get 
\[
E^{pq}_2 = H_v^pH^q_h(C^{\bullet \bullet}) \Rightarrow H^{p+q}(\Tot C^{\bullet \bullet}),
\]
\end{exa}
where cohomology is computed using the horizontal differentials first, and then using the vertical differentials.

We can apply this example to the double complex $F(I^{\bullet \bullet})$ coming from the Cartan-Eilenberg resolution $I^{\bullet \bullet}$ that is used to compute $\R F^i\cE^\bullet$.

This yields an important result we will be using in the next section (see Section 5.7.9 in \cite{We1994}):
\begin{thm}
\label{rightsequence}
If the corresponding filtrations defined on the double complexes are bounded, then there are two convergent spectral sequences
\begin{align*}
& H^pR^qF(\cE^{\bullet}) \Rightarrow \R^{p+q}F(\cE^\bullet)\\
& R^pFH^q(\cE^{\bullet}) \Rightarrow \R^{p+q}F(\cE^\bullet).
\end{align*}
\end{thm}

\subsection{The Beilinson spectral sequence}
The next two examples of moduli spaces arise from applying a result of Beilinson (\cite{Be1978}).  We would like to give a proof of the following theorem:

\begin{thm}[Beilinson]
\label{beilinson}
For any coherent sheaf $\cF$ on $\P^n$ there exist two convergent spectral sequences:
\begin{align*}
& E^{pq}_1 = H^q(\P^n, \cF(p))\otimes \Omega^{-p}(-p) \Rightarrow H^{p+q} = \begin{cases}
      \cF & \text{for } p+q = 0 \\
      0 & \text{for } p+q \neq 0
    \end{cases}\\
& E^{pq}_1 = H^q(\P^n, \cF \otimes \Omega^{-p}(-p))\otimes \cO_{\P^n}(p) \Rightarrow H^{p+q} = \begin{cases}
      \cF & \text{for } p+q = 0 \\
      0 & \text{for } p+q \neq 0
    \end{cases}
\end{align*}
where $\Omega = \Omega^1_{\P^n}$ is the sheaf of differential $1$-forms on $\P^n$.  Note that in the above spectral sequences $p \le 0$, and the corresponding entries for $p > 0$ are equal to $0$.		
\end{thm}

A key part of proving the existence of the Beilinson spectral sequence involves constructing a resolution of the structure sheaf $\cO_{\Delta}$ of the diagonal $\Delta \hookrightarrow \P^n \times \P^n$.  To do this, we will require the following lemma:
\begin{lmm}
\label{koszul}
Let $\cE$ be a locally free sheaf of rank $r$ on a smooth, algebraic variety $X$.  Let $s \in H^0(X, \cE)$ be a global section such that the zero locus $Z := Z(s)$ of $s$ is a codimension $r$ subvariety in $X$.  The structure sheaf $\cO_Z$ has a locally free resolution defined by the Koszul complex:
\[
0 \rightarrow \wedge^r \cE^{\vee} \rightarrow \cdots \rightarrow \wedge^2 \cE^{\vee} \rightarrow \cE^{\vee} \rightarrow \cO_X \rightarrow \cO_Z \rightarrow 0,
\] 
where the connecting maps are given by contraction with $s$.
\end{lmm} 
For a proof of this lemma see Section IV.2 in \cite{FuLa1985}.  It relies on the fact that locally $s$ defines a regular sequence.  We can now proceed with the construction of the Beilinson spectral sequence.

\begin{proof}[Proof of Theorem \ref{beilinson}]
Let $p_1, p_2: \P^n \times \P^n \rightarrow \P^n$ be the two projections onto $\P^n$.  If $\cF, \cG$ are coherent sheaves over $\P^n$, denote by $\cF \boxtimes \cG := p_1^*\cF \otimes p_2^* \cG$ the exterior tensor product.  Consider the following version of the Euler short exact sequence on $\P^n$:
\[
0 \rightarrow \cO(-1) \rightarrow \cO^{n+1} \rightarrow \cQ \rightarrow 0, 
\]
where $\cQ = \cT(-1)$ is the tangent sheaf twisted by $\cO(-1)$.  Let $\cE := \cO(1) \boxtimes \cQ \cong \HOM(p_1^*\cO(-1),p_2^*\cQ)$.  The Euler sequence induces the long exact sequence:
\[
0 \rightarrow H^0(\P^n, \cO(-1)) \rightarrow H^0(\P^n, \cO^{n+1}) \rightarrow H^0(\P^n, \cQ) \rightarrow H^1(\P^n, \cO(-1)) \rightarrow \cdots .
\] 
Since $H^0(\P^n, \cO(-1)) = H^1(\P^n, \cO(-1)) =  0$ and $H^0(\P^n, \cO^{n+1}) = \C^{n+1}$, we get $H^0(\P^n, \cQ) \cong \C^{n+1}$.  Furthermore, we can compute
\begin{align*}
& H^0(\P^n \times \P^n, \cE) = H^0(\P^n \times \P^n, p_1^*\cO(1) \otimes p_2^* \cQ) \cong H^0(\P^n, \cO(1)) \otimes H^0(\P^n, \cQ)\\
& \cong (\C^{\vee})^{n+1} \otimes \C^{n+1} \cong \End(\C^{n+1}). 
\end{align*}
Let $s \in H^0(\P^n \times \P^n, \cE)$ be the section corresponding to the identity endomorphism in the above identification.  We wish to verify that the diagonal $\Delta \hookrightarrow \P^n \times \P^n$ is the zero locus of $s$.  Indeed, let $v, w \in \C^{n+1} - \{0\}$ and let $x_v,y_w \in \P^n$ correspond to the lines through the origin and $v, w$, respectively.  The fiber over $(x_v,y_w)$ of $\cE$ (as a vector bundle) is equal to
\begin{align*}
& \cE_{(x_v, y_w)} = (p_1^*\cO(1) \otimes p_2^* \cQ)_{(x_v, y_w)} = p_1^*\cO(1)_{(x_v, y_w)} \otimes p_2^* \cQ_{(x_v, y_w)}\\
& = \cO(1)_{x_v} \otimes \cQ_{y_w} \cong \Hom(\cO(-1)_{x_v}, \cQ_{y_w}).
\end{align*}
Therefore, after identifying $\cO(-1)_{x_v}$ with the line $\C v$ and $\cQ_{y_w}$ with the quotient $\C^{n+1}/\C w$ (via the Euler sequence), we may interpret $s(x_v, y_w) \in \cE_{(x_v, y_w)}$ as a morphism of vector spaces defined by:
\[
s(x_v, y_w)(a v) = [av] \in \C^{n+1}/\C w \textrm{ for any } a \in \C,
\]
where $[av]$ is equivalence class of $av$ under the quotient map.  This means that $s(x_v, y_w)(a v) = 0$ if and only if $x_v = y_w$.  From this is not hard to see that the zero locus $Z(s) = \Delta$ as a subvariety of $\P^n \times \P^n$.  Noting that $\cE$ is locally free of rank $n$, we get by Lemma \ref{koszul} a locally free resolution of $\cO_{\Delta}$:
\[
0 \rightarrow \wedge^n \cE^{\vee} \rightarrow \cdots \rightarrow \wedge^2 \cE^{\vee} \rightarrow \cE^{\vee} \rightarrow \cO_{\P^n \times \P^n} \rightarrow \cO_{\Delta} \rightarrow 0.
\] 
We can rewrite this as
\[
0 \rightarrow p_1^*\cO(-n) \otimes \wedge^n p_2^* \cQ^{\vee} \rightarrow \cdots \rightarrow p_1^*\cO(-1)\otimes p_2^*\cQ^{\vee} \rightarrow \cO_{\P^n \times \P^n} \rightarrow \cO_{\Delta} \rightarrow 0.
\]
Using $\cQ^{\vee} = \Omega^1(1)$, we get 
\[
0 \rightarrow \cO(-n) \boxtimes \Omega^{n}(n) \rightarrow \cdots \rightarrow \cO(-1) \boxtimes \Omega^1(1) \rightarrow \cO_{\P^n \times \P^n} \rightarrow \cO_{\Delta} \rightarrow 0.
\]
Tensoring by $p_1^*\cF$, we obtain the following sequence of sheaves:
\[
0 \rightarrow \cF(-n) \boxtimes \Omega^{n}(n) \rightarrow \cdots \rightarrow \cF(-1) \boxtimes \Omega^1(1) \rightarrow \cF \boxtimes \cO \rightarrow p_1^*\cF \otimes \cO_{\Delta} \rightarrow 0.
\]
Note that this sequence is exact because $\textrm{Tor}_i(\cO_{\Delta}, p_1^*\cF) = 0$ for all $i > 0$.  Let $C^{-i} = \cF(-i)\otimes \Omega^{i}(i)$ for $i = 0, 1, \dots, n$, and denote by $C^\bullet$ the complex
\[
0 \rightarrow C^{-n} \rightarrow \cdots \rightarrow C^0 \rightarrow 0.
\]
We wish to compute $\R^ip_{2*}(C^{\bullet})$.  Using the two spectral sequences of Theorem \ref{rightsequence} associated to the corresponding injective Cartan-Eilenberg resolution, we obtain
\begin{align*}
& \prescript{1}{}E^{pq}_2 = H^pR^qp_{2*}(C^{\bullet}) \Rightarrow \R^{p+q}p_{2*}(C^{\bullet})\\
& \prescript{2}{}E^{pq}_2 = R^p p_{2*}H^q(C^{\bullet}) \Rightarrow \R^{p+q}p_{2*}(C^{\bullet}).
\end{align*}
Note that the complex $C^{\bullet}$ is exact in all but one term, so we have
\begin{align*}
& H^q(C^{\bullet})  = \begin{cases} p_1^* \cF \otimes \cO_{\Delta} & q = 0\\ 0 & q \neq 0 \end{cases}\\
& \textrm{ which implies }\\
& R^p p_{2*}H^q(C^{\bullet})  = \begin{cases} R^p p_{2*}(p_1^* \cF \otimes \cO_{\Delta}) & q = 0\\ 0 & q \neq 0. \end{cases}
\end{align*}
 Note that $\cO_{\Delta} = \Delta_*\cO$, where $\Delta: \P^n \hookrightarrow \P^n \times \P^n$ denotes the diagonal embedding by abuse of notation.  Therefore, by the projection formula, we get
\begin{align*}
& R^p p_{2*}(p_1^* \cF \otimes \Delta_*\cO) \cong R^p p_{2*}(\Delta_*(\Delta^* p_1^* \cF \otimes \cO)) \cong R^p p_{2*}(\Delta_*((p_1\Delta)^* \cF \otimes \cO))\\ & = (R^p p_{2*})\Delta_*(\cF) = \begin{cases} \cF & p = 0\\ 0 & p \neq 0. \end{cases}
\end{align*}
The fact that $(R^p p_{2*})\Delta_*(\cF) = 0$ for $p > 0$ can be shown by noting $(R^p p_{2*})\Delta_*(\cF)$ is the sheaf associated to the presheaf $U \mapsto H^p(U \times \P^n,\Delta_*(\cF)|_{U \times \P^n}) \cong H^p(U ,\cF|_{U})$, which vanishes for affine $U$ if $p > 0$.  Convergence of the spectral sequence $\prescript{1}{}E^{pq}$ implies that 
\[
\R^{p+q}p_{2*}(C^{\bullet}) = \begin{cases} \cF & p + q = 0\\ 0 & p+q \neq 0. \end{cases}
\]
The remaining spectral sequence gives us
\begin{align*}
& \prescript{1}{}E^{pq}_1 = R^q p_{2*}(C^p) = R^q p_{2*}(\cF(p) \boxtimes \Omega^{-p}(-p)) = R^q p_{2*}(p_1^* \cF(p) \otimes p_2^* \Omega^{-p}(-p))\\ & = R^q p_{2*}(\cF(p))\otimes \Omega^{-p}(-p) = H^q(\P^n, \cF(p))\otimes \Omega^{-p}(-p),
\end{align*}
where the last two equalities follow from the projection formula and the base change theorem, respectively.  Putting this together with the $\prescript{2}{}E^{pq}$ computation, we obtain
\[
\prescript{1}{}E^{pq}_1 = H^q(\P^n, \cF(p))\otimes \Omega^{-p}(-p) \Rightarrow \R^{p+q}p_{2*}(C^{\bullet}) = \begin{cases}
      \cF & \text{for } p+q = 0 \\
      0 & \text{for } p+q \neq 0,
    \end{cases}
\]
which proves the first part of the theorem.  The spectral sequence in the second part is obtained analogously after exchanging the roles of $p_1$ and $p_2$.
\end{proof}
\begin{rmk}
Note that the proof of the theorem implies somewhat more than the statement.  Namely, both spectral sequences are situated at the points of the $n \times n$ finite lattice in the second quadrant defined by $-n \le p \le 0$ and $0 \le q \le n$.
\end{rmk}
\section{Example: Vector Bundles on $\P^1$}
\subsection{Spectral sequence computation}

Our first application of Beilinson's spectral sequence is in the context of vector bundles on the projective line.  Applying the first spectral sequence of Theorem \ref{beilinson} to a vector bundle $E$ on $\P^1$ yields the following:

\begin{equation*}
\begin{tikzpicture}
	\matrix (m) [matrix of math nodes,
    nodes in empty cells,nodes={minimum width=4ex,
    minimum height=4ex,outer sep=-2pt},
    column sep=1ex,row sep=1ex]{
               &   &    & q\\  
               &  H^1(\P^1, E(-1)) \otimes \Omega^1(1)  & H^1(\P^1, E) \otimes \cO & 1\\
               &  H^0(\P^1, E(-1)) \otimes \Omega^1(1)  & H^0(\P^1, E) \otimes \cO & 0\\
    \quad\strut p & -1  & 0 & \strut \\};
 
		\draw[-stealth] (m-2-2.east) -- (m-2-3.west);
		\draw[-stealth] (m-3-2.east) -- (m-3-3.west);

\draw[thick] (m-1-4.west) -- (m-4-4.west) ;
\draw[thick] (m-4-1.north) -- (m-4-4.north) ;

\node [above=3pt, align=flush left,text width=1cm] at (m-1-1.north west)
        {
            $E_1$
        };
\end{tikzpicture}
\end{equation*}

Note that by a theorem of Grothendieck (see e.g. Section 10.5 in \cite{Kem1993}) any vector bundle $E$ on $\P^1$ is isomorphic to the direct sum of line bundles $\cO(d_1) \oplus \cdots \oplus \cO(d_r)$ where $d_1, \dots, d_r \in \Z$.  Therefore, if $E$ is generated by global sections then $d_i \ge 0$ for all $1 \le i \le r$.  Consequently, $H^0(\P^1, E^{\vee}) = 0$.  Since $\Omega_{\P^1}^1 = \cO(-2)$, by Serre duality we have 
\begin{align*}
& H^1(\P^1, E) \cong H^0(\P^1, E^{\vee}(-2))^{\vee}\\
& H^1(\P^1, E(-1)) \cong H^0(\P^1, E^{\vee}(-1))^{\vee}.
\end{align*}
Thus, for vector bundles $E$ generated by global sections, the spectral sequence collapses at $E_2$, making $\ker d = 0$ and giving us the following description of $E$:
\[
E \cong \coker(H^0(\P^1, E(-1))\otimes \cO(-1) \xrightarrow{d} H^0(\P^1, E) \otimes \cO)
\]

\subsection{Moduli space of vector bundles}

Now consider the classification problem where:
\[
 \cA = \left\{ 
     \begin{tabular}{lllll}
       vector bundle $E$ on $\P^1$ generated by global sections\\
       $\deg E = d$ and $\rk E = r$\\
			two bases of global sections\\
       $s_1, \dots , s_n \in H^0(\P^1, E)$\\
			$t_1, \dots , t_m \in H^0(\P^1, E(-1))$
     \end{tabular}
   \right\}
\]
and $\sim$ is given by vector bundle isomorphisms compatible with the chosen bases.  Note that $m, n, r, d$ are not independent, since by Grothendieck's theorem we have that $n = d + r$ and $m = d$ (this means $m$ can be equal to $0$, in which case the basis $\{t_1, \dots, t_m\}$ is empty).  This can be extended to the moduli functor $\cM: \opcat{Sch} \rightarrow \Set$ as follows:
\[
\cM(T) = \{(E,s,t)\} / \textrm{iso. compatible with $s$ and $t$},
\]
where $E$ is a rank $r$ vector bundle on  $T \times \P^1 \xrightarrow{p} T$, $\deg E|_{\{x\}\times \P^1} = d$ for all $x \in T$, and $s: p_*(E) \rightarrow \cO_T^n$, $t: p_*(E(-1)) \rightarrow \cO_T^m$ are isomorphisms (the direct images are vector bundles by Grauert's theorem).  As usual, $\cM$ sends each morphism to the pullback along that morphism.

\begin{thm}
\label{p1moduli}
The functor $\cM$ is represented by 
\[ 
X = \{(b_0,b_1)| \lambda_0b_0 + \lambda_1b_1 \textrm{ injective for all } [\lambda_0:\lambda_1] \in \P^1\},
\]
where $b_0,b_1 \in \Hom(\C^m, \C^n)$.
\end{thm}

One may interpret points of $X$ as representations of the Kronecker quiver $K_2$ in standard coordinate spaces (i.e. the vector spaces assigned to the vertices are the coordinate spaces $\C^n$ and $\C^m$) subject to the given relations.  We will not fully prove this theorem, instead only establishing a bijection between the set $\cA/ \sim$ and the set of points of $X$ (similar to the proof of the equivalence of categories in Lemma 5.5 of \cite{CB2004}).  For a complete proof of the representability of the functor see \cite{ASo2014}).

\begin{proof}[Proof of Theorem \ref{p1moduli}]
Let $E$ be a vector bundle on $\P^1$ generated by global sections on $\P^1$.  The above computation using the Beilinson spectral sequence gives us that 
\[
E \cong \coker(H^0(\P^1, E(-1))\otimes \cO(-1) \xrightarrow{d} H^0(\P^1, E) \otimes \cO).
\]
Let $V =  H^0(\P^1, E)$ and $W = H^0(\P^1, E(-1))$. Since $d \in \Hom(W\otimes \cO(-1), V\otimes \cO) \cong \Hom(W, V) \otimes H^0(\P^1, \cO(1))$, we can write
$d = b_0x_0 + b_1x_1$, where $b_0,b_1 \in \Hom(W,V)$ and $x_0,x_1 \in H^0(\P^1, \cO(1))$ is the usual basis of global sections.  The trivializations of $V$ and $W$ identify $b_0, b_1$ with matrices in $\Hom(\C^m, \C^n)$.  Since $d$ is injective and the cokernel is a vector bundle, then it is an injective morphism of vector bundles, making it injective on the fibers.  That is, if $\lambda = [\lambda_0: \lambda_1] \in \P^1$, then $d$ restricted to the fiber over $\lambda$ gives us the injective linear transformation $d_{\lambda}  = \lambda_0b_0 + \lambda_1b_1$.

Conversely, given $(b_0, b_1) \in X$ we have the morphism $d: \C^m \otimes \cO(-1) \rightarrow \C^n \otimes \cO$ given by $d = b_0x_0 + b_1x_1$.  The condition that $\lambda_0b_0 + \lambda_1b_1$ is injective for all $[\lambda_0: \lambda_1] \in \P^1$ means that $d$ is an injective morphism of vector bundles.  Thus, $E := \coker(d)$ is a vector bundle.  This gives us the short exact sequence
\[
0 \rightarrow \C^m\otimes \cO(-1) \rightarrow \C^n \otimes \cO \rightarrow E \rightarrow 0.
\]  
The corresponding long exact sequence for cohomology results in:
\begin{align*}
& 0 \rightarrow H^0(\P^1, \C^m\otimes \cO(-1)) \rightarrow H^0(\P^1, \C^n\otimes \cO)\\ 
&\rightarrow H^0(\P^1, E) \rightarrow H^1(\P^1, \C^m\otimes \cO(-1)) \rightarrow \cdots.
\end{align*}
 Since $H^0(\P^1, \C^m\otimes \cO(-1)) = H^1(\P^1, \C^m\otimes \cO(-1)) = 0$, then the isomorphism $H^0(\P^1, \C^n\otimes \cO) \cong H^0(\P^1, E)$ determines a basis of global sections of $E$.  A similar computation for the short exact sequence twisted by $\cO(-1)$:
\[
0 \rightarrow \C^m\otimes \cO(-2) \rightarrow \C^n \otimes \cO(-1) \rightarrow E(-1) \rightarrow 0,
\]
yields an isomorphism $H^0(\P^1, E(-1)) \cong H^1(\P^1, \C^m\otimes\cO(-2))$.  Noting that by Serre duality $ H^1(\P^1, \C^m\otimes\cO(-2)) \cong H^0(\P^1, (\C^m)^{\vee}\otimes \cO)^{\vee}$, we get a basis for $H^0(\P^1, E(-1))$.  It is easy to check the two constructions are mutually inverse.  Therefore, we obtain a bijection between $\cA/\sim$ and the points of $X$.
\end{proof}

\begin{rmk}
If we start with the assumption that $E^{\vee}$ is generated by global sections, apply the Beilinson spectral sequence, and take the dual of the resulting complex of sheaves, we get that
\[
E \cong \ker(H^0(\P^1, E^{\vee})^{\vee}\otimes \cO \rightarrow H^0(\P^1, E^{\vee}(-1))^{\vee})\otimes \cO(1).
\] 
This leads to an alternative moduli space 
\[ 
X' = \{(b_0,b_1)| \lambda_0b_0 + \lambda_1b_1 \textrm{ surjective for all } [\lambda_0:\lambda_1] \in \P^1\},
\]
parametrizing vector bundles $E$ on $\P^1$ generated by global sections, together with bases chosen for $H^0(\P^1, E^{\vee})^{\vee}$ and $H^0(\P^1, E^{\vee}(-1))^{\vee}$.
\end{rmk}

\section{Example: Torsion-free sheaves on $\P^2$}
\label{sheafp2}
\subsection{Preliminaries}

The second application of Beilinson's spectral sequence allows us to describe the moduli space of torsion-free sheaves on $\P^2$ framed at infinity in terms of quiver representations.  

More concretely, let $\l = \{[0:\lambda_1: \lambda_2] \in \P^2 \} \subset \P^2$.  Let $\cF$ be a torsion-free sheaf on $\P^2$ such that for some positive integer $r$ we have that $\cF|_{\l} \cong \cO_{\l}^r$.  We say $\cF$ is \emph{trivial at infinity}.  Consider the following lemma:
\begin{lmm}
\label{cohvanish}
Let $\cF$ be a torsion-free sheaf on $\P^2$ trivial at infinity.  We have:
\begin{align*}
& H^q(\P^2, \cF(-p)) = 0 \textrm{ for } p = 1,2 \textrm{ and } q = 0,2,\\
& H^q(\P^2, \cF(-1)\otimes \Omega^1(1)) = 0 \textrm{ for } q = 0,2.
\end{align*}
\end{lmm}
\begin{proof}
We have the following short exact sequence of coherent sheaves:
\[
0 \rightarrow \cO_{\P^2}(-1) \xrightarrow{x_0} \cO_{\P^2} \rightarrow \cO_{\l} \rightarrow 0, 
\]
where $\cO_{\P^2}(-1) \xrightarrow{x_0} \cO_{\P^2}$ is defined by multiplication by the global section $x_0 \in H^0(\P^2, \cO_{\P^2}(1))$.  Let $k \in \Z$.  Since $\cF$ is trivial at infinity, then $\textrm{Tor}_1(\cO_{\l}, \cF(-k)) = 0$.  Consequently, tensoring the above sequence by $\cF(-k)$ gives us the short exact sequence:
\[
0 \rightarrow \cF(-k-1) \rightarrow \cF(-k) \rightarrow \cF(-k)|_{\l} \rightarrow 0.
\]
The corresponding long exact sequence is:
\begin{align*}
0 & \rightarrow  H^0(\P^2, \cF(-k-1)) \rightarrow  H^0(\P^2, \cF(-k)) \rightarrow  H^0(\l, \cF(-k)|_{\l})\\ &\rightarrow  H^1(\P^2, \cF(-k-1)) \rightarrow  H^1(\P^2, \cF(-k)) \rightarrow  H^1(\l, \cF(-k)|_{\l})\\ &\rightarrow  H^2(\P^2, \cF(-k-1)) \rightarrow  H^2(\P^2, \cF(-k)) \rightarrow  0.
\end{align*}
Since $\cF$ is trivial at infinity, it follows that
\[
\begin{cases}
H^0(\l, \cF(-k)|_{\l}) \cong H^0(\l, \cO_{\l}^r(-k)) = 0 & \textrm{for } k \ge 1\\
H^1(\l, \cF(-k)|_{\l}) \cong H^1(\l, \cO_{\l}^r(-k)) = 0 & \textrm{for } k \le 1,
\end{cases}
\]
so the long exact sequence gives us
\[
\begin{cases}
H^0(\P^2, \cF(-k-1)) \cong H^0(\P^2, \cF(-k)) & \textrm{for } k \ge 1\\
H^2(\P^2, \cF(-k-1)) \cong H^2(\P^2, \cF(-k)) & \textrm{for } k \le 1.
\end{cases}
\]
Since $\cF$ is torsion-free, then the double dual $\cF^{\vee \vee}$ is locally free (see e.g. Section 30.12 of \cite{Sta2019}), and there is a natural inclusion $\cF \hookrightarrow \cF^{\vee \vee}$.  Noting that $\Omega^2_{\P^2} = \cO(-3)$, by Serre duality we have $H^0(\P^2, \cF(-k)) \hookrightarrow H^0(\P^2, \cF^{\vee \vee}(-k)) \cong H^2(\P^2, (\cF^{\vee \vee})^{\vee}(k-3))^{\vee}$.  Now, let $k \in \Z$ be large enough so that by the Serre vanishing theorem $H^2(\P^2, (\cF^{\vee \vee})^{\vee}(k-3)) = H^2(\P^2, \cF(k)) = 0$.  Using the isomorphisms above gives us the following:
\begin{align*}
& H^0(\P^2, \cF(-1)) \cong H^0(\P^2, \cF(-2)) \cong \cdots  \cong H^0(\P^2, \cF(-k)) = 0\\
& H^2(\P^2, \cF(-2)) \cong H^2(\P^2, \cF(-1)) \cong \cdots  \cong H^2(\P^2, \cF(k)) = 0.
\end{align*}
This proves the first part of the lemma.  The proof of the second part is analogous, using the short exact sequence
\[
0 \rightarrow \cF(-k-1) \otimes \Omega^1(1) \rightarrow \cF(-k) \otimes \Omega^1(1) \rightarrow \cF(-k) \otimes \Omega^1(1)|_{\l} \rightarrow 0,
\]
and the relation $\Omega^1(1)|_{\l} \cong \cO_{\l} \oplus \cO_{\l}(-1)$.
\end{proof}

Applying the second spectral sequence of Theorem \ref{beilinson} to $\cF(-1)$ we obtain the following:

\begin{equation*}
\begin{tikzpicture}[scale=0.8, every node/.style={scale=0.8}]
	\matrix (m) [matrix of math nodes,
    nodes in empty cells,nodes={minimum width=5ex,
    minimum height=5ex,outer sep=-1pt},
    column sep=1ex,row sep=1ex]{
               &  &  &    & q\\  
							 & H^2(\P^2, \cF(-1) \otimes \Omega^2(2)) \otimes \cO(-2) &  H^2(\P^2, \cF(-1) \otimes \Omega^1(1)) \otimes \cO(-1)  & H^2(\P^2, \cF(-1)) \otimes \cO & 2\\
               & H^1(\P^2, \cF(-1) \otimes \Omega^2(2)) \otimes \cO(-2) &  H^1(\P^2, \cF(-1) \otimes \Omega^1(1)) \otimes \cO(-1)  & H^1(\P^2, \cF(-1)) \otimes \cO & 1\\
               & H^0(\P^2, \cF(-1) \otimes \Omega^2(2)) \otimes \cO(-2) & H^0(\P^2, \cF(-1) \otimes \Omega^1(1)) \otimes \cO(-1)  & H^0(\P^2, \cF(-1)) \otimes \cO & 0\\
    \quad\strut p & -2 & -1  & 0 & \strut \\};
 
		\draw[-stealth] (m-2-2.east) -- (m-2-3.west);
		\draw[-stealth] (m-3-2.east) -- (m-3-3.west);
		\draw[-stealth] (m-4-2.east) -- (m-4-3.west);
		
		\draw[-stealth] (m-2-3.east) -- (m-2-4.west);
		\draw[-stealth] (m-3-3.east) -- (m-3-4.west);
		\draw[-stealth] (m-4-3.east) -- (m-4-4.west);
		
\draw[thick] (m-1-5.west) -- (m-5-5.west) ;
\draw[thick] (m-5-1.north) -- (m-5-5.north) ;

\node [above=1pt, align=flush left,text width=2pt] at (m-1-1.north west)
        {
            $E_1$
        };
\end{tikzpicture}
\end{equation*}

Applying Lemma \ref{cohvanish} results in all of the entries in the first and third rows being equal to $0$.  This leaves the middle row, which can be written as:
\[
H^1(\P^2, \cF(-1)) \otimes \cO(-2) \xrightarrow{a'} H^1(\P^2, \cF(-1)\otimes \Omega^1(1)) \otimes \cO(-1) \xrightarrow{b'} H^1(\P^2, \cF(-1)) \otimes \cO.
\]
The convergence condition implies $\ker(a') = \coker(b') = 0$ and that the cohomology of the middle term is $\cF(-1)$.  Twisting this complex by $\cO(1)$ we obtain the complex 
\[
H^1(\P^2, \cF(-1)) \otimes \cO(-1) \xrightarrow{a} H^1(\P^2, \cF(-1)\otimes \Omega^1(1)) \otimes \cO \xrightarrow{b} H^1(\P^2, \cF(-1)) \otimes \cO(1).
\]
As before, we get that $\ker(a) = \coker(b) = 0$ and $\cF \cong \ker(b)/\Im(a)$.  This kind of three term complex where the first morphism is injective and the second is surjective is called a \emph{monad}.  We compute the dimensions of the vector spaces involved in constructing the above monad in following lemma we will make use of later.
\begin{lmm}
\label{cohdim}
Given a torsion-free sheaf $\cF$ on $\P^2$ that is trivial at infinity, we have:
\begin{align*}
&\dim H^1(\P^2, \cF(-2)) = \dim H^1(\P^2, \cF(-1)) = c_2(\cF) = n\\
&\dim H^1(\P^2, \cF(-1) \otimes \Omega^1(1)) = 2c_2(\cF) + \rk(\cF) = 2r + n.
\end{align*}
\end{lmm} 
\begin{proof}
Recall if $E$ is a rank $r$ locally free sheaf over a smooth projective scheme $X$, then the splitting principle lets us write the total Chern class of $E$ as $\textrm{c}(E) = \prod_{i=1}^r(1 + \alpha_i)$, where $\alpha_i$ are the Chern roots of $E$.  Furthermore, the Chern character can be computed as $\textrm{ch}(E) = \sum_i e^{\alpha_i}$ and Todd class is $\textrm{td}(E) = \prod_i \frac{\alpha_i}{1 - e^{-\alpha_i}}$.  Together these formulas imply:
\begin{align*}
& \textrm{ch}(E) = \rk(E) + c_1(E) + \frac{1}{2}(c_1(E)^2-2c_2(E)) + \cdots \\
& \textrm{td}(E) = 1 + \frac{1}{2}c_1(E) + \frac{1}{12}(c_1(E)^2 + c_2(E)) + \cdots
\end{align*}
If $\cL$ is a line bundle on $X$, then a further computation using Chern roots gives us that
\[
c(E \otimes \cL) = 1 + (c_1(E) + r c_1(\cL)) + \left(c_2(E) + (r-1)c_1(E)c_1(\cL) + \binom{r}{2}c_1(\cL)^2\right) + \cdots.
\]
The integral cohomology ring for the projective plane is $H^*(\P^2, \Z) = \Z[H]/H^3$, where $H$ is the class corresponding to the hyperplane $\{(0,x_1,x_2)| x_1, x_2 \in \C\} \subset \C^3$.  Therefore, the Euler sequence $0 \rightarrow \cO_{\P^2} \rightarrow \cO_{\P^2}(1)^{3} \rightarrow \cT_{\P^2} \rightarrow 0$ gives us that: 
\[
c(\cT_{\P^2}) = \frac{c(\cO_{\P^2}(1)^{3})}{c(\cO_{\P^2})} = (1+H)^3  = 1 + 3H + 3H^2.
\]  
Consequently, we have $\textrm{td}(\cT_{\P^2}) = 1 + \frac{3}{2}H + H^2$.  Consider the following short exact sequence of coherent sheaves:
\[
0 \rightarrow H^1(\P^2, \cF(-1)) \otimes \cO(-1) \rightarrow \ker(b) \rightarrow \cF \rightarrow 0,
\]
where $b$ is as in the definition of a monad.  Note this is a resolution of $\cF$ by locally free sheaves. We have that the Euler characteristic for the locally free sheaf $\cO(-1)$ is $\chi(\cO(-1)) = 0$.  This means $\chi(\cF) = \chi(\ker(b))$, and we may assume for the rest of the proof that $\cF$ is locally free of rank $r$.

By the Hirzebruch-Riemann-Roch theorem we have
\[
\chi(\cF(-2)) = \int_{\P^2} \textrm{ch}(\cF(-2)) \textrm{td}(\cT_{\P^2}),
\]
where the integral implies we are taking the component of the product belonging to the top cohomology group.  Using the formula above, $c(\cF(-2)) = 1 - 2r H + (n + 4 \binom{r}{2})H^2$.  Therefore, we have $\textrm{ch}(\cF(-2)) = r - 2r H + (2r - n)H^2$.  By Lemma \ref{cohvanish}, we get
\begin{align*}
& \chi(\cF(-2)) = - \dim H^1(\P^2, \cF(-2)) = \int_{\P^2} \textrm{ch}(\cF(-2)) \cdot \textrm{td}(\cT_{\P^2})\\ & = r - 3r + 2r - n = - n.
\end{align*} 
We can use analogous computations to prove the remaining two formulas.
\end{proof}

\subsection{Moduli space of torsion-free sheaves}
As in the case of $\P^1$ we consider the related classification problem described as follows:
\[
 \cA = \left\{ 
     \begin{tabular}{lll}
       $\cF$ is a torsion-free coherent sheaf on $\P^1$ trivial at infinity\\
       $\rk \cF = r$ and $c_2(\cF) = n$\\
			$\Phi: \cF|_{\l} \rightarrow \cO_{\l}^r$ is a trivialization on $\l$
     \end{tabular}
   \right\}
\]
and $\sim$ is given by sheaf isomorphisms that are compatible with the trivialization.  Note that the inclusion $\P^1 \hookrightarrow \P^2$ induces the natural quotient homomorphism $\Z[H]/H^3 \rightarrow \Z[H]/H^2$ on cohomology.  By the functoriality of Chern classes with respect to pullback and $c_1(\cF|_{\l}) = c_1(\cO_{\P^1}^r)= 0$ it follows that $c_1(\cF) = 0$.  The corresponding moduli functor may be written as:
\[
\cM(T) = \{(\cF,\Phi)\} / \textrm{iso. compatible with $\Phi$},
\]
where $\cF$ is a rank $r$ torsion-free coherent sheaf on  $T \times \P^2$, $c_2(\cF|_{\{x\}\times \P^2}) = n$ for all $x \in T$, and $\Phi: \cF|_{T \times \l} \rightarrow \cO_{T \times \l}^r$ is an isomorphism. 

\begin{thm}
\label{p2moduli}
The functor $\cM$ is represented by 
\[ 
X = \{(B_1,B_2,i,j) \}/\GL(n, \C),
\]
where $B_1,B_2 \in \Hom(\C^n, \C^n)$, $i \in \Hom(\C^r, \C^n)$, and $j \in \Hom(\C^n, \C^r)$ subject to the conditions:
\begin{enumerate}
\item $[B_1,B_2] + ij = 0$
\item there is no proper subspace $U \subset \C^n$ such that $B_1(U) \subset U$, $B_2(U) \subset U$, and $\Im(i) \subset U$,
\end{enumerate}
and the action of $\GL(n, \C)$ is given by change of basis for $\C^n$:
\[
g(B_1, B_2, i, j) = (gB_1\inv{g}, gB_2\inv{g}, gi, j\inv{g}).
\]
\end{thm}
We can visualize the points of $X$ as quiver representations up to $\GL(n, \C)$ action:

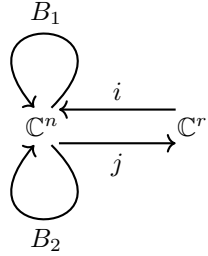
\begin{equation*}
\begin{tikzpicture}[scale=1.0, every node/.style={scale=1.0}]

\node[] at (3,1) {$\C^n$};
\node [above] (adhm11) at (3.1,1.1) {};
\node [below] (adhm21) at (3.1,0.9) {};

\node [above] (adhm1) at (3,1) {};
\node [below] (adhm2) at (3,1) {};

\node[] at (5,1) {$\C^r$};
\node [above] (adhm12) at (4.9,1.1) {};
\node [below] (adhm22) at (4.9,0.9) {};

\draw [thick, <-] (adhm1) to [out=130,in=45,looseness=20] node[above] {$B_1$}(adhm1);
\draw [thick, <-] (adhm2) to [out=230,in=315,looseness=20] node[below] {$B_2$}(adhm2);

\draw [thick, <-] (adhm11)  -- node [above] {$i$} (adhm12);
\draw [thick, ->] (adhm21) -- node [below] {$j$}(adhm22);
\end{tikzpicture}
\end{equation*}

Once again, we will not fully prove this theorem, instead only establishing a bijection between equivalence classes of torsion-free sheaves (with trivializations) and the points of $X$.  
\begin{proof}[Proof of Theorem \ref{p2moduli}]
Let $\cF$ be a torsion-free sheaf on $\P^2$ with trivialization at infinity given by $\Phi$.  Let  $V = H^1(\P^2, \cF(-2))$, $V' = H^1(\P^2, \cF(-1))$, and $\widetilde{W} = H^1(\P^2, \cF(-1) \otimes \Omega^1(1))$.  Using the monad representation of $\cF$ we have that $\cF$ is isomorphic to the middle cohomology of the following monad:
\[
V\otimes \cO(-1) \xrightarrow{a} \widetilde{W} \otimes \cO \xrightarrow{b} V' \otimes \cO(1).
\]
Since $a \in \Hom(V\otimes \cO(-1), \widetilde{W} \otimes \cO) \cong \Hom(V, \widetilde{W}) \otimes \cO(1)$ and $b \in \Hom(\widetilde{W} \otimes \cO, V' \otimes \cO(1)) \cong \Hom(\widetilde{W}, V') \otimes \cO(1)$, we can write:
\begin{align*}
& a = a_0x_0 + a_1x_1 + a_2x_2\\
& b = b_0x_0 + b_1x_1 + b_2x_2,
\end{align*}
where $a_0, a_2, a_3 \in \Hom (V, \widetilde{W})$, $b_0, b_1, b_2 \in \Hom(\widetilde{W}, V')$, and $x_0, x_1, x_2$ is the standard basis of global sections of $\cO_{\P^2}(1)$.  Thus the data from the monad may be expressed in terms of a quiver representation:

\begin{equation*}
\begin{tikzpicture}[scale=1.0, every node/.style={scale=1.0}]

\node[] at (2,1) {$V$};
\node [above] (bel-a01) at (2.1,1.4) {};
\node  (bel-a11) at (2.1,1) {};
\node [below] (bel-a21) at (2.1,0.6) {};

\node[] at (5,1) {$\widetilde{W}$};
\node [above] (bel-a02) at (4.9,1.4) {};
\node  (bel-a12) at (4.9,1) {};
\node [below] (bel-a22) at (4.9,0.6) {};

\node [above] (bel-b01) at (5.1,1.4) {};
\node  (bel-b11) at (5.1,1) {};
\node [below] (bel-b21) at (5.1,0.6) {};

\node[] at (8,1) {$V'$};
\node [above] (bel-b02) at (7.9,1.4) {};
\node  (bel-b12) at (7.9,1) {};
\node [below] (bel-b22) at (7.9,0.6) {};

\draw [thick, ->] (bel-a01)  -- node [above] {$a_0$} (bel-a02);
\draw [thick, ->] (bel-a11)  -- node [above] {$a_1$} (bel-a12);
\draw [thick, ->] (bel-a21)  -- node [above] {$a_2$} (bel-a22);

\draw [thick, ->] (bel-b01)  -- node [above] {$b_0$} (bel-b02);
\draw [thick, ->] (bel-b11)  -- node [above] {$b_1$} (bel-b12);
\draw [thick, ->] (bel-b21)  -- node [above] {$b_2$} (bel-b22);

\end{tikzpicture}
\end{equation*}

Since $ba = 0$ we have the arrows in the above representation are subject to the following relations:
\begin{align*}
& b_0a_0 = b_1a_1 = b_2a_2 = 0\\
& b_0a_1 + b_1a_0 = 0 \\
& b_1a_2 + b_2a_1 = 0\\
& b_2a_0 + b_0a_2 = 0
\end{align*}

Restricting our monad to $\l$, we obtain:
\[
V \otimes \cO_{\l}(-1) \xrightarrow{a_{\l}} \widetilde{W} \otimes \cO_{\l} \xrightarrow{b_{\l}} V' \otimes \cO_{\l}(1),
\]
where $a_{\l} := a|_{\l} = a_1x_1 + a_2x_2$ and $b_{\l} := b|_{\l} = b_1x_1 + b_2x_2$.  This gives rise to the complex
\[
0 \rightarrow V \otimes \cO_{\l}(-1) \rightarrow \ker(b_{\l}) \rightarrow \cF|_{\l} \rightarrow 0,
\]
which is a short exact sequence since $\cF|_{\l}$ is a vector bundle.  The corresponding long exact sequence on cohomology is:
\begin{align*}
& 0 \rightarrow H^0(\l, \cO_{\l}(-1)) \otimes V \rightarrow H^0(\l, \ker(b_{\l})) \rightarrow H^0(\l, \cF|_{\l})\\
& \rightarrow H^1(\l, \cO_{\l}(-1)) \otimes V \rightarrow H^1(\l, \ker(b_{\l})) \rightarrow H^1(\l, \cF|_{\l}) \rightarrow 0.
\end{align*}
Since $H^0(\l, \cO_{\l}(-1)) = H^1(\l, \cO_{\l}(-1)) = 0$, we have $H^0(\l, \ker(b_{\l})) \cong H^0(\l, \cF|_{\l})$ and $H^1(\l, \ker(b_{\l})) \cong H^1(\l, \cF|_{\l})$.  Therefore, $\cF|_{\l} \cong \cO_{\l}^r$ implies $H^1(\l, \ker(b_{\l})) = H^1(\l, \cF|_{\l}) = 0$.  The framing $\Phi$ allows us to obtain the trivialization $H^0(\l, \ker(b_{\l})) \cong H^0(\l, \cF|_{\l}) \cong \C^r$. 

We can use our monad to obtain a second short exact sequence:
\[
0 \rightarrow \ker(b_{\l}) \rightarrow \widetilde{W} \otimes \cO_{\l} \rightarrow V' \otimes \cO_{\l}(1) \rightarrow 0
\]
with the corresponding long exact sequence on cohomology:
\begin{align*}
& 0 \rightarrow H^0(\l, \ker(b_{\l})) \rightarrow H^0(\l, \cO_{\l}) \otimes \widetilde{W} \rightarrow H^0(\l, \cO_{\l}) \otimes V'\\
& \rightarrow H^1(\l, \ker(b_{\l})) \otimes V \rightarrow H^1(\l, \cO_{\l}) \otimes \widetilde{W} \rightarrow H^1(\l, \cO_{\l}) \otimes V' \rightarrow 0.
\end{align*}
From the previous computation $H^1(\l, \ker(b_{\l})) = 0$.  Moreover, since $x_1, x_2$ form a basis for the vector space $H^0(\l, \cO_{\l}(1))$, we have $H^0(\l, \cO_{\l}(1)) \otimes V' \cong (\C x_1 \oplus \C x_2) \otimes V ' \cong V' \oplus V'$.  Therefore, the long exact sequence yields the following short exact sequence:
\[
0 \rightarrow H^0(\l, \ker(b_{\l})) \rightarrow \widetilde{W} \xrightarrow{\begin{pmatrix}b_1\\ b_2 \end{pmatrix}} V' \oplus V' \rightarrow 0.
\]
Let $W := H^0(\l, \ker(b_{\l})) = \ker \begin{pmatrix}b_1\\ b_2 \end{pmatrix} = \ker(b_1) \cap \ker(b_2)$.  Now, let us consider the dual of the restriction of our monad to $\l$:
\[
(V')^{\vee} \otimes \cO_{\l}(1) \xrightarrow{\ltrans{b_{\l}}} (\widetilde{W})^{\vee} \otimes \cO_{\l} \xrightarrow{\ltrans{a_{\l}}} V^{\vee} \otimes \cO_{\l}(1).
\]
Using an analogous long exact sequence, we obtain the short exact sequence:
\[
0 \rightarrow H^0(\l, \ker(\ltrans{a_{\l}})) \rightarrow \widetilde{W} \xrightarrow{(\ltrans{a_1}, \ltrans{a_2})} V^{\vee}\oplus V^{\vee} \rightarrow 0.
\]
This means $(a_1, a_2): V \oplus V \rightarrow \widetilde{W}$ is injective, so the morphism $a_{p} = \lambda_1 a_1 + \lambda_2 a_2$ induced by $a$ on the fiber over $p = [0: \lambda_1: \lambda_2] \in \l$ is also injective.  Furthermore, $\ker{b_p}/\Im(a_p) = \cF_p$ is isomorphic to $W$ via the trivialization $\Phi$.  This isomorphism factors through the quotient map, so if $w \in \Im(a_p) \cap W$, then $w = 0$.  Consequently, for $p = [0:1:0]$, we get $\ker(a_1) \cap \Im(b_2) = \ker(a_1) \cap W = 0$.  This makes $b_1a_2 = - b_2a_1: V \rightarrow V'$ injective, so it is an isomorphism by Lemma \ref{cohdim}.

Noting that $\Im(a_1) \cap \Im(a_2) = 0$ and $\Im(a_2) \cap \ker(b_1) = 0$, we see $\widetilde{W}$ is isomorphic to $V \oplus V \oplus W$ by Lemma \ref{cohdim} (i.e. $\dim \Im(a_1) \oplus \Im(a_2) \oplus W = \dim \widetilde{W}$).  Identifying $V$ with $V'$ and $\widetilde{W}$ with $V \oplus V \oplus W$ via the isomorphisms, we see that 
\[
a_1 = \begin{pmatrix}-\Id_V\\ 0\\ 0 \end{pmatrix} \quad a_2 = \begin{pmatrix}0\\ -\Id_V\\ 0\end{pmatrix} \quad b_1 = \begin{pmatrix} 0 & -\Id_V & 0 \end{pmatrix} \quad b_2 = \begin{pmatrix} \Id_V & 0 & 0\end{pmatrix}.
\]

Similarly, we can decompose $a_0$ and $b_0$ as follows:
\[
a_0 = \begin{pmatrix} B_1\\ B_2\\ j \end{pmatrix} \quad b_0 = \begin{pmatrix} -B_2 & B_1 & i \end{pmatrix}.
\]

\begin{equation*}
\begin{tikzcd}[row sep=5pt, column sep=10pt]
& V \arrow[ddr, "-B_2"] &\\
& \oplus &\\
V \arrow[r, "B_2"] \arrow[uur, "B_1"] \arrow[ddr, "j"] & V \arrow[r, "B_1"] & V\\
& \oplus &\\
& W \arrow[uur, "i"]&
\end{tikzcd}
\end{equation*}
 
Using $ba = 0$, we get that $[B_1, B_2] + ij = 0$.  Note that $\Phi$ defines a trivialization for $W$, identifying it with $\C^r$.  Choosing a basis of $V$, we can identify it with $\C^n$.  This means $B_1,B_2 \in \Hom(\C^n, \C^n)$, $i \in \Hom(\C^r, \C^n)$, and $j \in \Hom(\C^n, \C^r)$.  For any other choice of basis of $V$, the corresponding matrices differ from $(B_1, B_2, i, j)$ by an action of $\GL(n,\C)$ as described in the theorem statement.

In order to show that $(B_1, B_2, i, j)$ defines a point in $X$ it remains to verify condition (2).  Indeed, suppose there exists a proper subspace $U \subset V$ such that $B_1(U) \subset U$, $B_2(U) \subset U$, and $\Im(i) \subset U$.  Let $U^{\bot} = \{f \in V^{\vee}|f(U)=0\}$ be the orthogonal complement of $U$.  Note that $\Im(i) \subset U$ if and only if $U^{\bot} \subset \ker(\ltrans{i})$.  Additionally, $[B_1, B_2] + ij = 0$ implies $[\ltrans{B_1}, \ltrans{B_2}] + \ltrans{j}\ltrans{i} = 0$.

Since $B_1(U) \subset U$, $B_2(U) \subset U$ we have $\ltrans{B_1}(U^{\bot}) \subset U^{\bot}$ and $\ltrans{B_2}(U^{\bot}) \subset U^{\bot}$.  Moreover, restricting to $U^{\bot}$ we get $[\ltrans{B_1}, \ltrans{B_2}]|_{U^{\bot}} = 0$.  This means $\ltrans{B_1}, \ltrans{B_2}$ have a common invariant subspace, so they share an eigenvector.  Let $0 \neq f \in U^{\bot}$ be such that 
\begin{align*}
\ltrans{B_1} f = \lambda_1 f\\
\ltrans{B_2} f =  \lambda_2 f,
\end{align*}
for some $\lambda_1, \lambda_2 \in \C$.  Let $\lambda = [0:\lambda_1: \lambda_2] \in \P^2$. If $b_{\lambda}$ is the restriction of $b$ to the fiber over $\lambda$, we have:
\[
\ltrans{b}_{\lambda} f  = \begin{pmatrix}-\ltrans{B_2} f + \lambda_2 f\\ \ltrans{B_1}f - \lambda_1 f\\ \ltrans{i}f  \end{pmatrix} = 0.
\]
This means $\ltrans{b_{\lambda}}$ is not injective, and consequently $b_{\lambda}$ is not surjective.  Thus, $b$ is not surjective, leading to a contradiction.

 Now, let $(B_1, B_2, i, j)$ represent a point in $X$.  Consider the following complex of sheaves on $\P^2$:
\[
\C^n\otimes \cO(-1) \xrightarrow{a} (\C^n \otimes \cO) \oplus (\C^n \otimes \cO) \oplus (\C^r \otimes \cO) \xrightarrow{b} \C^n \otimes \cO(1),
\]
where $a = \begin{pmatrix}B_1x_0 - x_1\\B_2x_0-x_2\\j x_0 \end{pmatrix}$ and $b = \begin{pmatrix}-(B_2x_0 -x_2) & B_1x_0-x_1 & ix_0 \end{pmatrix}$ (it is not hard to check that indeed $ba = 0$).  We will need to show that this complex is a monad, and the cohomology $\cF = \ker(b)/\Im(a)$ is a torsion-free sheaf.

It is easy to see from the way it is defined that $a$ is injective on $\l$, so we only need to show that $a$ is injective on the complement $\P^2 - \l = \C^2$. If $z_1 = \frac{x_1}{x_0},z_2 = \frac{x_2}{x_0}$, we get that 
\[
a|_{\C^2}f = \begin{pmatrix}B_1f - z_1f\\B_2f-z_2f\\jf \end{pmatrix},
\]
where $f(z_1,z_2)$ is a local section of $\C^n\otimes \cO(-1)|_{\C^2}$.  If $a|_{\C^2}f = 0$, then $a|_{\C^2}f(\lambda_1,\lambda_2) = 0$ for all $(\lambda_1, \lambda_2) \in \C^2$.  Therefore, $f(\lambda_1,\lambda_2) = 0$ or $\lambda_1$ is an eigenvalue of $B_1$ and $\lambda_2$ is an eigenvalue of $B_2$.  However, $B_1$ and $B_2$ have only finitely many eigenvalues, so  $f = 0$ at all but finitely many points of $\C^2$.  Thus it is identically $0$ on $\C^2$, and $a$ is injective.

Note that $b$ is surjective on $\l$ by construction.  Now, suppose $b$ is not surjective on $\P^2 - \l$.  This means  $\ltrans{b_{\lambda}}$ is not injective for some $\lambda = [1:\lambda_1, \lambda_2]$, and consequently there is a $0 \neq f \in \ker(\ltrans{b_{\lambda}})$.  Let $U = \ker(f)$.  Reversing the argument above, we can see that $B_1(U) \subset U$, $B_2(U) \subset U$ and $\Im(i) \subset U$.  Therefore, $b$ is surjective. 

Thus we see that our complex is a monad.  Furthermore, if we restrict it to $\l$, then it can be written as the direct sum of complexes:
\[
\C^n \otimes \cO_{\l}(-1) \xrightarrow{\tilde{a}} \C^n\otimes \cO_{\l} \oplus \C^n \otimes \cO_{\l} \xrightarrow{\tilde{b}} \C^n \otimes \cO_{\l}(1),
\]
where $\tilde{a} = \begin{pmatrix}- x_1\\-x_2 \end{pmatrix}$ and $\tilde{b} = \begin{pmatrix} x_2 & -x_1\end{pmatrix}$, and
\[
0 \rightarrow \C^r \oplus \cO_{\l} \rightarrow 0.
\]
Note that the first complex is exact, so we get that $\cF_{\l} \cong \C^r \otimes \cO_{\l} \cong \cO_{\l}^r$, making $\cF$ trivial at infinity.

We can see that $\cF$ is torsion-free on the rest of $\P^2$ as follows.  Let $s = (s_1, s_2. s_3) \in \ker(b) \subset (\C^n \otimes \cO) \oplus (\C^n \otimes \cO) \oplus (\C^r \otimes \cO)(U)$ be a local section over an open subset $U \subset \P^2 - \l = \C^2$ such that these is a nonzero $f \in \cO(U)$ satisfying $fs \in \Im (a|_{U})$.  That is, there exists a $t \in \C^n\otimes \cO(-1)(U)$ such that:
\begin{align*}
& fs_1 = (B_1 - z_1)t\\
& fs_2 = (B_2 - z_2)t\\
& fs_3 = jt.
\end{align*}
For $l = 1,2$, we have that if $(B_l - z_l)$ is invertible on $U$, then $t = f\inv{(B_l-z_l)}s_l$.  Consequently, if either $(B_1 - z_1)$ or $(B_2 - z_2)$ is invertible, then $h = \frac{t}{f} \in \C^n\otimes \cO(-1)(U)$ is well-defined, and moreover $s = a|_{U}h$.  However,  $(B_1 - z_1)$ and  $(B_2 - z_2)$ simultaneously fail to be invertible only at finitely many points of $\C^2$ (i.e. pairs consisting of eigenvalues of $B_1$ and $B_2$, respectively).  Therefore, $h$ is well-defined at all but finitely many points of $U$, so it is defined on all of $U$.  Thus, $s \in \Im(a|_{U})$, and therefore $\cF(U)$ is torsion-free.

As seen previously, a long exact sequence for cohomology gives us the isomorphism $H^0(\l, \cF_{\l}) \cong H^0(\l, \C^r \otimes \cO_{\l})$.  Denote this isomorphism by $\Phi$.  Since $\cF$ is trivial at infinity, $\Phi: \cF_{\l} \rightarrow \cO_{\l}^r$ define a trivialization.  Note we have the two exact sequences 
\begin{align*}
& 0\rightarrow \C^n \otimes \cO(-1) \rightarrow \ker(b) \rightarrow \cF \rightarrow 0\\ 
& 0 \rightarrow \ker(b) \rightarrow (\C^n \otimes \cO) \oplus (\C^n \otimes \cO) \oplus (\C^r \otimes \cO) \rightarrow \C^n \otimes \cO(1) \rightarrow 0.
\end{align*}  
This means the total Chern class of $\cF$ may be computed from 
\[
c(\cF)c(\C^n \otimes \cO(-1))c(\C^n \otimes \cO(1)) = c((\C^n \otimes \cO) \oplus (\C^n \otimes \cO) \oplus (\C^r \otimes \cO)). 
\]
Indeed, $c(\C^n \otimes \cO(-1)) = (1-H)^n$, c($\C^n \otimes \cO(1)) = (1+H)^n$, and $c((\C^n \otimes \cO) \oplus (\C^n \otimes \cO) \oplus (\C^r \otimes \cO)) = 1$.  Therefore, 
$c(\cF) (1-H^2)^n = 1$.  Expanding this and noting that $H^*(\P^2, \Z) = \Z[H]/H^3$, gives us $c(\cF) = 1 + nH^2$.  Thus, $c_2(\cF) = n$.

It is not hard to check that the equivalence relation on the framed torsion-free sheaves gives quadruples $(B_1, B_2, i, j)$ equivalent under the given $\GL(n, \C)$ action and vice versa.  Furthermore, the two constructions described above are mutually inverse.  Therefore, we obtain a bijection of sets.
\end{proof}
\section{Algebraic group actions and quotients}
A common approach to constructing a moduli space (or moduli stack) para\-met\-rizing equivalence classes of elements of a set $\cA$ with respect to the equivalence relation $\sim$ is to identify $\cA$ with points of a \emph{parameter space} $Y$ and to encode $\sim$ in the action of an algebraic group $G$ on $Y$.  That is, the set of orbits $Y/G$ is in bijection with the set of equivalence classes $\cA/\sim$.  In the ideal scenario, $Y/G$ can be given the structure of an algebraic variety and made into a coarse or fine moduli space. 

\subsection{Actions of algebraic groups}
Let us recall a few basic definitions from the theory of algebraic groups.  As before, all varieties and schemes are over the field  of complex numbers $\C$.  Recall that an \emph{algebraic group} is a group $G$ such that $G$ is an algebraic variety where the multiplication operation $m: G \times G \rightarrow G$ and the inversion operation $i: G \rightarrow G$ are both morphisms of algebraic varieties.  Such a group is called an \emph{affine algebraic group} if the underlying variety is affine.  Examples include $\C$ with the standard addition operation, $\C^{\times}$ with respect to multiplication, and $\GL_n(\C)$.
The definitions of subgroup and normal subgroup can all be transferred to the algebraic setting by specifying that the subgroups should also be subvarieties.

An action of $G$ on a variety $X$ is called \emph{algebraic} if the corresponding mapping $a: G \times X \rightarrow X$ is a morphism of varieties.  A morphism of varieties $f: X \rightarrow Y$ is called $G$-\emph{equivariant} if $f(gx) = gf(x)$.  It is called $G$-\emph{invariant} if $f(gx) = f(x)$.The action of $G$ on $X$ induces an action on the $\C$-algebra of regular functions $\cO(X)$ which may be written as $g \cdot f(x) = f(\inv{g}\cdot x)$.  A regular function $f \in \cO(X)$ is called $G$-\emph{invariant} if $g \cdot f = f$.  Note that $G$-invariant functions can be defined similarly for regular functions $\cO(U)$ on any open subset $U \subset X$.    

For any point $x \in X$, the stabilizer $G_x$ is a closed subvariety of $G$ since it is the preimage of $x$ under the morphism $a_x: G \rightarrow X$ given by $a_x := a(-, x)$.  The structure of an orbit is more complicated and may be described as follows:

\begin{prp}
\label{orbit}
Let $G$ be an affine algebraic group and let $X$ be an algebraic variety with an algebraic action of $G$.  For any $x \in X$, the orbit $G \cdot x$ is open in its closure, and the boundary $\overline{G \cdot x} - G \cdot x$ consists of a union of orbits of strictly smaller dimension.
\end{prp}
\begin{proof}
The orbit $G \cdot x$ is the image of $G \times \{x\} \cong G$ in $X$ under the action morphism. Therefore, by Chevalley's Theorem, $G \cdot x$ is a constructible subset (i.e. a finite union of locally closed subsets) of $X$.  This is equivalent to the existence of a dense open subset $U \subset \overline{G \cdot x}$ of $G \cdot x$ such that $U \subset G \cdot x$ (Section AG.1.3 in \cite{Bor1991}).  Since the action of $G$ is transitive on the orbit, it follows that every element of $G \cdot x$ is contained in $gU$ for some $g \in G$.  Consequently, $G \cdot x = \bigcup_{g \in G} gU$, so $G \cdot x$ is open in $\overline{G \cdot x}$.

Note that the boundary $\overline{G \cdot x} - G \cdot x$ is invariant under the action of $G$.  Indeed, $\overline{G \cdot x}$ is invariant under the action of $G$ because multiplication by any $g \in G$ is an isomorphism and the image of $\overline{G \cdot x}$ contains $G \cdot x$.  Since $G \cdot x$ is also invariant under the action of $G$, then $\overline{G \cdot x} - G \cdot x$ is invariant.  This means $\overline{G \cdot x} - G \cdot x$ may be written as the union of orbits.  Furthermore, we have $\dim G \cdot x = \dim \overline{G \cdot x}$ since it is dense and open in its closure, while $\overline{G \cdot x} > \dim (\overline{G \cdot x} - G \cdot x)$.  Therefore, the dimension of the orbits comprising $\overline{G \cdot x} - G \cdot x$ is strictly smaller than the dimension of $G \cdot x$.

\end{proof}

Note that the proposition implies orbits of minimal dimension are closed (and hence that closed orbits exist).  Let us looks at some examples of algebraic group actions.

\begin{exa}
\label{quotientex}
Here are a few examples demonstrating the orbit structures of actions of affine algebraic groups. 
\begin{enumerate}
\item
\begin{enumerate}
\item The group $\C^{\times}$ acts on $\A^2 = \C^2$ by $t \cdot (x,y) = (tx, ty)$.  The orbits under this action are either punctured lines $\ell-\{(0,0)\}$ through the origin and the origin $\{(0,0)\}$ itself.  Note that $\{(0,0)\}$ is a $0$-dimensional closed orbit contained in the closure of each $1$-dimensional punctured line.  It's easy to generalize this to an action of $\C^{\times}$ on $\A^n$.  The orbits are once again just punctured lines through the origin together with the origin.   
\item The action of $\C^{\times}$ on $\A^2$ descends to $\A^2 - \{(0,0)\}$.  This removes the origin from the set of orbits described above, leaving only the punctured lines through the origin $\ell-\{(0,0)\}$.  This suggests that the orbit space may be viewed as an algebraic variety since it can be identified (set-theoretically) with $\P^1$.  Similarly, the orbit space of the action of $\C^{\times}$ on $\A^n$ can be identified with $\P^{n-1}$. 
\end{enumerate}
\item Slightly modifying the first example above, we can define a $\C^{\times}$ action on $\A^2$ by $t \cdot (x,y) = (tx, \inv{t}y)$.  For this action we can describe the orbits as follows:
\begin{itemize}
\item the curves $\{(x,y)| xy = a\}\subset \A^2$ where $a \in \C^{\times}$ are $1$-dimensional orbits,
\item degenerate case $xy = 0$ is the union of the punctured coordinate axes $\{(x,0)\} \subset \A^2$ and $\{(0,y)\}\subset \A^2$, which are both $1$-dimensional orbits,
\item the origin $\{(0,0)\}$ is a $0$-dimensional closed orbit contained in the closure of each of the other orbits.
\end{itemize}
\item The group $\GL(2,\C)$ acts on the space of $2 \times 2$ matrices $\Hom(\C^2, \C^2)$ by conjugation $g \cdot B = g B\inv{g}$.  Note that the orbit of $B$ may be identified with the Jordan normal form of $B$.  Therefore, we can see that the orbits under this action may be described as follows:
\begin{itemize}
\item the conjugacy class of each Jordan block $\begin{pmatrix}\lambda & 1 \\ 0 & \lambda\end{pmatrix}$ is a $2$-dimensional orbit,
\item the conjugacy class of each diagonal matrix $\begin{pmatrix}\lambda_1 & 0 \\ 0 & \lambda_2\end{pmatrix}$ where $\lambda_1 \neq \lambda_2$ defines a closed $2$-dimensional orbit,
\item each scalar matrix $\begin{pmatrix} \lambda & 0 \\ 0 & \lambda\end{pmatrix}$ is a closed $0$-dimensional orbit contained in the closure of the orbit defined by $\begin{pmatrix}\lambda & 1 \\ 0 & \lambda\end{pmatrix}$. 
\end{itemize}  
\end{enumerate}
\end{exa}  

\subsection{Quotients}
Recall from differential geometry that given a smooth manifold $M$ with a (smooth) free, proper action of Lie group $G$, the set of orbits $M/G$ can be given a unique structure of a smooth manifold such that the quotient map $M \rightarrow M/G$ is a submersion (in fact, $M \rightarrow M/G$ is a principal $G$-bundle).

  Similarly, in the case of an affine algebraic group $G$ acting on an algebraic variety $X$, we would like to endow the set of orbits $X/G$ with a reasonable enough structure of an algebraic variety.  Unfortunately, $X/G$ may be poorly behaved and not allow for such a structure.  Thus, before constructing quotients we should specify what properties we would like these quotients to have.  We begin by defining a general concept of quotient in the category of varieties. 

\begin{defn}
Let $G$ be an affine algebraic group acting on a variety $X$.  A \textbf{categorical quotient} of the action of $G$ on $X$ is a pair $(Y, \phi)$ where $Y$ is an algebraic variety and $\phi: X \rightarrow Y$ is a morphism such that the following conditions hold:
\begin{enumerate}
\item $\phi(g \cdot x) = \phi(x)$
\item For any algebraic variety $Z$ with an action of $G$ and any morphism $\psi: X \rightarrow Z$ satisfying (1), there is a unique $f: Y \rightarrow Z$ such that $\psi = f \circ \phi$.
\end{enumerate}
If $\inv{\phi}(y)$ is a single orbit of the $G$ action, then $(Y, \phi)$ is called an \textbf{orbit space}.
\end{defn}
Note that the first condition means that $\phi$ is a $G$-invariant morphism and the second is the universal property.  This means $Y$ is unique up to unique isomorphism. We will refer to just $Y$ as the categorical quotient when no confusion can occur. 

\begin{exa}
\label{categoricalex}
 If $(Y, \phi)$ is a categorical quotient, then $\phi$ is invariant on any orbit, so it maps the entire orbit to a single point.  However, since $\phi$ is continuous, the preimage of that point is a closed subset.  Therefore, $\phi$ is constant on the closure of any orbit.  

In part 1 (a) of Example \ref{quotientex} there is an orbit consisting of single point contained in the closure of every other orbit.  This means $(\Spec \C, \phi)$, where $\phi$ is the constant morphism, is the categorical quotient for each of these examples. Indeed, it is easy to see that the constant morphism is $G$-invariant.  Furthermore, any $G$-invariant morphism $\psi: \A^2 \rightarrow Z$ must be constant.  If the image of $\psi$ is $z \in Z$, then $f: \Spec \C \rightarrow Z$ defined by $f(\Spec \C) = z$ satisfies $\psi = f \circ \phi$.  It is also clearly unique.

In part 1 (b) of Example \ref{quotientex}, the categorical quotient is $\P^1$ together with the standard quotient map $\phi: \A^2 - \{(0,0)\} \rightarrow \P^1$.  It's easy to see that $\phi$ is $\C^{\times}$-invariant.  Moreover, any morphism $\A^2 - \{(0,0)\} \rightarrow Z$ constant on lines through the origin descends to a unique $f: \P^1 \rightarrow Z$ such that $\psi = f \circ \phi$.
\end{exa}

Note that the discussion in the above example implies that a categorical quotient cannot be an orbit space unless all of the orbits are closed.  We would like to refine the concept of a categorical quotient in order to obtain a variety with additional geometric properties.  To do this, we will be working with the following definition motivated by geometric invariant theory.

\begin{defn}
Let $G$ be an affine algebraic group acting on a variety $X$.  A \textbf{good quotient} is a pair $(Y, \phi)$ where $Y$ is an algebraic variety and $\phi: X \rightarrow Y$ is an affine morphism such that the following holds:
\begin{enumerate}
\item $\phi$ is $G$-invariant,
\item $\phi$ is surjective,
\item if $U \subset Y$ is open, the induced morphism on regular functions $\cO_Y(U) \rightarrow \cO_X(\inv{\phi}(U))$ is an isomorphism between $\cO_Y(U)$ and the $G$-invariant functions in $\cO_X(\inv{\phi}(U))$,
\item if $W \subset X$ is a closed $G$-invariant subset, then $\phi(W)$ is closed,
\item if $W_1, W_2$ are closed $G$-invariant subsets of $X$ such that $W_1\cap W_2 = \varnothing$, then $\phi(W_1) \cap \phi(W_2) = \varnothing$.
\end{enumerate}
If $\inv{\phi}(y)$ is a single orbit of the $G$ action, then $(Y, \phi)$ is called a \textbf{geometric quotient}. 
\end{defn}

\begin{rmk}
\label{goodrmk}
An immediate consequence of the definition of good quotient is that for any $x_1, x_2 \in X$, we have $\phi(x_1) = \phi(x_2)$ if an only if $\overline{G \cdot x_1} \cap \overline{G \cdot x_2} \neq \varnothing$.  Indeed, since $\phi$ is $G$-invariant, then it is constant on orbit closures, and therefore $\overline{G \cdot x_1} \cap \overline{G \cdot x_2} \neq \varnothing$ implies $\phi(x_1) = \phi(x_2)$.  Conversely, if $\phi(x_1) = \phi(x_2)$, then (5) in the definition of good quotient implies $\overline{G \cdot x_1} \cap \overline{G \cdot x_2} \neq \varnothing$.  This means good quotients separate orbit closures.

Another consequence of the definition is that if all the orbits of an action with a good quotient are closed, then the quotient is geometric.  Let $y \in Y$.  The preimage $\inv{\phi}(y)$ is nonempty, so it contains at least one orbit.  It must also contain that orbit's closure, and therefore it contains a closed orbit (by Proposition \ref{orbit}).  It cannot contain more than one closed orbit, since that would contradict property (5) in the definition of good quotient.  If all orbits are closed, then $\inv{\phi}(y)$ must contain exactly one orbit, so the quotient is geometric. 
\end{rmk}

All of the examples of group actions considered in Example \ref{quotientex} have good quotients.  It is not hard to check that the categorical quotients described in Example \ref{categoricalex} also happen to be good.  The remaining cases are addressed below.   

\begin{exa}
We will look at the following examples more closely after we define the affine GIT quotient.
\begin{enumerate}
\item The pair $(\A^1, \phi_1)$, where $\phi_1: \A^2 \rightarrow \A^1$ is defined by $\phi_1(x,y) = xy$, is a good quotient for the action in Example \ref{quotientex} (2).
\item The pair $(\A^2, \phi_2)$, where $\phi_2: \Hom(\C^2,\C^2) \rightarrow \A^2$ is defined by $\phi_2(A) = (\tr(A), \det(A))$ is a good quotient for the action in Example \ref{quotientex} (3).
\end{enumerate}
\end{exa}

As indicated by these examples, categorical and good quotients are related to each other.

\begin{prp}
Let $G$ be an affine algebraic group acting on a variety $X$.  Any good quotient for this action is a categorical quotient.
\end{prp}

\begin{rmk}
\label{stackyquotient}
There is a third possible type of quotient for an algebraic group $G$ acting on an algebraic variety $X$ that we have already considered.  Namely, the quotient stack $[X/G]$.
Recall that a quotient stack is a sheaf of groupoids in the fppf topology on the category of schemes:
\begin{align*}
& [X/G]: \opcat{Sch}_{fppf} \rightarrow \Gpd \\
& [X/G](T) = \left\langle 
     \begin{tabular}{lll} $\pi: E \rightarrow S$ is a principal $G$-bundle over $X$\\ $\p: E \rightarrow X$ is an equivariant morphism \end{tabular}
   \right\rangle
\end{align*}
There is a $2$-categorical version of the Yoneda lemma (\cite{Vi2005}) which implies that $\Hom(T, [X/G]) \cong [X/G](S)$.  Since the $\C$-points of [X/G] are equivalence classes of morphisms $\Spec \C \rightarrow [X/G]$, then this implies these points may be viewed as $G$-equivariant morphisms $p: G \rightarrow X$.  In other words, they are orbits under the action of $G$ on $X$.  
\end{rmk}

\section{Geometric invariant theory}
\label{GIT}
\subsection{The affine quotient}
\label{affinegit}
We would like to begin our construction of good quotients by considering the case when an affine algebraic group $G$ acts on an affine variety $X$.  Recall that this action may be written as a morphism of varieties $a: G \times X \rightarrow X$.  Such a morphism induces a corresponding homomorphism of $\C$-algebras of regular functions:
\begin{align*}
& a^\#: \cO(X) \rightarrow \cO(G \times X)\\ 
& a^\#(f)(g, x) = f(g \cdot x).
\end{align*}
This, in turn, defines a $G$-action on $\cO(X)$ as follows:
\[
g \cdot f(x) = f(\inv{g} \cdot x).
\] 
The set $\cO(X)^G  := \{f \in \cO(X)|g \cdot f = f\}$, consisting of $G$-invariant functions, forms a subalgebra of $\cO(X)$.  If $\phi: X \rightarrow Y$ is a $G$-invariant morphism of varieties, then the induced homomorphism on the regular functions $\phi^\#: \cO(Y) \rightarrow \cO(X)$ satisfies $g \cdot \phi^\#(f)(x) = f(\phi(\inv{g}\cdot x)) = \phi^\#(f)(x)$.  It follows that the image of $\phi^\#$ is contained in $\cO(X)^G$.  In fact, if $(Y, \phi)$ is a good quotient, then the image must be isomorphic to $\cO(X)^G$.  Thus, in order for a good quotient to exist (as a variety), we need this subalgebra to be finitely generated.  Unfortunately, this is not the case for the action of an arbitrary affine algebraic group on an affine variety (Nagata constructs a counterexample in \cite{Nag1960}).  However, it can be shown to be true if we only consider the actions of a specific type of group.

Recall that there is a multiplicative version of the Jordan decomposition that can be obtained for elements of $\GL(n, \C)$.  Indeed, for any $g \in \GL(n, \C)$ the Jordan normal form yields $g = C + N$, where $C$ is a diagonalizable matrix and $N$ is a nilpotent matrix.  Noting that $C$ is invertible, we can write $U = I_n + \inv{C}N$ and obtain the decomposition $g = CU$.  The matrix $C$ is diagonalizable, so $\inv{C}N$ is nilpotent.  This makes $U$ a unipotent matrix.  

A similar result may be obtained for an arbitrary affine algebraic group $G$ by considering faithful representations $G \hookrightarrow \GL(n,\C)$ (these always exist; see Proposition I.1.10 in \cite{Bor1991}).  

\begin{thm}[Jordan decomposition]
\label{jordan}
Let $G$ be an affine algebraic group. 
\begin{enumerate}
\item For any $g \in G$ and faithful representation $\rho: G \hookrightarrow \GL(n, \C)$ there exist unique elements $g_s,g_u \in G$ such that $g = g_sg_u = g_ug_s$, where $\rho(g_s)$ is diagonalizable and $\rho(g_u)$ is unipotent.
\item The decomposition in (1) does not depend on the choice of $\rho$.
\item If $f: G \rightarrow G'$ is a homomorphism of algebraic groups, then we have that $f(g) = f(g_s)f(g_u)$ is the decomposition given in (1) for $f(g)$.\
\end{enumerate}
\end{thm}

We can immediately formulate the following definition:

\begin{defn}
An element $g$ of an affine algebraic group $G$ is called \textbf{semisimple} if there exists a faithful representation $\rho: G \hookrightarrow \GL(n,\C)$ such that $\rho(g)$ is a diagonalizable matrix.  We call $g$ \textbf{unipotent} if $\rho(g)$ is a unipotent matrix.
\end{defn}

\begin{defn}
An algebraic group is called \textbf{unipotent} if all of its elements are unipotent.
\end{defn}  

\begin{exa}
Let $\bU_n \subset \GL(n, \C)$ be the subgroup consisting of upper triangular matrices with 1s on the diagonal.  It is clear that any subgroup of $\bU_n$ is unipotent.

\end{exa}

\begin{defn}
An affine algebraic group is called \textbf{reductive} if every unipotent normal algebraic subgroup is trivial.
\end{defn}  

\begin{exa}
\leavevmode
\begin{enumerate}
\item The group $\C^{\times} = \GL(1,\C)$ is clearly reductive, since its only unipotent element is the identity. 
\item More generally, the group $\GL(n, \C)$ is reductive.  Indeed, let $U \in G$ be an element of a unipotent normal subgroup $G \subset \GL(n, \C)$.  Since $G$ is normal, then the Jordan normal form $J_U$ of $U$ is also in $G$.  Moreover, since $\ltrans{J_U}$ has the same Jordan normal form as $J_U$, then it is also in $G$.  The matrix $\ltrans{J_U}J_U$ is real symmetric (the eigenvalues of unipotent matrix are all equal to $1$), so it is diagonalizable.  The only way for such a matrix to also be unipotent is if it is equal to $I_n$.  However, this is only possible if $J_U = I_n$.  Thus, $G$ is trivial.
\item The direct product of general linear groups $\GL(n,\C) \times \GL(m,\C)$ is reductive.  Since there is an injective morphism $\GL(n,\C) \times \GL(m,\C) \hookrightarrow \GL(n+m,\C)$ defined $(A,B) \mapsto \begin{pmatrix} A &0\\ 0 & B\end{pmatrix}$, the projection of a normal unipotent subgroup $H$ of $\GL(n,\C) \times \GL(m,\C)$ onto either $\GL(n,\C)$ or $\GL(m,\C)$ is a normal unipotent subgroup.  This means both projections are trivial, and therefore $H$ is trivial.  It is easy to see this generalizes to show that any finite direct product of reductive groups is reductive.

\item The (additive) group $\C$ is not reductive because it is unipotent.  Indeed, the homomorphism $\C \rightarrow \GL(2, \C)$ defined by $x \mapsto \begin{pmatrix}1 & x\\ 0 & 1\end{pmatrix}$
sends every element of $\C$ to a unipotent matrix.
\end{enumerate}
\end{exa}

If a reductive group $G$ is acting on an affine variety $X$, then a theorem of Nagata (\cite{Nag1963}) says that the algebra of invariants $\cO(X)^G$ is finitely generated.  Let $X//G$ be the affine variety corresponding to $\cO(X)^G$.  Note that $\cO(X)^G$ is automatically reduced because $\cO(X)$ is reduced.  The inclusion $\cO(X)^G \hookrightarrow \cO(X)$ induces an affine morphism of varieties $\phi: X \rightarrow X//G$.  We call $(X//G, \phi)$ the \emph{affine GIT quotient}.

\begin{thm}
Let $G$ be a reductive group acting on an affine variety $X$.  The affine GIT quotient $(X//G, \phi)$ is a good quotient.
\end{thm}

A full proof of this theorem may be found in \cite{New1978}.  Our previous discussion concerning good quotients in Remark \ref{goodrmk} immediately yields the following corollary.
\begin{cor}
\label{gitgeom}
Let $G$ be a reductive group acting on an affine variety $X$ such that $(X//G, \phi)$ is the affine GIT quotient.  We have $\phi(x_1) = \phi(x_2)$ if and only if $\overline{G \cdot x_1} \cap \overline{G \cdot x_2} \neq \varnothing$.  Furthermore, for any $y \in X//G$ the preimage $\inv{\phi}(y)$ contains a unique closed orbit.  If all the $G$-orbits under the action on $X$ are closed, then $(X//G, \phi)$ is a geometric quotient.
\end{cor}

\begin{rmk}
To keep the exposition simple we have combined several different notions of reductivity into one.  Namely, in addition to definition given above, there are two more that appear in geometric invariant theory.  If $G$ is an affine algebraic group, then
\begin{itemize}
\item $G$ is called \emph{geometrically reductive} if for every finite dimensional representation $\rho: G \rightarrow \GL(n,\C)$ and $G$-invariant vector $v\in V$ (i.e. $g \cdot v = v$) there exists a $G$-invariant homogeneous polynomial $f$ with $\deg f \ge 1$ such that $f(v) \neq 0$;
\item $G$ is called \emph{linearly reductive} if the same is true for an invariant homogenous polynomial of degree $1$.
\end{itemize}  
Nagata's theorem concerning algebras of invariants used geometrically reductive groups. A similar result was proved for actions of linearly reductive groups (\cite{Hil1890}, \cite{Wey1925}).  However, since we are working over $\C$, which has characteristic $0$, reductive, linearly reductive, and geometrically reductive groups all coincide (see \cite{Wey1925}, \cite{Nag1963}, and \cite{Hab1975}, among others).
\end{rmk}

Let us compute the affine GIT quotients for some of the examples we have seen so far.

\begin{exa}
\label{affinegitex}
\leavevmode
\begin{enumerate}
\item If $\C^{\times}$ acts on $\A^2$ by $t \cdot (x,y) = (tx, ty)$, then the induced action on the algebra of functions $\cO(\A^2) = \C[x,y]$ is $t \cdot f(x,y) = f(\inv{t}x, \inv{t}y)$.  The only invariants with respect to this action are the constant functions.  That it, $\cO(X)^{\C^{\times}} = \C$.  Thus, the affine GIT quotient is the one point variety $\Spec \C$ together with the constant morphism $\phi: \A^2 \rightarrow \Spec \C$.

\item If $\C^{\times}$ acts on $\A^2$ by $t \cdot (x,y) = (tx, \inv{t}y)$, then the induced action on $\cO(\A^2) = \C[x,y]$ is $t \cdot f(x,y) = f(\inv{t}x, ty) = \sum_{i,j} t^{j-i}x^iy^j$.  In order for $f(x,y)$ to be $G$-invariant we must have $i = j$.  Therefore, we have $\cO(X)^{\C^{\times}} = \C[xy]$.  It follows that the affine GIT quotient is $(\A^1, \phi)$, where $\phi: \A^2 \rightarrow \A^1$ is defined by $\phi(x,y) = xy$.

\item If $G = \GL(2,\C)$ acts on $\A^4 = \Hom(\C^2, \C^2)$ by $g \cdot B = g B \inv{g}$, where $B = \begin{pmatrix}x_1 & x_2\\ x_3 & x_4\end{pmatrix}$, then the induced action on $\cO(\A^4) = \C[x_1,x_2,x_3,x_4]$ is $g\cdot f(x) = f(\inv{g}\cdot x)$.  

If $f$ is $G$-invariant, then it is constant on the conjugacy class closures of $2 \times 2$ matrices.  Therefore, by Example \ref{quotientex}, we see that $f$ is determined by its values on diagonal matrices.  This means $\cO(\A^4)^G$ is contained in the subalgebra $\C[\lambda_1,\lambda_2]$ (where $\lambda_1$ and $\lambda_2$ represent the eigenvalues).  Moreover, since the Jordan normal form is only determined up to permutation of the Jordan blocks, then $G$-invariants in $\C[\lambda_1,\lambda_2]$ must be invariant under the action of the symmetric group $S_2$.  Therefore, $\cO(\A^4)^G$ is contained in $\C[\lambda_1,\lambda_2]^{S_2}$.  Note that $\tr B$ and $\det B$ are both $G$-invariant regular functions, restricting to the elementary symmetric polynomials $\lambda_1 + \lambda_2$ and $\lambda_1 \lambda_2$ on diagonal matrices.  Thus $\cO(\A^4)^G = \C[\lambda_1,\lambda_2]^{S_2} \cong \C[\lambda_1,\lambda_2]$.  Consequently, the affine GIT quotient is $(\A^2, \phi)$, where $\phi: \A^4 \rightarrow \A^2$ is given by $\phi(B) = (\tr B, \det B)$

\item The above example can be generalized to the action of $G = \GL(n,\C)$ on $\A^{n^2} = \Hom(\C^n, \C^n)$ by conjugation. As in the case of $2\times2$ matrices, each conjugacy class closure contains a diagonal matrix (the diagonal of the Jordan normal form).  Therefore, the algebra of invariants $\cO(\A^{n^2})^G$ is contained in the algebra of regular functions on the diagonal matrices $\C[\lambda_1, \lambda_{2} \cdots, \lambda_{n}]$.  Invariance under the permutation of diagonal elements means $\cO(\A^{n^2})^G$ is actually contained in the algebra $\C[\lambda_1,\lambda_{2} \cdots, \lambda_{n}]^{S_n}$ of invariants under the action of the symmetric group.

Let $\phi: \A^{n^2} \rightarrow \A^n$ be defined by 
$\phi(B) = (a_0, \cdots, a_{n-1})$, where $p_B(t) = a_0 + \cdots a_{n-1}t^{n-1} + t^n$ is the characteristic polynomial of $B$. Restricting $\phi$ to diagonal matrices, we see that $a_0, \cdots, a_{n-1}$ are the elementary symmetric polynomials in $\lambda_1,\lambda_{2} \cdots, \lambda_{n}$.  These functions are $G$-invariants (since the characteristic polynomial is invariant under conjugation) and also generate $\C[\lambda_1,\lambda_{2} \cdots, \lambda_{n}]^{S_n}$.  Thus, $(\A^n, \phi)$ is the affine GIT quotient for the action of $\GL(n, \C)$.
\end{enumerate}
\end{exa}

As we have already seen in Corollary \ref{gitgeom}, an affine GIT quotient $X//G$ cannot be a geometric quotient unless the orbits of the group action are closed.  Since this is not always the case, we would like there at least to be a large enough subset of $X$ such that the restriction of the quotient morphism $\phi$ to this subset defines a geometric quotient.  In order to characterize the points of such a subset we need the following definition:

\begin{defn}
A point $x \in X$ is called \textbf{stable} if the orbit $G\cdot x$ is closed and the stabilizer $G_x$ is finite.
\end{defn}

The set of stable points of $X$ is exactly the open subset we need to obtain geometric quotient. 

\begin{prp}
\label{affinestab}
Let $G$ be a reductive group acting on an affine variety $X$ with affine GIT quotient $(X//G, \phi)$.  The set $X_s$ of stable points is a $G$-invariant subset of $X$ such that the following properties hold:
\begin{enumerate}
\item $X_s = \inv{\phi}(\phi(X_s))$,
\item $\phi(X_s)$ is open,
\item $(\phi(X_s), \phi|_{X_s})$ is a geometric quotient.
\end{enumerate}
\end{prp}

For a proof of this proposition see Section 5.1.c in \cite{Muk2003}.  Let us determine the stable points for each of the group actions in the previous example.

\begin{exa}
\leavevmode
\begin{enumerate}
\item The action of $\C^{\times}$ on $\A^2$ by $t \cdot (x,y) = (tx, ty)$ has no stable points, since the $\{(0,0)\}$ is the only closed orbit, and the entire group $\C^{\times}$ is its stabilizer.

\item The stable points of the $\C^{\times}$ action on $\A^2$ by $t \cdot (x,y) = (tx, \inv{t}y)$ are just the points $(x,y)$ belonging to the curves $xy = a$ for $a \in \C^{\times}$.  The restriction $\phi: \A^2 - \{(x,y)| xy = 0\} \rightarrow \A^1-\{0\}$ of the quotient morphism $\phi(x,y) = xy$ defines a geometric quotient.

\item The action by conjugation of $\GL(n, \C)$ on the space of matrices $\A^{n^2} = \Hom(\C^n, \C^n)$ also has no stable points, since every Jordan normal form has a nontrivial stabilizer.
\end{enumerate}
\end{exa}

\subsection{The Proj quotient}
\label{projgit}
We now proceed to describe a variation of Mumford's general GIT quotient construction for reductive group actions on affine varieties.  Let $G$ be a reductive group acting on an affine algebraic variety $X$.  Let $\chi: G \rightarrow \C^{\times}$ be a character of the group $G$ ($\chi$ is taken to be algebraic).  The group $G$ acts on the direct product $X \times \C$ by 
\[
g \cdot (x,z) = (g \cdot x, \inv{\chi(g)}z).
\]
Note that the algebra of regular functions is $\cO(X \times \C) = \cO(X)\otimes\C[z] = \cO(X)[z]$.  The induced $G$-action on this algebra may be written as 
\[
g \cdot f(x,z) = g \cdot (\sum_{n=0}^l f_n(x)z^n) =  \sum_{n=0}^l \chi(g)^{-n}f_n(\inv{g} \cdot x)z^n.
\]
Therefore, if $f(x,z)$ is $G$-invariant, then $f_n(\inv{g}x) = \chi(g)^n f_n(x)$ for all $0 \le n \le l$.  We call the functions $f_n(x) \in \cO(X)$ such that $f_n(\inv{g}x) = \chi(g)^n f_n(x)$ is satisfied \emph{$\chi^n$-semi-invariants}.  Note that semi-invariants define a $\Z_{\ge 0}$-grading on the algebra of invariants so that we have:
\[
A_{\chi} := \cO(X \times \C)^G = \bigoplus_{n\ge 0} \cO(X)^{\chi^n}, 
\]
where $\cO(X)^{\chi^n} \subset \cO(X)$ is the subspace of $\chi^n$-semi-invariants.  It can be shown using Hilbert's theorem that the graded algebra $\cO(X \times \C)^G$ is finitely generated (Section 6.1.b in \cite{Muk2003}).  We will denote by $X//_{\chi}G$ the quasi-projective variety corresponding to the closed points of $\Proj \cO(X \times \C)^G$.

Specifically, recall from Section 2 of \cite{Ha1977} that for a graded algebra $S = \bigoplus_{n \ge 0} S_n$ there exists a quasi-projective scheme $\Proj S$ whose points are homogeneous prime ideals that do not contain the \emph{irrelevant ideal} $S_{+} = \bigoplus_{n > 0} S_n$.  The closed points of $\Proj S$ are the ideals that are maximal among homogeneous ideals not containing the irrelevant ideal.  Considering only the closed points when $\Proj S$ is a reduced scheme of finite type we obtain a quasi-projective variety.  

Note that the inclusion $S_0 \subset S$ induces a projective morphism $\Proj S \rightarrow \Spec S_0$ (this is true on the level of varieties as well).  In the case of the graded algebra of invariants $A_{\chi}$ we can interpret this as saying that there is a projective morphism $X//_{\chi}G \rightarrow X//G$ induced by the inclusion $\cO(X)^G = \cO(X)^{\chi^0} \subset \cO(X \times \C)^G$.

We would like $X//_{\chi}G$ to be a good quotient of a $G$-action just as we did in the case of the affine GIT quotient.  To do this we require the following definition:

\begin{defn}
Let $G$ be a reductive group acting on an affine variety $X$, and for nonzero $f \in \cO(X)^{\chi^n}$ let $X_f = \{x \in X| f(x) \neq 0\}$. 
\begin{enumerate}
\item A point $x \in X$ is called \textbf{$\chi$-semistable} if for some $n > 0$ there exists a semi-invariant $f \in \cO(X)^{\chi^n}$ such that $x \in X_f$.  We denote the set of $\chi$-semistable points of $X$ by $X^{ss}_{\chi}$.  Points that are not semistable are called \textbf{unstable}.
\item If $x_1, x_2 \in X^{ss}_{\chi}$, then $x_1$ and $x_2$ are called \textbf{S-equivalent} if and only if $\overline{G \cdot x_1} \cap \overline{G \cdot x_2} \cap X^{ss}_{\chi}\neq \varnothing$.
\end{enumerate}
\end{defn}

Note that the subset $X^{ss}_{\chi} = \bigcup X_f$ (the union is taken over all semi-invariants $f$ or positive degree).  Since each $X_f$ is open by definition and $G$-invariant because $f$ is a semi-invariant, then $X^{ss}_{\chi}$ is an open $G$-invariant subset of $X$.

If $x \in X^{ss}_{\chi}$, then let $I_x = \{f \in A_{\chi}| f(x) = 0\}$ be a homogeneous ideal of $A_{\chi}$.  This ideal does not contain the irrelevant ideal because $x$ is semistable.  Furthermore, it is maximal since $x$ is a point.  Thus it defines a point in $X//_{\chi}G$.  This defines a natural morphism $\phi: X^{ss}_{\chi} \rightarrow X//_{\chi}G$.  The pair $(X//_{\chi}G, \phi)$ is called the \emph{Proj GIT quotient} with respect to $\chi$.  The following theorem is a special case of the results found in Section 1.4 of \cite{MFK1994}:

\begin{thm}
\label{projgood}
Let $G$ be a reductive group acting on an affine variety $X$.  The Proj GIT quotient $(X//_{\chi}G, \phi)$ is a good quotient, and the morphism $\phi$ induces a bijection between the points of $X//_{\chi}G$ and the S-equivalence classes of $G$-orbits in $X^{ss}_{\chi}$
\end{thm}

\begin{exa}
\label{projgitex}
\leavevmode
\begin{enumerate}
\item Consider the action of $G = \C^{\times}$ on $X = \A^2$ by $t \cdot (x,y) = (tx, ty)$, together with the character $\chi: \C^{\times} \rightarrow \C^{\times}$ given by $\chi(t) = \inv{t}$.  For $f(x,y) = \sum_{i,j}a_{ij}x^iy^j \in \cO(X)$ we have 
\[
t \cdot f(x,y) = \sum_{i+j} a_{ij} t^{-n}x^iy^j,
\]
so the $\chi^n$-semi-invariants are degree $n$ homogeneous polynomials.  It follows that $A_{\chi} = \C[x,y]$ with the standard $\Z_{\ge 0}$-grading.  Consequently, $\A^2//_{\chi} \C^{\times} = \Proj \C[x,y] = \P^1$.  Note that $X^{ss}_{\chi} = \A^2 - \{(0,0)\}$, since every homogenous polynomial of degree greater than $0$ vanishes at the origin.  Therefore, $(\P^1, \phi)$ where $\phi: \A^2 - \{(0,0)\} \rightarrow \P^1$ is the morphism sending every line through the origin to the corresponding point of $\P^1$ is the Proj GIT quotient.  It is not difficult to generalize this construction to the action of $\C^{\times}$ on $\A^{n+1}$ to obtain $\P^n$ as the quotient.  In this case, since orbits intersect outside of $X^{ss}_{\chi}$, the $S$-equivalence classes of orbits are the orbits themselves (i.e. lines going through the origin).

\item Let $G$ be reductive group acting on an affine variety $X$, and let $\chi: G \rightarrow \C^{\times}$ be the trivial character.  In this case, we have that $A_{\chi} = \cO(X)^G\otimes \C[z] = \cO(X)^G[z]$, where elements of $\cO(X)^G$ have degree $0$ and $z$ has degree $1$. Note that a homogeneous ideal of $\cO(X)^G[z]$ contains the irrelevant ideal if and only if it contains $z$. Therefore, any maximal ideal $\fm \subset \cO(X)^G$ determines a unique homogeneous maximal ideal $M \subset \cO(X)^G[z]$ not containing the irrelevant ideal such that $\fm  = M \cap \cO(X)^G$.  It follows that $X//_{\chi}G = \Proj \cO(X)^G[z] = \Spec \cO(X)^G = X//G$.  Since $z$ is nonzero on all of $X$, then $X^{ss}_{\chi} = X$.  Consequently, the Proj GIT quotient $(X//_{\chi}G, \phi)$ is the same as the affine GIT quotient.
\end{enumerate}
\end{exa}

As in the case of the affine GIT quotient, we can impose additional conditions on the points we consider in order to obtain an open subset of $X$ on which the $G$-action gives us a geometric quotient.

\begin{defn}
Let $G$ be a reductive group acting on an affine variety $X$.  A point $x \in X$ is called \textbf{$\chi$-stable} if for some $n > 0$ there exists a semi-invariant $f \in \cO(X)^{\chi^n}$ such that $x \in X_f$ (i.e. $x$ is semistable), the action $G$ restricted to $X_f$ is closed, and the stabilizer $G_x$ is finite.  We denote the subset of $\chi$-stable points of $X$ by $X^{s}_{\chi}$.
\end{defn}

We have the following analogue of Proposition \ref{affinestab}:

\begin{prp}
\label{projstab}
Let $G$ be a reductive group acting on an affine variety $X$ with Proj GIT quotient $(X//_{\chi}G, \phi)$.  The set $X^s_{\chi}$ of stable points is a $G$-invariant subset of $X$ such that $X^s_{\chi} = \inv{\phi}(\phi(X^s_{\chi}))$, $\phi(X^s_{\chi})$ is open, and $(\phi(X^s_{\chi}), \phi|_{X^s_{\chi}})$ is a geometric quotient.
\end{prp}

The proof can be found in Section 1.4 of \cite{MFK1994}.  

\begin{exa}
\leavevmode
\begin{enumerate}
\item If $G = \C^{\times}$ acts on $X = \A^{n+1}$ by $t \cdot (x_0, \dots, x_n) = (tx_0, \dots, tx_n)$, and $\chi: \C^{\times} \rightarrow \C^{\times}$ is given by $\chi(t) = \inv{t}$, then the stable points are those with finite stabilizers and closed orbits in $\A^{n+1} -  \{(0, \dots, 0) \}$.  This means $X^s_{\chi} = \A^{n+1} -  \{(0, \dots, 0) \}$.
\item If $G$ is a reductive group acting on an affine variety $X$, and $\chi: G \rightarrow \C^{\times}$ is the trivial character, then $X^s_{\chi} = X_s$. 
\end{enumerate}
\end{exa}

Note that an analogue of the corresponding result for quotients by actions of Lie groups can be seen from the following theorem (see Section 1.4 and Appendix 1.D in \cite{MFK1994} ):
\begin{thm}
\label{projbun}
Let $G$ be a reductive group acting on a smooth affine variety $X$ such that all points of $X$ have connected stabilizers, and let $\chi$ be a character of $G$.  Let $(Y, \phi)$ be the geometric quotient corresponding to $X^s_{\chi}$.  

If the action of $G$ is free on $X^{ss}_{\chi}$, then all $\chi$-semistable points of $X$ are $\chi$-stable, $Y$ is smooth, and $\phi: X^s_{\chi} \rightarrow Y$ is a principal $G$-bundle in the \'{e}tale topology.  
\end{thm}

We can see that the relationship between the $\chi$-stable and $\chi$-semistable points of $X$ in the following diagram:

\begin{equation*}
\begin{tikzcd}
X^s_{\chi} \arrow[d] & \subset & X^{ss}_{\chi} \arrow[d] & \subset & X \arrow[d] \\
X^s_{\chi}/G & \subset & X//_{\chi}G \arrow[rr] & & X//G
\end{tikzcd}
\end{equation*}

\begin{rmk}
The Proj GIT quotient described in this section is a special case of Mumford's GIT quotient with respect to a specific linearization.  Namely, let $G$ be a reductive group acting on a quasi-projective variety $X$.  A \emph{linearization} of the action of $G$ with respect to a line bundle $\pi: L \rightarrow X$ is an action of $G$ on $L$ such that 
\begin{enumerate}
\item the morphism $\pi$ is $G$-equivariant
\item for all $x \in X$ and $g \in G$, the morphism $L_x \rightarrow L_{g\cdot x}$ induced by the action of $G$ on the fibers is linear.
\end{enumerate}
In other words, $L$ is a $G$-equivariant line bundle on $X$.  The action of $a: G \times X \rightarrow X$ induces the action $\tilde{a}: G \times L \rightarrow L$, which in turn defines a linear action of $G$ on $H^0(X,L^{\otimes r})$ for $r \ge 0$.  The GIT quotient is defined in terms of 
\[
X//_L G : =\Proj \left( \bigoplus_{r \ge 0} H^0(X, L^{\otimes r})^G \right).
\]
The semistable and stable points are defined as
\begin{itemize}
\item $x \in X$ is \emph{semistable} if for some $r > 0$ there exists an invariant section $\sigma \in H^0(X, L^{\otimes r})^G$ such that $x \in X_{\sigma} = \{x \in X| \sigma(x) \neq 0 \}$ and $X_{\sigma}$ is affine;
\item  $x \in X$ is \emph{stable} if in addition to being semistable, the stabilizer $G_x$ is finite, and the action of $G$ on the corresponding $X_{\sigma}$ is closed.
\end{itemize}
In this case, there are versions of Theorems \ref{projgood}, \ref{projbun}, and Proposition \ref{projstab} (see Chapter 1 of \cite{MFK1994}) that hold true.

If $X$ is an affine variety with the action of a reductive group $G$ and $\chi: G \rightarrow \C^{\times}$ is a character, then we can define a linearization on the trivial line bundle $X \times \C \rightarrow X$ by 
\[
g \cdot (x,z) = (g \cdot x, \chi(g)z).
\]
 One can check that this linearization gives us the Proj GIT quotient construction.
\end{rmk}

\subsection{Quotients and coarse moduli spaces}

We conclude this section by outlining the relationship between quotients and coarse moduli spaces.  Recall from Section 2 that if we are interested in finding a variety that classifies objects of a set $\cA$ up to an equivalence relation $\sim$, we begin by upgrading $\sim$ to an equivalence relation on the set of families $\cA_T$ for each algebraic variety $T$.  This allows us to define a moduli functor $\cM: \opcat{\Var} \rightarrow \Set$ by $\cM(T) = \cA_T/\sim_T$ on the objects and $\cM(f)(\cF) = f^*\cF$ on morphisms $f: S \rightarrow T$.

Now, fix an algebraic variety $X$, and let $\cF \in \cA_X$.  We say $\cF$ has the \emph{local universal property} if for any family $\cG \in \cA_T$ and point $t \in T$ there exists an open neighborhood $ t \in U \subseteq T$ and a morphism $f: U \rightarrow X$ such that $\cG|_{U} \sim_U f^*\cF$.  We will make use of the following lemma.

\begin{lmm}
\label{natinvlemma}
Let $\cM$ be a moduli functor, and let $\cF \in \cA_X$ be a family with the local universal property over an algebraic variety $X$.  Suppose there is an algebraic group $G$ acting on $X$ such that $x_1, x_2 \in X$ belong to the same $G$-orbit if and only if there is an equivalence of the fibers $\cF_{x_1} \sim \cF_{x_2}$.  

For any algebraic variety $Y$ there is a bijection between the set of natural transformations $\Nat(\cM, h_Y)$ and the set of $G$-invariant morphisms $f: X \rightarrow Y$.
\end{lmm}
\begin{proof}
The mapping $\psi(\eta) = \eta_X(\cF)$ sends $\eta \in \Nat(\cM, h_Y)$ to a morphism $\eta_X(\cF): X \rightarrow Y$.  Since $\eta$ is a natural transformation, we have 
\[
\eta_X(\cF)(x) = \eta_{\Spec \C}(\cF_x) = \eta_{\Spec \C}(\cF_{g \cdot x}) = \eta_X(\cF)(g \cdot x),
\]
so $\eta_X(\cF)$ is $G$-invariant.  

We construct the inverse of $\psi$ by gluing.  Let $T$ be an algebraic variety and $\cG \in \cA_t$ a family over $T$.  Let $f: X \rightarrow Y$ be a $G$-invariant morphism.  Since $\cF$ has the local universal property, then there exists a cover $\{U_i \rightarrow T\}$ and a collection of morphisms $h_i: U_i \rightarrow X$ such that $\cG|_{U_i} \sim_{U_i} h_i^* \cF$.  

The compositions $f \circ h_i$ glue to define a morphism $T \rightarrow Y$.  Indeed, if $u \in U_i \cap U_j$, we have
\[
\cF_{h_i(u)} \sim (h_i^* \cF)_u \sim \cG_u \sim (h_j^* \cF)_u \cF_{h_j(u)}.
\]
This implies $h_i(u)$ and $h_j(u)$ belong to the same $G$-orbit, so $f \circ h_i(u) = f \circ h_j(u)$.  This means the $f \circ h_i(u) $ agree on intersections, so they glue.  Note that a similar argument shows that the resulting morphism does not depend on the choice of open cover and $h_i$. 

Define $\nu: \cM \rightarrow h_Y$ by letting $\nu_T(\cG)$ be result of gluing the $f \circ h_i$.  Let $\varphi: S \rightarrow T$ be a morphism. Let $\{U_i \rightarrow T\}$ and $h_i$ be as above.  We have that $V_i = \inv{\varphi}(U_i)$ form an open cover of $S$.  Moreover, we have
\[
(h_i \circ \varphi)^*\cF|_{V_i} \sim_{V_i} (\varphi^* \circ h_i^*\cF)|_{V_i} \sim_{\V_i} \varphi^* \cG|_{V_i}.
\]
Therefore $\nu_T(\cG) \circ \varphi|_{V_i} = (f \circ h_i) \circ \varphi = \nu_S(\varphi^* \cG)|_{V_i}$.  Since the $V_i$ cover $S$, then $\nu_T(\cG) \circ \varphi = \nu_S(\varphi^* \cG)$, so $\nu$ is a natural transformation.

The mapping $\overline{\psi}(f) = \nu$ is the inverse of $\psi$.  Indeed, $\psi \circ \overline{\psi} = \nu_X(\cF)$.  By the above construction, $\nu_X(\cF) = f \circ \Id_X = f$.  Conversely, $\overline{\psi} \circ \psi(\eta) = \nu$ such that $\nu_T(\cG)|_{U_i} = \eta_X(\cF) \circ h_i$, where $U_i$ and $h_i$ are as above.  Therefore, we get
\[
\nu_T(\cG)|_{U_i} = \eta_X(\cF) \circ h_i = \eta_{U_i}(h_i^*\cF) = \eta_T(\cG)|_{U_i}
\]
by functoriality.  Thus $\nu_T(\cG) = \eta_T(\cG)$ and we are done.
\end{proof}

This lemma gives us the following result.
\begin{thm}
\label{coarsemodulithm}
Under the assumptions of Lemma \ref{natinvlemma} we have that \begin{enumerate}
\item If the moduli functor $\cM$ admits a coarse moduli space $(Y, \eta: \cM \rightarrow h_Y)$, then there exists a morphism $\phi: X \rightarrow Y$ such that $(Y, \phi)$ is a categorical quotient.
\item If $(Y, \phi)$ is a categorical quotient of the action of $G$ on $X$, then there exists a natural transformation $\eta: \cM \rightarrow h_Y$ that makes $Y$ into a coarse moduli space for $\cM$ if and only if $(Y, \phi)$ is an orbit space.
\end{enumerate} 
\end{thm}
\begin{proof}
If $(Y, \eta: \cM \rightarrow h_Y)$ is a coarse moduli space, the existence of $\phi$ follows immediately from Lemma \ref{natinvlemma}, and its universal property is a consequence of the universal property of $\eta$.

If $(Y, \phi)$ is a categorical quotient, then by Lemma \ref{natinvlemma} $\phi = \eta_X(\cF)$ for some natural transformation $\eta: \cM \rightarrow h_Y$.  Note that by the local universal property every family in $\cM(\Spec \C)$ is equivalent to a fiber of $\cF$.  Moreover, two fibers of $\cF$ are not equivalent if and only if they are over points from different orbits.  Thus the mapping $\eta_{\Spec \C}: \cM(\Spec \C) \rightarrow \Hom(\Spec \C, Y)$ is a bijection if an only if $(Y, \phi)$ is an orbit space. 
\end{proof}

\begin{exa}
Consider the moduli functor $\cM$ classifying families of vector spaces of dimension $n$ up to isomorphism (seen in Example \ref{vspacenonrep}). 

Let $X = \GL(n, \C)$ be viewed as the space parametrizing choices of basis for the vector space $\C^n$.  The group $\GL(n, \C)$ acts on $X$ by left multiplication.  It is easy to check that the pair $(\Spec \C, \phi)$, where $\phi: X \rightarrow \Spec \C$ is the constant morphism, is an orbit space for is action.
  
	Let $F = X \times \C^n$ be the trivial family over $X$.  Let $E$ be a rank $n$ vector bundle over an algebraic variety $T$. Since $E$ is trivializable, then each point $t \in T$ has an open neighborhood $t \in U \subset T$ such that $U \times \C^n \sim_U h^*F$, where $h: U \rightarrow X$ is any constant morphism.  This means, $F$ has the local universal property with respect to $\cM$.  Since the action of $\GL(n,\C)$ on $X$ only has one orbit, it is clear that two fibers of $F$ are equivalent if and only if the corresponding points of $X$ lie in the same orbit.
		
Thus, the coarse moduli space corresponding to $\cM$ is the pair $(\Spec \C, \eta)$, where $\eta: \cM \rightarrow h_{\Spec \C}$ is a natural transformation such that $\eta_T$ is the constant map.  We have already seen in Example \ref{coarsemodex} that this is the case. 
 
\end{exa}

\section{Moduli spaces of quiver representations}
\subsection{Quiver basics}

We begin our discussion of moduli space constructions for quiver representations by introducing basic notation and terminology.  Recall from the introduction that a \emph{quiver} $Q = (I_Q, A_Q, s, t)$ is a multidigraph. That is, $Q$ is a directed graph where multiple edges between two vertices and loops are allowed.  The data defining $Q$ consists of a set of vertices $I_Q$, a set of arrows $A_Q$, and two mappings $s,t: A_Q \rightarrow I_Q$ defined such that the arrow $a$ starts at $s(a)$ and ends at $t(a)$.  A quiver is said to be \emph{finite} if $I_Q$ and $A_Q$ are finite.  From now on, we will assume all quivers we work with are finite, though much of this section remains valid even if this is not the case. 

\begin{defn}
A \textbf{representation} $V$ of a quiver $Q$ consists of a collection of $\C$-vector spaces $\{V_i\}_{i \in I_Q}$ and a collection of linear transformations $\{f_a\}_{a \in A_Q}$ where $f_a: V_{s(a)} \rightarrow V_{t(a)}$.  

A representation $V$ of $Q$ is said to be \textbf{finite-dimensional} if all of the $V_i$ are finite-dimensional. In this case the \textbf{dimension vector} $\alpha = (\alpha_i)_{i \in I_Q}$ of $V$ is defined by $\alpha_i = \dim V_i$. 
\end{defn}

  We will often consider representations $V$ in coordinate spaces, meaning that $V_i = \C^{\alpha_i}$.  This makes each representation $V = (\{V_i\}, \{f_a\})$ with dimension vector $\alpha = (\alpha_i)$ an element of the vector space $\prod_{a \in A_Q}\Hom(\C^{\alpha_{s(a)}}, \C^{\alpha_{t(a)}})$.

\begin{defn}
Let $V = (\{V_i\}, \{f_a\})$ and $W = (\{W_i\}, \{h_a\})$ be two representations of a quiver $Q$.  A \textbf{morphism} $\varphi: V \rightarrow W$ of quiver representations is a collection of linear transformations $\varphi_i: V_i \rightarrow W_i$ such that $\varphi_{t(a)} \circ f_a = h_a\circ \varphi_{s(a)}$ for all $a \in A_Q$.
\end{defn}

We can generalize many notions from linear algebra to the context of quiver representations.  For example, a \emph{subrepresentation} $W \subset V$ of a representation of $Q$ consists of subspaces $W_i \subset V_i$ and arrows $h_a: W_{s(a)} \rightarrow W_{t(a)}$ such that $f_a|_{W_{s(a)}} = h_a$.  One can similarly extend the definitions of quotient representation, as well as of the kernel, the cokernel, the image of a morphism, etc.  Note that the representations of a quiver $Q$ form a category $\cR(Q)$, and it is not hard to see that $\cR(Q)$ is a $\C$-linear abelian category.

We will denote by $R(Q, \alpha)$ the set of representations of $Q$ with dimension vector $\alpha$.  If we are considering representations of $Q$ with dimension vector $\alpha$ in coordinate spaces, then $R(Q, \alpha)$ has the structure of an affine space, and we denote it by $\Rep(Q, \alpha)$.  The group $G(\alpha):= \prod_{i \in I_Q} \GL(V_i, \C)$ acts on $\Rep(Q, \alpha)$ by change of basis at each vertex.  That is, if $V = (f_a) \in \Rep(Q, \alpha)$ and $g = (g_i) \in G(\alpha)$, then $g \cdot V = (g_{t(a)}f_a\inv{g_{s(s)}})$.  Note that two representations in $\Rep(Q, \alpha)$ are isomorphic if and only if they lie on the same $G(\alpha)$-orbit.

\begin{exa}
\leavevmode
We have already seen several examples of quiver representations.

\begin{enumerate} 
\item
The Jordan quiver $J$ consists of a single vertex and single loop.  Its representations may be interpreted as pairs $(V, f_a)$, where $V$ is a vector space and $f_a \in \End(V)$ is its endomorphism.  
We saw in Example \ref{jorquivex} that there is no moduli space parametrizing the finite-dimensional representations of $J$ up to isomorphism.
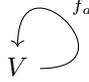
\[
\begin{tikzcd}
V \arrow[out=0,in=90,loop, swap, "f_a"]
\end{tikzcd}
\]

\item The Kronecker quiver $K_2$ consists of a pair of vertices and two arrows in the same direction.  In Section 4 we saw that an open subset of $\Rep(K_2, \alpha)$ parametrizes vector bundles on $\P^1$ with some additional rigidity structure.

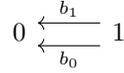
\[
\begin{tikzcd}
0 & 1 \arrow[l, shift left=1ex, "b_0"] \arrow[l, shift right=1ex, swap, "b_1"]
\end{tikzcd}
\]

\item The $A_n$ quiver a quiver associated to the Dynkin graph $A_n$ consisting of $n$ vertices and $n-1$ arrows.  
\[
\begin{tikzcd}
1 & 2 \arrow[l, "a_1"] & \cdots \arrow[l, "a_2"] & n \arrow[l, "a_{n-1}"]
\end{tikzcd}
\]
\end{enumerate}
\end{exa}

Note that representations of $A_1$ are just $\C$-vector spaces.  In representation theory it is often convenient to view representations of objects as modules over an algebra.  We can obtain this kind of description for quiver representations as follows.

Given a quiver $Q$, a \emph{path} $p = a_n \cdots a_2a_1$ in $Q$ is a sequence of arrows $a_n, \dots, a_2, a_1$ such that $s(a_i) = t(a_{i+1})$ for $1 \le i \le n-1$.  We will allow paths that contain no arrows.  Namely, for each vertex $i \in I_Q$ there is a \emph{trivial path} denoted by $p_i$ such that $s(p_i) = t(p_i) = i$.  We will denote $s(p) := s(a_n)$ and $t(p) := t(a_1)$. If $s(p) = t(p)$, then the path is called a \emph{cycle}.   

The \emph{path algebra} $\C Q$ of $Q$ is the associative $\C$-algebra generated as a vector space by $e_p$ for all paths $p$ in $Q$.  Multiplication is defined by concatenation of paths.  Namely, 
\[
e_p e_q = \begin{cases}
e_{pq} & \textrm{if } s(p) = t(q), \\
0 & \textrm{otherwise}.
 \end{cases} 
\] 
The elements $e_i := e_{p_i}$ are idempotent.  In fact, they satisfy $e_p e_i = e_p$ for any $p$ such that $s(p) = i$ and $e_i e_p = e_p$ for any $p$ such that $t(p) = i$.  Note that if $Q$ is finite, then $1_Q := \sum_{i \in I_Q} e_i$ is a unit element in the algebra $\C Q$.

If $V$ is a representation of the quiver $Q$, then any path $p = a_n \cdots a_2a_1$ in $Q$ defines a linear transformation $f_p: V_{s(p)} \rightarrow V_{t(p)}$ given by $f_{a_n} \circ \cdots \circ f_{a_1}$. If $V$ is a representation of $Q$, then we can define a $\C Q$-module structure on $\bigoplus_{i \in I_Q} V_{i}$ by 
\[
e_pv  = \begin{cases} f_p(v) & \textrm{if } v \in V_{s(p)}.\\
0 & \textrm{otherwise}.
\end{cases}
\] 
Conversely, if $V$ is a $\C Q$-module, then let $V_i = e_iV$.  For each $a \in A_Q$, let $f_a: V_{s(a)} \rightarrow V_{t(a)}$ be given by $f_a(v) = e_av$.  This morphism is well-defined since
\[
f_a(V_{s(a)}) = e_aV_{s(a)} = e_ae_{s(a)}V = e_{t(a)}e_aV \subset e_{t(a)}V = V_{t(a)}.  
\]
Thus, we see that $(\{V_i\}, \{f_a\})$ is a representation of $Q$.

It is not hard to extend the above constructions to morphisms and obtain the following result (see Section 1.2 of \cite{Bri2012} for details).
\begin{thm}
For a finite quiver $Q$, the category $\cR(Q)$ is equivalent to the category $\C Q-\modl$ of modules over the path algebra of $Q$.
\end{thm}

\begin{exa}
\leavevmode
\begin{enumerate}
\item Since the only nontrivial paths in the Jordan quiver $J$ consists of concatenations of the loop with itself, the path algebra $\C J$ is just the polynomial algebra $\C[z]$.  
\item The path algebra of the Kronecker quiver $K$ may be realized as a subalgebra of $\Hom(\C^2,\C^2)[z]$
that is generated by
\[
\begin{pmatrix}1 & 0\\ 0 & 0\end{pmatrix}, \begin{pmatrix}0 & 0\\ 0 & 1\end{pmatrix}, \begin{pmatrix}0 & 1\\ 0 & 0\end{pmatrix}, \begin{pmatrix}0 & z\\ 0 & 0\end{pmatrix},
\]
these correspond to the elements $e_1$, $e_2$, $e_{b_1}$, and $e_{b_2}$ respectively.
\end{enumerate}
\end{exa}

Let us recall several definitions from algebra in the context of quiver representations.
\begin{defn}
Let $V$ be a representation of the quiver $Q$.  
\begin{itemize}
\item We say $V$ is \textbf{simple} if it contains no nontrivial subrepresentations.  
\item We say $V$ is \textbf{semisimple} if $V  = \oplus_{i = 1}^m V^i$ where $V^i$ are simple.  
\item We say $V$ is \textbf{indecomposable} if it cannot be written as the direct sum $V = W^1 \oplus W^2$ such that $W^1$ and $W^2$ are both nontrivial subrepresentations.
\end{itemize}
\end{defn}

Note that the correspondence between representations of $Q$ and $\C Q$-modules preserves subobjects.  Therefore, we can import the following basic results from algebra.

\begin{thm}[Jordan-H\"{o}lder filtration]
\label{jhfilt}
Let $V$ be a finite-dimensional representation of a quiver $Q$.  There exists a filtration of $V$ by subrepresentations
\[
0 = V_0 \subset V_1 \subset \cdots \subset V_{n-1} \subset V_n = V,
\]
such that $V_{i}/V_{i-1}$ is simple for $1 \le i \le n$.  Moreover, the quotients 
$V_{i}/V_{i-1}$ are unique up to reordering.
\end{thm}

\begin{lmm}[Schur's Lemma]
\label{schur}
If $V$ and $W$ are simple representations of the quiver $Q$, then all morphisms $V \rightarrow W$ are either isomorphisms or equal to zero.  All endomorphisms of $V$ are equal to $c\Id$, where $c \in \C$ and $\Id$ is the identity endomorphism.
\end{lmm}

\subsection{Affine quotients for quiver representations}
We would like to describe affine GIT quotients for the action of the reductive group $G(\alpha)$ on $\Rep(Q, \alpha)$.  This can be done at the level of points by determining the closed orbits under this action, or at the level of functions by determining the invariants.  Let us begin with a few examples:

\begin{exa}
\leavevmode
\begin{enumerate}
\item  Consider the group action $\GL(1,\C) \times \GL(1, \C)$ on the product of affine spaces $\A^{1} \times \A^{1}  = \Hom(\C, \C)\times \Hom(\C, \C)$ by $(t,s) \cdot (x,y) = (tx\inv{s}, sy\inv{t})$.  Notice that $(tx\inv{s}, sy\inv{t}) = (t\inv{s}x, s\inv{t}y)$, so the orbits and the affine GIT quotient are exactly the same as in Example \ref{affinegitex} (2).

\item Let us look at an upgraded version of the previous example.  Consider the quiver $\overline{A_2}$ consisting of the $A_2$ quiver together with one additional arrow in the opposite direction.
\[
\begin{tikzcd}
1 \arrow[r, shift left=0.7ex, ""] & 2 \arrow[l, shift left=0.7ex, ""]
\end{tikzcd}
\]
  Let $\alpha = (2,1)$.  The group $G(\alpha) = \GL(2,\C) \times \GL(1, \C)$ acts on the variety $\Rep(\overline{A_2}, \alpha) := \A^{2} \times \A^{2}  = \Hom(\C^2, \C)\times \Hom(\C, \C^2)$ by 
\[
(A,t) \cdot \left(\begin{pmatrix}x_1 & y_1 \end{pmatrix}, \begin{pmatrix}x_2\\ y_2 \end{pmatrix} \right) = \left(t\begin{pmatrix}x_1 & y_1 \end{pmatrix} \inv{A}, A\begin{pmatrix}x_2 \\ y_2 \end{pmatrix}\inv{t}\right),
\]
where $A = \begin{pmatrix}a_1 & a_2\\ a_3 & a_4 \end{pmatrix} \in \GL(2,\C)$.  Note that the product $\begin{pmatrix}x_1 & y_1 \end{pmatrix}\begin{pmatrix}x_2\\ y_2 \end{pmatrix} = x_1x_2 + y_1y_2$ is constant on the orbits of $G(\alpha)$.  Furthermore, all products $X := \begin{pmatrix}x_2\\ y_2 \end{pmatrix}\begin{pmatrix}x_1 & y_1 \end{pmatrix} = \begin{pmatrix}x_1x_2 & x_2y_1\\ x_1y_2 & y_1y_2\end{pmatrix}$ coming from elements in the same $G(\alpha)$-orbit lie in the same $\GL(2, \C)$ conjugacy class. 

Consider all elements of $\Rep(\overline{A_2}, \alpha)$ such that $x_1x_2 + y_1y_2 = c$ for a fixed $c \in \C$.  Note that the eigenvalues of $X$ are $c$ and $0$.  If $c \neq 0$, this means $X$ is diagonalizable by some $A \in \GL(2, \C)$.  Let $\begin{pmatrix}z_1 & w_1 \end{pmatrix} = \begin{pmatrix}x_1 & y_1 \end{pmatrix}\inv{A}$ and $\begin{pmatrix}z_2 \\ w_2 \end{pmatrix} = A\begin{pmatrix}x_2\\ y_2 \end{pmatrix}$.  We have$\begin{pmatrix}z_1z_2 & z_2w_1\\ z_1w_2 & w_1w_2\end{pmatrix} = \begin{pmatrix}c & 0 \\0 & 0\end{pmatrix}$.  Therefore, $w_1 = w_2 = 0$ and $z_1z_2 = c$.  Choosing $t  = c/z_1$, we get that each element of $X$ such that $x_1x_2 + y_1y_2 = c \neq 0$ is contained in the $G(\alpha)$-orbit of $\left(\begin{pmatrix}c & 0 \end{pmatrix}, \begin{pmatrix}1\\ 0 \end{pmatrix} \right)$  Thus, the equation $x_1x_2 + y_1y_2 = c \neq 0$ defines a unique closed orbit of $G(\alpha)$.

In the case that $c = 0$, a similar computation shows that each orbit of $G(\alpha)$ satisfying $x_1x_2 + y_1y_2 = 0$ is an orbit of an element of one of the following forms 
\begin{align*}
& \left(\begin{pmatrix}0 & 0 \end{pmatrix}, \begin{pmatrix}x_2\\ y_2 \end{pmatrix} \right), \left(\begin{pmatrix}x_1 & y_1 \end{pmatrix}, \begin{pmatrix}0\\ 0 \end{pmatrix} \right),\\ & \left(\begin{pmatrix}0 & y_1 \end{pmatrix}, \begin{pmatrix}x_2\\ 0 \end{pmatrix} \right), \left(\begin{pmatrix}0 & 0 \end{pmatrix}, \begin{pmatrix}0\\ 0 \end{pmatrix} \right).
\end{align*}  
Only the last of these defines a closed $G(\alpha)$-orbit (consisting of a single point) as all of the others contain it in their closures.

Thus, we see that every $G(\alpha)$-invariant regular function is defined by its values on $\left(\begin{pmatrix}c & 0 \end{pmatrix}, \begin{pmatrix}1\\ 0 \end{pmatrix} \right)$ and $\left(\begin{pmatrix}0 & 0 \end{pmatrix}, \begin{pmatrix}0\\ 0 \end{pmatrix} \right)$.  In other words, we have 
\[
\cO(\Rep(\overline{A_2}, \alpha))^{G(\alpha)} = \C[x_1x_2 + y_1y_2],
\]
and consequently the affine GIT quotient is $(\A^1, \phi)$, where $\phi: \Rep(\overline{A_2}, \alpha) \rightarrow \A^1$ is given by $\phi\left(\begin{pmatrix}x_1 & y_1 \end{pmatrix}, \begin{pmatrix}x_2\\ y_2 \end{pmatrix} \right) = x_1x_2 + y_1y_2$.

\item Let us consider the action of the automorphism group on the representations of the $A_2$ quiver in standard coordinate spaces.  That is, we consider the action of $\GL(n,\C) \times \GL(m, \C)$ on the affine space $\A^{mn} = \Hom(\C^m, \C^n)$ of $n \times m$ matrices defined by $(A,B) \cdot C = AC\inv{B}$.

Note that this action is transitive when restricted to matrices of rank $r$, where $0 \le r \le \min(m,n)$.  Since rank $r$ matrices are those with at least one nonvanishing $r \times r$ minor, then the subset $M_{\le r} \subset \Hom(\C^m, \C^n)$ of matrices of rank less than or equal to $r$ is a closed subvariety defined by vanishing $r \times r$ minors.  Similarly, the set of matrices $M_{>r}$ of rank greater than $r$ is open.  It follows that the orbits of the $\GL(n,\C) \times \GL(m, \C)$ are all locally closed (matrices of maximal rank form an open subset), with the only closed orbit being that of the zero matrix.  Thus there are no stable points for this action.

The affine GIT quotient in this case is $\Spec \C$ since there are no nonconstant $G$-invariant functions (this is already true just for left multiplication by elements of $\GL(n, \C)$).  
\end{enumerate}
\end{exa}

The following theorem (proved in \cite{LeBPro1990}) allows us to compute invariants for the standard group action on quiver representations in coordinate spaces.

\begin{thm}[Le Bryun and Procesi]
\label{lbp}
Let $Q$ be a quiver.  The algebra of invariants $\cO(\Rep(Q, \alpha))^{G(\alpha)}$ is generated by the functions $\tr(f_p)$, where $p$ is a cycle in $Q$.
\end{thm}

If a quiver contains no cycles, we immediately obtain.
\begin{cor}
Let $Q$ be a quiver with no cycles.  We have $\cO(\Rep(Q, \alpha))^{G(\alpha)} = \C$, and consequently $\Rep(Q,\alpha)//G(\alpha) = \Spec \C$.
\end{cor}

Note that any cycle in $Q$ can be concatenated with itself any number of times to obtain additional cycles.  This means, the above theorem provides an infinite number of generators, even though $\cO(\Rep(Q, \alpha))^{G(\alpha)}$ is finitely generated.  Therefore, these generators must be subject to additional relations.  

This can be seen directly as follows.  Let $p = a_m \cdots a_1$ be a cycle in $Q$ such that $a_k \neq a_l$ for $k \neq l$, and let $i$ be a vertex in this cycle such that $\alpha_i$ is minimal.  This cycle defines a morphism of affine spaces $\Rep(Q, \alpha) \rightarrow \A^{\alpha_i^2}$ given by $V \mapsto f_p$.  Since $\alpha_i$ is minimal, the morphism is surjective. Therefore $\cO(\A^{\alpha_i^2})$ injects into $\cO(\Rep(Q, \alpha))$.  The action of $G(\alpha)$ induced on $\cO(\Rep(Q, \alpha))$ is just conjugation by $\GL(\alpha_i,\C)$ when restricted to $\cO(\A^{\alpha_i^2})$.  It follows that, the $G(\alpha)$-invariants $\tr(f_p^l)$ are contained in $\cO(\A^{\alpha_{i}^2})^{\GL(\alpha_i,\C)}$ for $l \ge 0$.  As seen in Example \ref{affinegitex}, the latter is just the ring of symmetric polynomials $\C[\lambda_1, \cdots, \lambda_{\alpha_i}]^{S_{\alpha_i}}$.  Under this identification, $\tr(f_p^l)$ corresponds to the power sum symmetric polynomial $p_l := \lambda_1^l + \cdots \lambda_{\alpha_i}^l$.  It is well-known that $p_1, \dots, p_{\alpha_i}$ generate the ring of symmetric polynomials, so we need only consider the traces of $f_p, \dots, f_p^{\alpha_i}$.

\begin{exa}
Consider $\overline{A_2}$, the double of the $A_2$ quiver, consisting of two vertices and two mutually inverse arrows.

\[
\begin{tikzcd}
1 \arrow[r, shift left=0.7ex, "a_1^*"] & 2 \arrow[l, shift left=0.7ex, "a_1"]
\end{tikzcd}
\]

The group $GL(\alpha) = \GL(\alpha_1, \C) \times \GL(\alpha_2, \C)$ acts on $\Rep(\overline{A_2}, \alpha)$.  By Theorem \ref{lbp} and the subsequent discussion, the algebra of invariants $\cO(\Rep(Q, \alpha))^{G(\alpha)}$ is $\C[\tr(f_p), \dots, \tr(f_p^{n})]$, where $p = a_1^*a_1$ and $n = \min\{\alpha_1, \alpha_2\}$.  This makes the affine GIT quotient $\Rep(\overline{A_2}, \alpha)//G(\alpha) = \A^{n}$.
\end{exa}

Recall that affine GIT quotients parametrize closed orbits under the group action.  Thus, in order to better understand $\Rep(\overline{A_2}, \alpha)//G(\alpha)$, we would like a nice description of the representations in $\Rep(\overline{A_2}, \alpha)$ that give rise to closed orbits.  We begin by describing orbit closures in terms of actions of one-parameter subgroups.  

\begin{defn}
Let $G$ be an algebraic group.  A \textbf{one-parameter subgroup} (also called a \textbf{cocharacter}) of $G$ is a morphism of algebraic groups $\lambda: \C^{\times} \rightarrow G$.
\end{defn}

This notion is dual to that of a character $\chi: G \rightarrow \C^{\times}$, which we have already seen in the previous section.  Note that both the set of characters $\Hom(G, \C^{\times})$ and the set of cocharacters $\Hom(\C^{\times}, G)$ of $G$ have natural group structures.  In the case that $G = \C^{\times}$, these groups coincide, and can be computed using the following result from the theory of algebraic groups.

\begin{prp}
\label{charcomp}
If $\chi(t) \in \Hom(\C^{\times}, \C^{\times})$, then $\chi(t) = t^n$ for some $n \in \Z$.  Furthermore, mapping $\Hom(\C^{\times}, \C^{\times}) \rightarrow \Z$ given by $\chi \mapsto n$ is an isomorphism of groups. 
\end{prp}
\begin{proof}
Let $m: \C^{\times} \times \C^{\times} \rightarrow \C^{\times}$ be the multiplication operation on $\C^{\times}$.  Note that this morphism of varieties induces a homomorphism of $\C$-algebras defined by
\begin{align*}
& m^*: \C[t,\inv{t}] \rightarrow \C[t,\inv{t}] \otimes_\C \C[t,\inv{t}]\\
& m^*(t)  = t \otimes t.
\end{align*}
Since $\chi(t)$ is a homomorphism of algebraic groups, we have that $\chi(m(t_1, t_2)) = m(\chi(t_1), \chi(t_2))$ for $t_1, t_2 \in \C^{\times}$.  On the regular functions this gives us $m^*\chi^*(t) = \chi^* \otimes \chi^* m^*(t)$, where $\chi^*: \C[t,\inv{t}] \rightarrow \C[t,\inv{t}]$ is induced by $\chi$.  Let $\chi^*(t) = a_{-m}t^{-m} + \cdots + a_Mt^M = \sum_k a_k t^k$.  We have
\[
m^*\chi^*(t) = \sum_k a_k t^k\otimes t^k = \sum_{k,l} a_ka_l t^k \otimes t^l = \sum_k a_k t^k \otimes \sum_k a_k t^k  = \chi^* \otimes \chi^* m^*(t).
\]
By comparing the two summations, we see that $\chi^*(t) = a_kt^k$ for some $k \in \Z$.  Let $e^*: \C[t,\inv{t}] \rightarrow \C$  be the algebra homomorphism corresponding to the group identity element $e: \Spec \C \rightarrow \C^{\times}$ such that $e(\Spec \C) = 1$.  Note that $e^*(p(t)) = p(1)$.  We have $\chi(e(\Spec \C)) = e(\Spec \C)$, which implies $e^*\chi^*(t) = e^*(t)$.  Since $e^*(t) = 1$, we get $a_k = 1$.  Thus every element $\chi \in \Hom(\C^{\times}, \C^{\times})$ has the form $\chi(t)  = t^n$ for some $n \in \Z$.

Since every $n \in \Z$ defines an element $t \mapsto t^n$ in $\Hom(\C^{\times}, \C^{\times})$, it is easy to check that $\chi \mapsto n$ for $\chi(t) = t^n$ is a group isomorphism $\Hom(\C^{\times}, \C^{\times}) \rightarrow \Z$.
\end{proof}  

Recall that an affine algebraic group $T$ is an algebraic torus if it is isomorphic to $(\C^{\times})^m$ for some $m > 0$.  The above proposition lets us immediately compute the characters and cocharacters of $T$.

\begin{cor}
\label{toruscharcomp}
Let $T \cong (\C^{\times})^m$ be an algebraic torus.  If $\chi(t) \in \Hom(T, \C^{\times})$, then $\chi(t) = t^{n_1} \cdots t^{n_m}$ for some $n_1, \dots, n_m \in \Z$.  If $\lambda(t) \in \Hom(\C^{\times}, T)$, then $\lambda(t) = (t^{l_1}, \dots, t^{l_m})$ for some $l_1, \dots, l_m \in \Z$.   
\end{cor}

The following proposition is a well-known result from the theory of algebraic groups (see e.g. Theorem 3.2.3 in \cite{Spr1998}) that we will be using later.

\begin{thm}
\label{torusweight}
Let $T$ be an algebraic torus, and let $\rho: T \rightarrow \GL(n,\C)$ be a morphism of algebraic groups (i.e. a linear representation of $T$).  There exists a direct sum decomposition 
\[
\C^n = \bigoplus_{\lambda} V_{\lambda},
\]
where $\lambda$ is a cocharacter of $T$ and $V_{\lambda} = \{v \in \C^n|t \cdot v = \lambda(t)v \textrm{ for all } t \in T \}$.
\end{thm}

If an algebraic group $G$ acting on an algebraic variety $X$, and $\lambda(t)$ is a one-parameter subgroup of $G$, then the action of $G$ induces an action of $\lambda(t)$ on $X$.  Note that there is an injective morphism of varieties $\C^{\times} \hookrightarrow \C = \A^1$.  For any $x \in X$, if the morphism $\lambda(t) \cdot x: \C^{\times} \rightarrow X$ induced by $\lambda$ extends to a morphism $\A^1 \rightarrow X$ we say that $\lim_{t \to 0} \lambda(t) \cdot x$ exists and is equal to the image of $0$ under this morphism. 

One useful consequence of the existence of such a limit is that $\lim_{t \to 0} \lambda(t) \cdot x$ is a limit point of $G \cdot x$, so it is contained in $\overline{G \cdot x}$.  The converse statement is also true as may be seen in the following theorem from \cite{Kem1978}. 
 
\begin{thm}
\label{fundamentalgit}
Let $X$ be an affine variety with the action of a reductive group $G$, and let $x \in X$.  Let $S$ be a closed $G$-invariant subset of $X$ such that $S \cap \overline{G \cdot x} \neq \varnothing$.  There exists a one-parameter subgroup $\lambda$ of $G$ such that $\lim_{t \to 0} \lambda(t) \cdot x$ exists and is in $S$.
\end{thm}

We will need the next two lemmas in our description of closed $G(\alpha)$-orbits.

\begin{lmm}
\label{filtlmm1}
Let $0 = V_0 \subset \cdots \subset V_{n} = V$ be a filtration of $V \in \Rep(Q, \alpha)$ by subrepresentations.  The associated graded representation $\gr(V_\bullet) := (V_1/V_0)\oplus \cdots \oplus (V_{n}/V_{n-1})$ lies in the orbit closure $\overline{G(\alpha) \cdot V}$.
\end{lmm}
\begin{proof}
Denote by $V_{ki}$ the vector spaces in the subrepresentation $V_{k}$, for each vertex $i \in I_Q$.  For $i \in I_Q$, choose a complementary subspace $V_{k+1i}'$ to the subspace $V_{ki}$ in $V_{k+1i}$ where $0 \le k \le n-1$.  Consider the one-parameter group $\lambda(t): \C^\times \rightarrow G(\alpha)$ where $\lambda(t) = (\lambda_i(t))_{i \in I_Q}$ is such that $\lambda_i(t)$ is the direct sum of scalar matrices that act by $t^{-k}$ on $V_{ki}'$.

Note that $V_{i} = V_{1i}'\oplus V_{2i}' \oplus \cdots \oplus V_{ni}'$.  Therefore, each of the matrices $f_a$ may be written as
\[
f_a = \begin{pmatrix} f_a^{11} & f_a^{12} & \cdots & f_a^{1n}\\
0 & f_a^{22} & \cdots & f_a^{2n}\\
\vdots & \vdots & \ddots & \vdots\\
0 & 0 & \cdots & f_a^{nn} 
 \end{pmatrix},
\]
where $f_a^{lm}: V_{ms(a)}' \rightarrow V_{lt(a)}'$ are obtained by restricting $f_a$ to the direct summands.  The action of $\lambda(t)$ restricted to $f_a^{lm}$ is just multiplication by $t^{m-l}$.  It follows that 
\[
\lim_{t \to 0} \lambda(t) \cdot f_a = \lim_{t \to 0} \begin{pmatrix} f_a^{11} & tf_a^{12} & \cdots & t^{n-1}f_a^{1n}\\
0 & f_a^{22} & \cdots & t^{n-2}f_a^{2n}\\
\vdots & \vdots & \ddots & \vdots\\
0 & 0 & \cdots & f_a^{nn} 
 \end{pmatrix} = 
\begin{pmatrix} f_a^{11} & 0 & \cdots & 0\\
0 & f_a^{22} & \cdots & 0\\
\vdots & \vdots & \ddots & \vdots\\
0 & 0 & \cdots & f_a^{nn} 
 \end{pmatrix}.
\]
This means $\lim_{t \to 0} V =  V_1'\oplus \cdots \oplus V_n'$, where $V_l' \cong V_{l}/V_{l-1}$.  Thus $\gr(V_\bullet)$ is in the closure of $G(\alpha) \cdot V$.
\end{proof}

\begin{lmm}
\label{filtlmm2}
Let $V, W \in \Rep(Q, \alpha)$ be such that there exists a one-parameter subgroup $\lambda: \C^{\times} \rightarrow G(\alpha)$ with the property that $\lim_{t \to 0} \lambda(t) \cdot V = W$.
There exists a filtration $0 = V_0 \subset \cdots \subset V_n = V$ such that $W$ is isomorphic to $\gr(V_{\bullet})$.
\end{lmm}
\begin{proof}
Let $\lambda_i: \C^{\times} \rightarrow \GL(\alpha_i, \C)$ be the components of $\lambda = (\lambda_i(t))$  corresponding to the vertex $i \in I_Q$.  By Theorem \ref{torusweight}, we have $V_i = \bigoplus_{k \in \Z} V_{ik}'$, where $V_{ik}'$ be the subspace of $V_i$ such that $\lambda_i(t)$ acts on $V_{ik}'$ by multiplication by $t^{-k}$ (see Proposition \ref{charcomp}). Since each $V_i$ is finite-dimensional and $Q$ is finite, we can order all $l$ such that $V_{il}' \neq 0$, appearing across all the direct sum decompositions for the $V_i$, obtaining
$l_1 < l_2 < \cdots < l_n$.

Writing $f_a = \sum_{k,l} f_a^{kl}$, where $f_a^{kl}: V_{s(a)l}' \rightarrow V_{t(a)k}'$ is a linear transformation, we see $\lambda(t)$ acts on the vector space $\Hom(V_{s(a)l}', V_{t(a)k}')$ by $t^{l-k}$.  Therefore, the limit $\lim_{t \to 0} \lambda(t) \cdot V$ exists only if $f_a^{kl} = 0$ for $k > l$.
 
Thus, for a fixed $m$, the $V_{mi} := \bigoplus_{l_1 \le l \le l_m} V_{il}'$ and $f_a^{m} := \sum_{\substack{l_1 \le l \le l_m \\ l_1 \le k \le l}} f_a^{kl}$ together define a subrepresentation $V_m$ of $V$.  Moreover, these subrepresentations form a filtration
\[
0 = V_0 \subset \cdots \subset V_n = V.
\]
By construction, $W = W_1 \oplus \cdots \oplus W_n$, where $W_m$ is a representation consisting of the subspaces $V_{im}'$ and the arrows $f_a^{mm}$.  It is easy to see that each $W_m$ is isomorphic to $V_{m+1}/V_m$, so the result of the lemma follows.
\end{proof}

We are now ready to give a description of the closed orbits of the $G(\alpha)$ action on $\Rep(Q, \alpha)$.

\begin{thm}
\label{closedorbitthm}
Let $V \in \Rep(Q, \alpha)$.  The orbit $G(\alpha) \cdot V$ is closed if and only if $V$ is semisimple.
\end{thm}
\begin{proof}
Suppose $V \in \Rep(Q, \alpha)$ is such that $G(\alpha) \cdot V$ is closed.  By Theorem \ref{jhfilt}, there exists a filtration $0 = V_0 \subset \cdots \subset V_n = V$ such that $V_{i+1}/V_i$ are simple representations.  By Lemma \ref{filtlmm1}, we have that $\gr(V_\bullet) = (V_1/V_0)\oplus \cdots \oplus (V_{n}/V_{n-1})$ is a semisimple representation contained in $G(\alpha) \cdot V$.  Therefore, $V$ is isomorphic to $\gr(V_\bullet)$, so it is semisimple.

Conversely, suppose $V = V_1 \oplus \cdots \oplus V_n$ where the $V_i$ are simple.  The orbit closure $\overline{G(\alpha) \cdot V}$ contains a closed orbit $G(\alpha) \cdot W$, so by Theorem \ref{fundamentalgit} there is a one-parameter subgroup $\lambda(t)$ of $G(\alpha)$ such that $\lim_{t \to 0} \lambda(t) \cdot V = W$.  By Lemma \ref{filtlmm2}, there exists a filtration
$0 = V_0' \subset \cdots \subset V_m' = V$ such that $W$ is isomorphic to $V_1'/V_0' \oplus \cdots \oplus V_m'/V_{m-1}'$.  Note that because $V_i$ is simple we have that $V_i \cap V_j'$ is either equal to the zero representation or to $V_i$.  Consequently, any nonzero $V_j'$ is the direct sum of some of the $V_i$.  Therefore, $V_1'/V_0' \oplus \cdots \oplus V_m'/V_{m-1}'$ is isomorphic to $V_1 \oplus \cdots \oplus V_n$.  It follows that $W$ is isomorphic to $V$, so $V \in G(\alpha) \cdot W$, which means $G(\alpha) \cdot V$ is closed.
\end{proof}

Thus, we see that the points of $\Rep(Q, \alpha)//G(\alpha)$ correspond to orbits of semisimple representations.  We would like to have a similar description for orbits of stable points.  Unfortunately, we immediately run into the problem that there is a one-parameter subgroup $\Delta$ in $G(\alpha)$ consisting of tuples of scalar matrices $(tI_{\alpha_i})$ such that any representation $V \in \Rep(\overline{Q}, \alpha)$ contains $\Delta$ in its stabilizer.  This means no stabilizer can by finite, so there are no stable points.  

To get around this problem we alter our definition of stable points to call $V$ stable if $G \cdot V$ is closed and $G_x/\Delta$ is finite.  This is equivalent to considering the induced action of $PG(\alpha) := G(\alpha)/\Delta$ on $\Rep(Q, \alpha)$.  Using this definition of stability we can state the following corollary to the above theorem.

\begin{cor}
A representation $V \in \Rep(Q, \alpha)$ is stable under the action of $G(\alpha)$ if and only if it is simple.
\end{cor}
\begin{proof}
If $V$ is simple, then $G(\alpha) \cdot V$ is closed by Theorem \ref{closedorbitthm}.  By Lemma \ref{schur} $G(\alpha)_V/\Delta$ is finite.  Therefore, $V$ is stable.  

Conversely, if $V$ is stable, then by Theorem \ref{closedorbitthm} we have that $V = V_1 \oplus \cdots \oplus V_n$, where the $V_i$ are simple.  This means $\dim G(\alpha)_V = \dim \Aut(V) = \dim \End(V) \ge n$.  It follows that $\dim G(\alpha)_V/\Delta \ge n-1$.  Thus, the only way for $G(\alpha)_V/\Delta$ to be finite is if $n = 1$. Consequently, $V$ is simple.  
\end{proof}

\begin{exa}
\leavevmode
\begin{enumerate}
\item The $A_2$ quiver has exactly two simple representations, corresponding to the dimension vectors $(1,0)$ and $(0,1)$.  This means semisimple representations of dimension vector $\alpha = (\alpha_0, \alpha_1)$ are necessarily equal to the zero matrix of size $\alpha_0 \times \alpha_1$.  This makes sense, since $\Rep(A_2,\alpha)//G(\alpha) = \Spec \C$, where the single point corresponds to the orbit of the zero matrix.  Since the only stable points occur for the dimension vectors $(1,0)$ and $(0,1)$, then in each of those two cases the mapping sending the single orbit of the simple element to $\Spec \C$ is the corresponding geometric quotient.
 
\item The two simple representations of $A_2$ are also simple representations of the doubled quiver $\overline{A_2}$.  In addition, any representation with dimension vector $(1,1)$ such that the arrows correspond to invertible matrices is also simple.

This means closed orbits of the $G(\alpha)$-action on $\Rep(\overline{A_2}, \alpha)$ are either a pair of zero matrices (as before) or orbits of pairs
\[
\left( \begin{pmatrix}A & 0\\ 0 & 0\end{pmatrix}, \begin{pmatrix}B & 0\\ 0 & 0\end{pmatrix} \right),
\]
where $A, B$ are invertible diagonal matrices of the same size.  Note that acting by a product of diagonal matrices from $G(\alpha)$ we can guarantee that $A$ is an invertible diagonal matrix and $B$ is the identity matrix.  The maximum possible rank of $A$ is $n = \min\{\alpha_0, \alpha_1\}$, so each closed orbit defines a point in $\Rep(\overline{A_2},\alpha)//G(\alpha) = \A^n$. 

Looking at stable points, it is not hard to check that in addition to the two geometric quotients seen in the $A_2$ case, there is one more in dimension $(1,1)$.  This is obtained by sending the $G(\alpha)$-orbit of the simple representation $\{(a,1)\}$ to $a \in \A^1 - \{0\}$ and $\{(0,0)\}$ to $0$.
\end{enumerate}
\end{exa}

\subsection{Moduli spaces of quiver representations}

We would like to use the Proj GIT quotient $\Rep(Q,\alpha)//_{\chi}G(\alpha)$ (see Section 7) in order to construct moduli spaces of quiver representations.  While it is possible to do so directly by describing the $\chi$-semi-invariants of the $G(\alpha)$ action (see \cite{SVdB2001}) and the corresponding $\chi$-semistable points in $\Rep(Q,\alpha)$, we will instead use an equivalent notion of semistability (introduced by King in \cite{King1994}) that is easier to verify.  

The key component in proving King's semistability criterion is an important result due to Mumford.  Let $G$ be an affine algebraic group.  Define the pairing $\langle - , - \rangle: \Hom(G, \C^{\times}) \times \Hom(\C^{\times}, G) \rightarrow \Z$ by $\langle \chi, \lambda \rangle = n$ where $\chi \circ \lambda(t) = t^n$.  The following theorem is a version of Theorem 2.1 in \cite{MFK1994}.

\begin{thm}[The Hilbert-Mumford criterion]
\label{hmcriterion}
Let $G$ be a reductive group acting on an affine variety $X$ and let $\chi$ be a character of $G$.  A point $x \in X$ is $\chi$-semistable if and only if for any one-parameter subgroup $\lambda$ of $G$ such that $\lim_{t \to 0} \lambda(t) \cdot x$ exists we have $\langle \chi, \lambda \rangle \ge 0$.  If the inequality is strict for any nontrivial such $\lambda$, then $x$ is $\chi$-stable.
\end{thm} 

Note that in our situation, each point of $\Rep(Q, \alpha)$ has the nontrivial subgroup $\Delta$ contained in its stabilizer.  Therefore, we must include the condition that $\chi(\Delta) = 1$ in the semistability condition.  Additionally, in the condition for stability we need to consider $\lambda$ that are not in $\Delta$.

In order to simplify the criterion for semistability, we would like to understand the characters of the group $G(\alpha)$.  Since $G(\alpha) = \prod_{i \in I_Q} \GL(\alpha, \alpha_i)$, then every character $\chi \in \Hom(G(\alpha), \C^{\times})$ is of the form $\chi = \prod_{i \in I_Q} \chi_i$, where $\chi_i$ is a character of $\GL(n, \alpha_i)$.  Since $\chi_i$ is a group homomorphism, then it is invariant with respect to the standard conjugation action of $\GL(n, \alpha_i)$ on itself.  This means, $\chi_i$ is defined where it maps invertible diagonal matrices.  That is, $\chi$ is defined by its restriction to the algebraic torus $(\C^{\times})^{\alpha_i}$. By Corollary \ref{toruscharcomp}, we know that $\chi_i = t_1^{n_1} \cdots t_m^{n_m}$ for some $n_1, \dots, n_m \in \Z$.  Furthermore, $\chi|_{(\C^{\times})^{\alpha_i}}$ must be invariant under permutation of the diagonal matrix entries.  This means $n_1 = \cdots = n_m$.  Consequently, $\chi_i = \det^{-\theta_i}$, for some $\theta_i \in \Z$ (here $\det$ is the character that sends each element of $\GL(\alpha)$ to its determinant).  Thus any character of $G(\alpha)$ has the form $\chi = \prod_{i \in I_Q} \det^{-\theta_i}$ for some $\theta_i \in \Z$.

The converse of this statement is also clearly true.  Namely, given $\theta = (\theta_i) \in \Z^{I_Q}$, the product $\prod_{i \in I_Q} \det^{-\theta_i}$ is a character of $G(\alpha)$.  Note that in order for us to apply the Hilbert-Mumford criterion to the $G(\alpha)$-action on $\Rep(Q, \alpha)$, we must have $\chi(\Delta) = 1$ (i.e. we need $\chi$ to be a character of $PG(\alpha)$).  This condition is equivalent to $\theta \cdot \alpha = \sum_{i} \theta_i\alpha_i = 0$.  Since choosing a character of $G(\alpha)$ (or $PG(\alpha)$) is the same as picking a $\theta \in \Z^{I_Q}$, we will denote such characters by $\chi_{\theta}$.  This equivalent way of interpreting characters of $G(\alpha)$ leads us to an alternative definition of semistability. 

\begin{defn}
\label{thetastabdef}
Let  $\theta = (\theta_i) \in \R^{I_Q}$.  A representation $V$ of the quiver $Q$ with (nonzero) dimension vector $\alpha = (\alpha_i)$ is called \textbf{$\theta$-semistable} if $\theta \cdot \alpha = 0$ and for any subrepresentation $W \subset V$ with dimension vector $\beta$ we have $\theta \cdot \beta \le 0$.  We say that $V$ is \textbf{$\theta$-stable} if under the previous assumptions $\theta \cdot \beta < 0$ for any nontrivial proper subrepresentation $W \subset V$ with dimension vector $\beta$. 
\end{defn}  

We can now state the following theorem.

\begin{thm}[King]
\label{kingstab}
Let $Q$ be a quiver, and let $\theta = (\theta_i) \in \Z^{I_Q}$. Let $\alpha \in \Z^{I_Q}_{\ge 0}$ be a dimension vector such that $\theta \cdot \alpha = 0$.  Any $V \in \Rep(Q, \alpha)$ is $\chi_{\theta}$-semistable (resp. $\chi_{\theta}$-stable) if and only if $V$ is $\theta$-semistable (resp. $\theta$-stable). 
\end{thm}
\begin{proof}
Let $V \in \Rep(Q, \alpha)$.  Using the argument from the proof of Lemma \ref{filtlmm2}, any one-parameter subgroup $\lambda$ of $G(\alpha)$ such that $\lim_{t \to 0} \lambda(t)\cdot V$ exists induces a filtration by subrepresentations
\[
0 = V_0 \subset V_1 \subset \cdots \subset V_n = V,
\]
 where $(V_k)_i = \bigoplus_{0 \le l \le k} V_{ik}'$ given by the weight decomposition $V_i = \bigoplus_k V_{ik}'$ with respect to the $\lambda(t)$-action (see Theorem \ref{torusweight}).  Furthermore, the weight decomposition gives us
\[
\chi_{\theta}(\lambda(t)) = \prod_i{ \det(\lambda(t))^{-\theta_i}} = \prod_i \prod_k t^{k\theta_i \dim V_{ik}'} = t^{\sum_{i,k}k \theta_i \dim V_{ik}' }.
\]
Let $\beta_i$ be the dimension vector of the quotient representation $V_k/V_{k-1}$ for $1 \le k \le n$.  The above computation gives us
\begin{align*}
\langle \chi_\theta, \lambda \rangle & = \sum_{i,k}k \theta_i\dim V_{ik}'  = \sum_k k \theta \cdot \beta_k = (n+1)\theta \cdot \alpha - \sum_k (n-k+1) \theta \cdot \beta_k\\ &= - \sum_k (n-k+1) \theta \cdot \beta_k, 
\end{align*}
since $\alpha = \beta_1 + \cdots + \beta_n$ and $\theta \cdot \alpha = 0$.  If $V$ is $\theta$-semistable, then this implies $\langle \chi_\theta, \lambda \rangle \ge 0$ by Theorem \ref{hmcriterion} because $V_k$ is a subrepresentation of $V$ with dimension vector $\beta_1 + \cdots + \beta_k$.

Conversely, let $V$ be $\chi_{\theta}$-semistable.  Let $W \subset V$ be a subrepresentation with dimension vector $\beta$.  Consider the one-parameter subgroup $\lambda(t)$ of $G(\alpha)$ defined such that $\lambda$ acts by $t^{-1}$ on $W$ and by $t^{-2}$ on $V/W$.  By the computation above we get
\[
\langle \chi_\theta, \lambda \rangle = - 2\theta \cdot \beta - \theta \cdot (\alpha - \beta) = - \theta \cdot \beta \ge 0.
\]
Therefore, $\theta \cdot \beta \le 0$, so $V$ is $\theta$-semistable.  An analogous argument in the case of stable representations completes the proof.
\end{proof}

We will denote by $\Rep^{ss}_{\theta}(Q, \alpha)$ the open subset of $\theta$-semistable representation in $\Rep(Q,\alpha)$, and by $\Rep^{s}_{\theta}(Q, \alpha)$ the open subset of $\theta$-stable ones.  The corresponding Proj GIT quotient and geometric quotient will be denoted by $\Rep(Q, \alpha)//_{\theta} G(\alpha)$ and $\Rep^s_\theta(Q, \alpha)/G(\alpha)$, respectively (note that we are using the definition of stability that takes into account the subgroup $\Delta$).


\begin{rmk}
\label{slopestab}
Using the formalism proposed in \cite{Rud1997} or \cite{Br2007} it is possible to give a slightly different definition of $\theta$-stability.  Namely, let $\theta = (\theta_i) \in \R^{I_Q}$, and let $V$ be a quiver representation with dimension vector $\alpha$.  The \emph{slope} of $V$ is the quotient 
\[
\phi(V) = \frac{\theta \cdot \alpha}{\sum_{i} \alpha_i}.
\]
We say $V$ is \emph{$\theta$-semistable} if for any nontrivial subrepresentation $W \subset V$ we have $\phi(W) \le \phi(V)$.  We say $V$ is \emph{$\theta$-stable} if the inequality is strict for any nontrivial proper subrepresentation $W$.  It is not hard to show that in the case $\theta \cdot \alpha = 0$ this corresponds to the definition of $\theta$-stability we saw earlier.  Moreover, $\theta$-semistability is invariant (semistable objects remain semistable) under shifts by multiples of the vector $(1,1, \dots, 1)$, and it is always possible to find such that shift that would make the dot product with $\alpha$ equal $0$ (Lemma 3.4 \cite{Rud1997}).  Unless specified otherwise, we will be using Definition \ref{thetastabdef} of stability, rather than the one given in the remark.
\end{rmk}

If $Q$ is a quiver, then for a fixed $\theta$ the $\theta$-semistable representations of $Q$ form a subcategory $\cR_{\theta}(Q)$ of $\cR(Q)$.  Furthermore, one can check that this category is abelian (see proof of Proposition 3.1 in \cite{King1994}). 

\begin{prp}
The category $\cR_{\theta}(Q)$ is abelian and an object in this category is simple if and only if it is $\theta$-stable.
\end{prp}

A consequence of this proposition is that an analogue of Theorem \ref{jhfilt} holds, so that every $\theta$-semistable representation of $Q$ has a filtration 
$0 = V_0 \subset \cdots \subset V_n = V$ by subrepresentations such that the quotients $V_k/V_{k-1}$ are $\theta$-stable and these quotients are uniquely determined up to permutation.  Another way of saying this is that $\gr(V_{\bullet})$ is determined by $V$ up to isomorphism.  We will therefore write $\gr^s(V)$ in this case.  A argument similar to the one found in Theorem \ref{closedorbitthm} can be used to prove the following (see Proposition 3.2 in \cite{King1994} for details).

\begin{thm}
\label{quiverpoly}
Let $V \in \Rep^{ss}(Q, \alpha)$. The orbit $G(\alpha) \cdot V$ is closed if and only if $V$ is the direct sum of $\theta$-stable representations.  Two representations $V_1, V_2 \in \Rep^{ss}(Q, \alpha)$ are S-equivalent if and only if $\gr^s(V_1) \cong \gr^s(V_2)$.
\end{thm}

The representation $V$ in the above theorem that is the direct sum of $\theta$-stable representation is called \emph{$\theta$-polystable}.  A consequence of the theorem is that the points of $\Rep(Q, \alpha)//_{\theta}G(\alpha)$ are in bijection with $\theta$-polystable representations.

\begin{exa}
\label{quivermoduliex}
\leavevmode
\begin{enumerate}
\item  Consider $\Rep(J, \alpha)$ for a nonzero $\alpha$.  Since the Jordan quiver consists of a single vertex, then the condition that $\theta \cdot \alpha = 0$ implies $\theta  = 0$.  This means all representations are $\theta$-semistable and $\Rep^{ss}_{\theta}(J, \alpha) = \Rep(J, \alpha)$.  As previously seen in Example \ref{projgitex}, we have $\Rep(J , \alpha)//_{\theta} G(\alpha) = \Rep(J,\alpha)^{G(\alpha)} = \A^{\alpha}$.

Note that $\theta = 0$ implies no $\theta$-stable representation of $J$ can have nontrivial proper subrepresentations.  This means that the only $\theta$-stable are simple representations.  This means $\alpha = 1$, and $\Rep^s_{\theta}(J, \alpha)/G(\alpha) = \A^1$.

\item In general, if $\theta = 0$, we see that the all representations in $\Rep(Q, \alpha)$ are $\theta$-semistable, and the only $\theta$-stable representations are simple representations.  This makes sense, since by Example \ref{projgitex}, the Proj GIT quotient for the trivial character $\chi_{\theta}$ is the same as the affine GIT quotient. 

\item Consider $\Rep(A_2, \alpha)$ with respect to the standard action of $G(\alpha)$.  Assume $\theta \neq 0$ ($\theta = 0$ was already discussed above).  If $\theta  = (0, \theta_2)$, then $\alpha_1 = 0$ for $\theta \cdot \alpha = 0$.  In this case, all representations of $A_2$ are $\theta$-semistable, and the only $\theta$-stable representation is the simple representation with dimension vector $\alpha = (1,0)$.  Similarly, if $\theta  = (\theta_1, 0)$, then all representations with dimension vector $(0, \alpha_1)$ are $\theta$-semistable, and the simple representation for $\alpha = (0,1)$ is $\theta$-stable.

If $\theta_1$ and $\theta_2$ are both nonzero, then for any $\alpha \neq 0$ such that $\theta \cdot \alpha = 0$ we have $\frac{\alpha_0}{\alpha_1} = -\frac{\theta_2}{\theta_1}$.  It follows that $\theta = c(\alpha_1, -\alpha_0)$ where $c \in \Z$.

Therefore, for a representation $V \in \Rep(A_2, \alpha)$ to be $\theta$-semistable we must have $\alpha_1\beta_0 -\alpha_0\beta_1 \le 0$ for any subrepresentation with dimension vector $\beta$ or $\alpha_0\beta_1 -\alpha_1\beta_0 \le 0$ for any subrepresentation with dimension vector $\beta$.  The latter situation is impossible since there is always a subrepresentation of dimension $(0,1)$.  

This means we have the first situation, and there cannot be a subrepresentation with $\beta = (1,0)$, so the linear transformation corresponding to the arrow must be injective.  Furthermore, we have $\alpha_1 \ge \alpha_0$.  However, even if this is true there is a subrepresentation with dimension vector $(1,1)$ which violates the inequality unless $\alpha_0 = \alpha_1$.  If $\alpha = (\alpha_0,\alpha_0)$, then every representation $V = (V_0,V_1,f_a)$ with injective $f_a$ is $\theta$-semistable.  Moreover, it follows that the only $\theta$-stable representations have dimension $\alpha = (1,1)$ and injective $f_a$.

Note that there is a single $G(\alpha)$-orbit containing all of the $\theta$-semistable (respectively, $\theta$-stable) representations whenever they exist.  It follows that for any fixed $\alpha$ and $\theta$ such that $\theta$-semistable representations exist, we have  $\Rep(A_2, \alpha)//_{\theta} G(\alpha) = \Spec \C$.
Similarly $\Rep^s_{\theta}(A_2, \alpha)/G(\alpha) = \Spec \C$ whenever $\theta$-stable representations exist.

\item Let $K_n$ be the Kronecker quiver consisting of $2$ vertices and $n$ arrows.  Consider $\Rep(K_n, \alpha)$ for $\alpha = (\alpha_0, \alpha_1) = (d,1)$, and let $\theta = (1,-d)$.

\[
\begin{tikzcd}
1 
& 0 \arrow[l, bend left, shift left=1.75ex]
\arrow[l, draw=none, "\raisebox{-4.5ex}{\vdots}" description]
\arrow[l, bend right, shift left=1.0 ex]
\arrow[l, bend right, shift right=1.0ex]  
\end{tikzcd}
\]

If the representation $V \in \Rep(K_n, \alpha)$ is $\theta$-semistable, then it cannot contain the subrepresentation with dimension vector $(1,0)$.  This means $\bigcap_{1 \le k \le n} \ker(f_k) = 0$, and consequently the linear transformation $f := (f_{a_1}, \cdots, f_{a_n})$ must be injective.  Conversely, if $f$ is injective, then every subrepresentation of $V$ must have dimension vector $(\beta_0, 1)$ with $\beta_0 \le d$.  Therefore, $V$ is $\theta$-semistable.  In fact, $V$ must also be $\theta$-stable, since every proper subrepresentation must have dimension vector $(\beta_0, 1)$ with $\beta_0 < d$.

It follows that representation $V$ is $\theta$-stable if and only if it is $\theta$-semistable if and only if the induced linear transformation $f: \C^{d} \rightarrow \C^n$
is injective.  Thus, $\Rep^s_\theta(K_n, \alpha)/G(\alpha) = \Rep(K_n, \alpha)//_{\theta} G(\alpha)$ parametrizes subspaces of $\C^n$ of dimension $d$, identifying it with the Grassmannian $\Gr(d,n)$.

\end{enumerate}
\end{exa}

We conclude this section by demonstrating that the quotient $\Rep^s_\theta(A_2, \alpha)/G(\alpha)$ is a moduli space.

Let $Q$ be a quiver and $\alpha$ a dimension vector. Let us consider the following moduli functor:
\begin{align*}
& \cM^s_{\theta}(Q, \alpha): \opcat{\Var} \rightarrow \Set\\
& \cM^s_{\theta}(Q, \alpha)(T) = \left \{\begin{tabular}{lll} $\cF = (\{ T \times \C^{\alpha_i}\}_{i \in I_Q}, \{\varphi_a\}_{a \in A_Q})$,\\ $\varphi_a: T \times \C^{\alpha_{s(a)}} \rightarrow T \times \C^{\alpha_{t(a)}}$\\ $\cF|_t \in \Rep(Q, \alpha)$ is $\theta$-stable \end{tabular} \right\} \Bigg/  \begin{tabular}{lll}vector bundle isomor-\\phisms commuting\\ with the $\varphi_a$\end{tabular}\\
& \cM^s_{\theta}(Q, \alpha)(f) = f^* \textrm{ for } f: S \rightarrow T.  
\end{align*}

Note that the identity morphism $\prod_{a \in A_Q} \Hom(\C^{\alpha_{s(a)}}, \C^{\alpha_{t(a)}}) \rightarrow \Rep(Q, \alpha)$ defines the \emph{tautological family} $\cT  = (\{\Rep(Q, \alpha)\times \C^{\alpha_i}\}, \{\varphi_a\})$ of representations over $\Rep(Q, \alpha)$ consisting of trivial bundles $\Rep(Q, \alpha)\times \C^{\alpha_i}$ and morphisms $\varphi_a: \Rep(Q, \alpha)\times \C^{\alpha_{s(a)}} \rightarrow \Rep(Q, \alpha)\times \C^{\alpha_{t(a)}}$ such that $(\varphi_a)_V = f_a$ on the fiber for each  $V \in \Rep(Q, \alpha)$.  Restricting this family to the open subset of $\theta$-stable representations yields the tautological family $\cT^s_{\theta}$.

\begin{thm}
\label{quivermodulic}
The quotient variety $M^s_{\theta}(Q, \alpha) := \Rep^s_\theta(A_2, \alpha)/G(\alpha)$ is a coarse moduli space for the functor $\cM^s_{\theta}(Q, \alpha)$. 
\end{thm}

\begin{proof}
Note that the tautological bundle $\cT^s_{\theta}$ has the local universal property with respect to the moduli functor $\cM^s_{\theta}(Q, \alpha)$.  Indeed, if $\cF$ is a family of stable representations of $Q$ over an algebraic variety $T$, then it can be trivialized in a neighborhood $U \subset T$ of any $t \in T$.  That is, $\cF|_{U} \cong (\{U \times V_i\}, \{\Id \times f_a\})$ for some $V = (\{V_i\}, \{f_a\}) \in \Rep^s_\theta(Q,\alpha)$.  Taking $h: U \rightarrow \Rep^s_\theta(Q,\alpha)$ to be the constant morphism sending $U$ to $V$, gives us $h^*\cT^s_{\theta} \cong \cF|_{U}$.

If $V_1, V_2 \in \Rep^s_\theta(Q,\alpha)$, then clearly $(\cT^s_{\theta})_{V_1} \cong (\cT^s_{\theta})_{V_2}$ if and only if $V_1 \cong V_2$, which only happens if and only if $V_1$ and $V_2$ belong to the same $G(\alpha)$-orbit.  Thus Theorem \ref{coarsemodulithm} implies there is a natural transformation $\eta: \cM^s_{\theta}(Q, \alpha) \rightarrow h_{M^s_{\theta}(Q, \alpha)}$ that makes $(M^s_{\theta}(Q, \alpha), \eta)$ into a coarse moduli space. 
\end{proof}

\begin{rmk}
\label{quiverssmoduli}
It is easy to modify the moduli functor above to describe families of $\theta$-semistable representations of $\Rep(Q, \alpha)$ up to S-equivalence.  One can show (see Proposition 5.2 in \cite{King1994}) that the quotient variety $M^{ss}_{\theta}(Q, \alpha) := \Rep(A_2, \alpha)//_{\theta} G(\alpha)$ is a coarse moduli space for this functor.  As we have already seen, Theorem \ref{quiverpoly} implies that $M^{ss}_{\theta}(Q, \alpha)$ parametrizes $\theta$-polystable representations up to isomorphism.    
\end{rmk}

If $\alpha = (\alpha_i)$ is such that the greatest common divisor among the $\alpha_i$ is $1$, then $\alpha$ is called \emph{indivisible}. Suppose $\alpha$ is indivisible. Let $c = (c_i) \in \Z^{I_Q}$ be such that $\sum_{i} c_i\alpha_i = 1$.  Consider the character $\chi_c: G(\alpha) \rightarrow \C^{\times}$ defined by $\chi_c(g) = \prod_i \det(g_i)^{-c_i}$.  

The $G(\alpha)$-action on $\Rep(Q,\alpha)$ induces a $G(\alpha)$-bundle structure on $\cT^s_{\theta}$, such that the one-parameter subgroup $\Delta$ acts by $t$ on the fibers.  Multiplying by the character $\chi_c$ along the fibers defines a $G(\alpha)$-bundle structure that $\Delta$ acts trivially.  Therefore, the family $\cT^s_{\theta}$ descends (see Theorem 2.3 in \cite{DN1989} for general case) to a family $\cU$ (i.e. $\cU$ pulls back to $\cT^s_{\theta}$) of $\theta$-stable representations on $M^s_{\theta}(Q, \alpha)$.  One can show this family is universal, giving us the following theorem (see Proposition 5.3 in \cite{King1994}.

\begin{thm}
\label{quivermodulif}
If $\alpha$ is indivisible (i.e. the $\alpha_i$ do not have a nontrivial common divisor), then $M^s_{\theta}(Q, \alpha)$ is a fine moduli space.
\end{thm}

\section{Framed representations}

\subsection{Framing and stability}
Given a quiver $Q$ we can define a bilinear form on the dimension vectors of its representations.

\begin{defn}
Let $Q$ be a quiver. The \textbf{Euler-Ringel form} associated with $Q$ is the bilinear form given by
\[
\langle \alpha, \beta \rangle = \sum_{i \in I_Q} \alpha_i \beta_i - \sum_{a \in A_Q} \alpha_{s(a)}\beta_{t(a)},
\]
where $\alpha, \beta \in \Z^{I_Q}$.
The quadratic form $q(\alpha) = \langle \alpha, \alpha \rangle$ is called the \bf{Tits form}.
\end{defn}

Note that the Euler-Ringel form computes the Euler characteristic of a pair of (finite-dimensional) representations $V$ and $W$ with dimension vectors $\alpha$ and $\beta$ (see Corollary 1.4.3 in \cite{Bri2012}).  That is $\langle \alpha, \beta \rangle = \dim \Hom(V,W) - \dim \Ext^1(V,W)$.

An easy consequence of Theorems \ref{projstab} and \ref{projbun} is the following proposition.
\begin{prp}
Let $\theta \in \Z^{I_Q}$ be such that there are $\theta$-stable representations in $\Rep(Q, \alpha)$.  We have that $M^s_\theta(Q, \alpha)$ is a smooth open subset of $M^{ss}_\theta(Q, \alpha)$ such that 
\[
\dim M^s_\theta(Q, \alpha) = \dim \Rep(Q,\alpha) - \dim G(\alpha) + 1 = 1 - q(\alpha).
\]
\end{prp}

Thus we see that for a given dimension vector $\alpha$ it is often the case that $M^s_\theta(Q, \alpha)$ is either empty or finite (see Example \ref{quivermoduliex}).  In order to obtain a more interesting class of varieties we introduce additional structure to the quiver $Q$.


\begin{defn}
Let $Q = (I_Q, A_Q, s, t)$ be quiver.  A \textbf{framing} of $Q$ is a quiver $Q^{fr} = (I_{Q^{fr}}, A_{Q^{fr}}, s^{fr},t^{fr})$ such that
\begin{enumerate}
\item $I_{Q^{fr}} = I_Q \sqcup I'_Q$ together with a bijection between $I_Q$ and $I'_Q$ via the mapping $i \mapsto i'$,
\item $A_{Q^{fr}} = A_Q \sqcup \{b_i\}_{i \in I_Q}$,
\item $s^{fr}(a) = s(a)$ and $t^{fr}(a) = t(a)$ for all $a \in A_Q$,
\item $s^{fr}(b_i) = i$ and $t^{fr}(b_i) = i'$.
\end{enumerate}
\end{defn}

In other words, $Q^{fr}$ consists of the quiver $Q$ together with one additional vertex for each vertex of $Q$ connected to that vertex by a single arrow.  It follows that a representation of $Q^{fr}$ consists of $V = (\{V_i\} \sqcup\{V'_{i}\}, \{f_a\} \sqcup \{f_{b_i}\})$, where $V'_{i}$ is placed at the vertex $i' \in I'_Q$ and $f_{b_i}: V_i \rightarrow V'_i$ is a linear transformation.

We can similarly define the representations of $Q^{fr}$ in the standard coordinate spaces.  For a dimension vector $(\alpha, \alpha') \in \Z_{\ge 0}^{I_Q} \sqcup \Z_{\ge 0}^{I'_Q} = \Z_{\ge 0}^{I_{Q^{fr}}}$, denote the affine variety of such representation by $\Rep(Q^{fr}, \alpha, \alpha')$.  We will call the elements of $\Rep(Q^{fr}, \alpha, \alpha')$ \emph{framed representations} of $Q$.

There is an action of the group $G(\alpha)$ on $\Rep(Q^{fr}, \alpha, \alpha')$ defined by $f_a \mapsto g_{t(a)} f_a \inv{g_{s(a)}}$ for $a \in A_Q$ and $f_{b_i} \mapsto f_{b_i} \inv{g_i}$.

An alternative description of framed representations can be given as follows.

\begin{prp}
\label{framedeqprp}
Let $Q = (I_Q, A_Q, s, t)$ be a quiver. Fix a dimension vector $\alpha$ for Q, and let $\hat{Q}$ be the quiver with vertex set $I_Q \sqcup \{\infty\}$ and an arrow set consisting of $A_Q$ together with $\alpha_i$ additional arrows from each vertex $i \in I_Q$ to $\infty$. Let the dimension vector $\hat{\alpha}$ be given by $\hat{\alpha}_i = \alpha_i$ for $i \in I_Q$ and $\hat{\alpha}_{\infty} = 1$.  There is a $G(\alpha)$-equivariant isomorphism $\Rep(Q^{fr}, \alpha, \alpha') \rightarrow \Rep(\hat{Q}, \hat{\alpha})$.
\end{prp}
\begin{proof}
The group $G(\alpha)$ acts on $\Rep(\hat{Q}, \hat{\alpha})$ via the embedding $G(\alpha) \hookrightarrow G(\hat{\alpha})$ given by $g = (g_i) \mapsto (g_i,\Id)$.  This may be viewed as the action of $G(\alpha) \cong G(\hat{\alpha})/\C^{\times}$ after scaling by $g_{\infty}$.

Consider the morphism of affine varieties $\Rep(Q^{fr}, \alpha, \alpha') \rightarrow \Rep(\hat{Q}, \hat{\alpha})$ defined by $(\{V_i\}, \{V'_{i}\}, \{f_a\}, \{f_{b_i}\}) \mapsto (\{V_i\}, \{V_{\infty}\}, \{f_a\}, \{r_{ik}\})$ where $V_{\infty} = \C$ and $r_{ik}$ is the $k$-th row of $f_{b_i}$ for $1 \le k \le \alpha_i$.  Note that $g_{t(a)} f_a \inv{g_{s(a)}} \mapsto  g_{t(a)} f_a \inv{g_{s(a)}}$ and $f_{b_i} \inv{g_i} \mapsto (r_{ik}\inv{g_i})$ under this morphism, so it is $G(\alpha)$-equivariant.  It is easy to check the morphism is bijective.
\end{proof}

Thus, if we are interested in constructing GIT quotients of $\Rep(Q^{fr}, \alpha, \alpha')$ with respect to the $G(\alpha)$ action described above, we can instead look at the action of $G(\alpha)$ on  $\Rep(\hat{Q}, \hat{\alpha})$.  This approach yields the following following theorem:

\begin{thm}
\label{framedthm}
Let $\theta \in \Z_{> 0}^{I_Q}$. Let $\chi_{\theta}$ be the corresponding character of $G(\alpha)$.  We have the following
\begin{enumerate}
\item Let $V = (\{V_i\}, \{f_a\}) \in \Rep(Q, \alpha)$.  A framed representation $V^{fr} = (\{V_i\},\\ \{V'_{i}\}, \{f_a\}, \{f_{b_i}\}) \in \Rep(Q^{fr}, \alpha, \alpha')$ is $\chi_{\theta}$-semistable if and only if there is no nontrivial subrepresentation $W \subset V$ such that $W_i \subset \ker f_{b_i}$ for each $i \in I_Q$. 
\item The $G(\alpha)$ action on the $\chi_{\theta}$-semistable elements of $\Rep(Q^{fr}, \alpha, \alpha')$ is free, and every $\chi_{\theta}$-semistable representation in $\Rep(Q^{fr}, \alpha, \alpha')$ is $\chi_{\theta}$-stable.
\item If $\cM_{\theta}(Q, \alpha, \alpha') := \Rep(Q^{fr}, \alpha, \alpha')//_{\chi_{\theta}}G(\alpha)$ is nonempty, then it is smooth and of dimension $\dim \Rep(Q^{fr}, \alpha, \alpha') - \dim G(\alpha) = \alpha \cdot \alpha' - q(\alpha)$.
\end{enumerate}
\end{thm}

\begin{proof}
Let $\hat{\theta} = (\theta, - \sum_{i \in I_Q} \theta_i \alpha_i)$. Let $\hat{V} = (\{V_i\}, \{V_{\infty}\}, \{f_a\}, \{r_{ik}\}) \in \Rep(\hat{Q}, \hat{\alpha})$ be such that $r_{ik}$ are the rows of $f_{b_i}$.  By Theorem \ref{kingstab}, we have that $\hat{V}$ is $\chi_{\hat{\theta}}$-semistable with respect to the action of $G(\hat{\alpha})$ if and only if it is $\hat{\theta}$-semistable.  Moreover, regarding $G(\alpha)$ as $G(\hat{\alpha})/\C^{\times}$ we see that $\hat{V}$ is $\chi_{\hat{\theta}}$-semistable if and only if it is $\chi_{\theta}$-semistable with respect to the induced $G(\alpha)$ action.  

Note that $\hat{V}$ is $\chi_{\hat{\theta}}$-semistable if and only if for all subrepresentations $\hat{W} \subset \hat{V}$ (with dimension vector $\hat{\beta}$) the following inequality holds 
\[
\hat{\theta} \cdot \hat{\beta}  = \sum_i \theta_i\beta_i - \beta_{\infty}\sum_i \theta_i \alpha_i \le 0.
\]
Let $\hat{W} \subset \hat{V}$ be a subrepresentation. Note that $\beta_{\infty}$ can only be $0$ or $1$.  If $\beta_{\infty} = 1$, then $\theta_i > 0$ implies the inequality is always satisfied.  If $\beta_{\infty} = 0$, then the inequality is only satisfied if $\beta_i = 0$ for all $i \in I_Q$.  Consequently, $\hat{V}$ is $\chi_{\theta}$-semistable if there is no nontrivial subrepresentation $W \subset V$ such that $W_i \subset \bigcap_k \ker r_{ik} = \ker f_{b_i}$.  

Conversely, any such nontrivial $W \subset V$ satisfying $W_i \subset \bigcap_k \ker r_{ik}$ may be completed to a subrepresentation $\hat{W} \subset \hat{V}$ that violates the above inequality (set $\beta_{\infty} = 0$).  Since the $G(\alpha)$-equivariant isomorphism between $\Rep(Q^{fr}, \alpha, \alpha')$ and $\Rep(\hat{Q}, \hat{\alpha})$ sends $\hat{V}$ to $V^{fr}$, then part (1) follows.  

If the stabilizer of a $\chi_{\theta}$-semistable representation $V^{fr} \in \Rep(Q^{fr}, \alpha, \alpha')$ contains a nontrivial $g$, then $g$ stabilizes $V$.  Furthermore, $f_{b_i}\inv{g} = f_{b_i}$, so $f_{b_i}(g - \Id) = 0$.  Consequently, $\Im(g - \Id)$ is a nontrivial subrepresentation of $V$ violating the semistability condition in (1).  Since the inequality in the proof of part (1) is strict for nontrivial subrepresentations, we see that part (2) follows by an argument analogous to the proof of (1).  Part (3) follows from Theorem \ref{projbun}. 
\end{proof}

Note that in the proof of the above theorem the stability condition for framed representations does not depend on the specific $\theta$, but rather on the fact that the components of $\theta$ were positive.  This implies as long as $\theta_1, \theta_2 \in \Z_{> 0}^{I_Q}$ we have $\cM_{\theta_1}(Q, \alpha, \alpha') = \cM_{\theta_2}(Q, \alpha, \alpha')$.

\begin{rmk}
\label{dualframed}
We can define framings of $Q$ with arrows pointing away from the supplementary vertices (rather than towards them).  This yields a dual notion of a framed representation of $Q$ with a corresponding $G(\alpha)$ action.  Indeed, let $\theta \in \Z_{< 0}^{I_Q}$.  In this case, the analogue of the semistability criterion given in part (1) of Theorem \ref{framedthm} becomes: $V^{fr}$ is $\chi_{\theta}$-semistable (a consequently $\chi_{\theta}$-stable) if and only if there is no proper subrepresentation $W \subset V$ such that $\Im f_{b_i} \subset W_i$ for each $i \in I_Q$.   
\end{rmk}

\subsection{Examples}

\begin{exa}[$A_n$ quiver]
\label{framedanex}
Consider the framed $A_n$ quiver.  Let $\alpha = (\alpha_1, \dots, \alpha_n)$ and $\alpha' = (\alpha_0, 0, \cdots, 0)$ be dimension vectors for this quiver such that $\alpha_0 > \alpha_1 > \cdots > \alpha_n$ holds.

\begin{equation*}
\begin{tikzcd}
\C^{\alpha_1}  \arrow[d, swap, "f_{b_1}"] & \C^{\alpha_2} \arrow[l, swap, "f_{a_1}"] \arrow[d, ""] & \cdots \arrow[l, swap, "f_{a_2}"] & \C^{\alpha_n} \arrow[l, swap, "f_{a_{n-1}}"] \arrow[d, ""]\\
\C^{\alpha_0} & 0 & \cdots & 0 
\end{tikzcd}
\end{equation*}

For a stability parameter $\theta \in \Z_{> 0}^{I_Q}$ we see that $\chi_{\theta}$-stable framed representations are exactly the ones such that the linear transformations $f_{a_1}, \cdots, f_{a_{n-1}}, f_{b_1}$ are all injective.  Indeed, if any one of them fails to be injective, then its kernel has a nontrivial intersection with the kernel of one of the $f_{b_i}$, yielding a subrepresentation that contradicts the stability condition in Theorem \ref{framedthm}.

Letting $\C^{\alpha_0} \supset F_i  = \Im(f_{b_1} \circ f_{a_1} \circ \cdots \circ f_{a_{i}})$, we get a partial flag 
\[
F_{n} \subset F_{n-1} \subset \cdots \subset F_0 = \C^{\alpha_0},
\]
such that $\dim F_i = \alpha_i$.  Since $G(\alpha)$ acts by change of basis at each vertex, we get that the points of $\cM_{\theta}(A_n, \alpha, \alpha')$ are in bijection with the points of the flag variety $\Fl(\alpha_0, \dots, \alpha_n)$, and in fact in can be shown that the two varieties are isomorphic.
\end{exa}

\begin{exa}[Jordan quiver]
Letting $\theta \in \Z_{< 0}^{I_Q}$, we consider the dual framing discussed in Remark \ref{dualframed} for the Jordan quiver $J$. 

\begin{equation*}
\begin{tikzcd}
\C^{\alpha} \arrow[out=135,in=225,loop, swap, "f_a"] & \C^{\alpha'} \arrow[l, swap, "f_b"]
\end{tikzcd}
\end{equation*}

Recall that finite-dimensional representations of $J$ may be viewed as finite-dimensional $\C[z]$-modules, where $z$ acts by $f_a$.  Using the stability criterion in the remark, we see that the a representation $V^{fr} \in \Rep(J^{fr}, \alpha, \alpha')$ is $\chi_{\theta}$-stable if only if $\Im(f_b)$ generates $V$ as a $\C[z]$-module.  

If $\alpha' = 0$, then the stability condition is only satisfied if $\alpha = 0$.  Therefore, for $\alpha \ge 1$ we have that $\cM_{\theta}(J, \alpha, 0)$ is empty.

If $\alpha' = 1$, then $f_b$ must have rank $1$ for there to be stable framed representations.  For a representation to be stable, the vector $f_b(1)$ must generate $\C^{\alpha}$ as a $\C[z]$-module.  In other words, stable representations correspond to cyclic $\C[z]$-modules (the vector $f_b(1)$ is called a \emph{cyclic vector}).  

Using the structure theorem for modules over a principal ideal domain, this module is isomorphic to $\C[z]/(p(z))$, where $p(z)$ is a polynomial of degree $\alpha$ (a polynomial of smaller degree would not suffice to generate $\C^{\alpha}$).  This implies $p(z)$ is the characteristic polynomial of $f_a$ up to multiplication by a constant.  Therefore, geometrically, a stable representation corresponds to an (unordered) collection of the eigenvalues of $f_a$ (counting multiplicity). Thus, $\cM_{\theta}(J, \alpha, 1) \cong \A^{(\alpha)} := \A^{\alpha}/S_{\alpha}$ is the quotient of $\A^{\alpha}$ by the action of the symmetric group $S^{\alpha}$ (i.e. $\alpha$-th symmetric power of $\A^1$).  

Note that in this case $\cM_{\theta}(J, \alpha, 1)$ may also be interpreted as the \emph{punctual Hilbert scheme} $\textrm{Hilb}^{\alpha}(\A^1)$.  In general, the variety $\textrm{Hilb}^{\alpha}(X)$ parametrizes collections of $\alpha$ points (potentially with multiplicities) in the algebraic variety $X$ (see \cite{Na1999} for details).
\end{exa}

\section{Double quivers and Hamiltonian reduction}
\subsection{Symplectic Geometry}

We begin with a quick reminder of basic symplectic geometry and the moment map construction for cotangent bundles to $C^{\infty}$-manifolds.  

A \emph{symplectic manifold} consists of a pair $(M, \omega)$, where $M$ is is a $C^{\infty}$-manifold and $\omega$ is a closed, non-degenerate $2$-form on $M$ (called the \emph{symplectic form}).  If $(M_1, \omega_1), (M_2, \omega_2)$ are symplectic manifolds, then a diffeomorphism $f: M_1 \rightarrow M_2$ is called a \emph{symplectomorphism} if $f^*\omega_2 = \omega_1$.

\begin{exa}
\leavevmode
\begin{enumerate}
\item The manifold $\R^{2n}$ is a symplectic manifold.  If the coordinates on $\R^{2n}$ are $x_1, 
\cdots, x_n, y_1, \cdots, y_n$, then $\omega = dx_1 \wedge dy_1 + \cdots + dx_n \wedge dy_n$.
\item If $M$ is a $C^{\infty}$-manifold, then the cotangent bundle $T^*M$ has a canonical structure of a symplectic manifold.  If $q_1, \cdots , q_n$ are the local coordinates along $M$ and $p_1, \cdots, p_n$ are the induced coordinates on the fiber, then $\omega = dp_1 \wedge dq_1 + \cdots + dp_n \wedge dq_n$.  It can be shown this defines a symplectic form that is independent of the choice of coordinates. 

We can also obtain this form without resorting to local coordinates.  Let $\pi: T^*M \rightarrow M$ be the natural projection, let $q \in M$, and let $p \in T_q^* M$.  There exists a canonical $1$-form $\Theta$ on $T^*M$, which can be defined fiberwise as 
\[
\Theta(v) = p(d_{(q,p)}\pi(v)),
\]
where $d_{(q,p)}\pi: T_{(q,p)} (T^*M) \rightarrow T_q M$ is the induced mapping of tangent spaces and $v \in T_{(q,p)}^* (T^*M)$.  Define the symplectic form by $\omega  = d\Theta$.  In the above local coordinates, this form can be written as $\Theta = \sum_i p_idq_i$, which matches the description of $\omega$ given above.
\end{enumerate}
\end{exa}

Let $G$ be a Lie group acting on a $C^{\infty}$-manifold $M$.  This induces an action of $G$ on $T^*M$.  The \emph{moment map} corresponding to this action is the mapping $\mu: T^*M \rightarrow \fg^*$ defined by 
\[
\mu(q,p)(\xi) := p(X_{\xi}(q)) = \Theta(X_{\xi})(q,p), 
\]
where $(q,p) \in T^*M$, $\xi \in \fg$, and $X_{\xi}(q)$ is the vector field corresponding to the infinitesimal action of $\xi$ on $M$ evaluated at $q$.  

\begin{exa}
\label{vspacesymex}
\leavevmode
\begin{enumerate}
\item Let $M = \R^n$.  In this case, the cotangent bundle $T^*M$ is just $\R^n \times (\R^n)^{\vee}$.  We  can use the symplectic form on $\R^{2n}$ in the previous example to get that (on a tangent space to $T^*M$) we can compute
\[
\omega ((x,x^*),(y, y^*)) = y^*(x) - x^*(y).
\]
Furthermore, if $\GL(n, \R)$ acts on $M$ in the standard way, then a simple computation shows that for any $A \in \gl_n (\R)$ we have 
\[
\mu(x, x^*)(A) = x^*(Ax).
\]
\item Let us consider a variant of the previous example where $M = \Hom(\R^n, \R^m)$ and the group $G = \GL(m, \R) \times \GL(n, \R)$ acts on $M$ by  $(A,B) \cdot P = AP\inv{B}$.  As above, we can compute that the moment map $\mu: T^*M \rightarrow \fg^{*}$ with respect to the induced action of $G$ on $T^*M = \Hom(\R^n, \R^m) \times \Hom(\R^n, \R^m)^{\vee}$ is equal to
 \[
\mu(P, P^*)(A,B) = P^*(AP - PB).
\]
\end{enumerate}
\end{exa}

Note that the above definitions of symplectic manifold and moment map can be transferred both to holomorphic manifolds and to (smooth) algebraic varieties.  For details see Chapter 1 of \cite{CG1997}.
	
We can use the  moment map to construct further symplectic manifolds via the following classic theorem (see e.g. Section 5.4 in \cite{McDS2017}):

\begin{thm}
\label{hamquotient}
Let $M$ be a $C^{\infty}$-manifold with a free and proper action (i.e. the preimages of compact sets under the mapping $G \times M \rightarrow M \times M$ induced by the $G$-action are compact) of a Lie group $G$.  The following is true:
\begin{enumerate}
\item If $\cO \subset \fg^{*}$ is a coadjoint orbit under the action of $G$, then the quotient $\inv{\mu}(\cO)/G$ has a natural structure of a symplectic manifold.
\item If $\cO = 0$, then the inclusion $\inv{\mu}(0) \hookrightarrow T^*X$ induces a symplectomorphism $T^*(X/G) \cong \inv{\mu}(0)/G$. 
\end{enumerate}
\end{thm}

\begin{rmk}
Since we are only interested in working with cotangent bundles, the definition of the moment map simplifies significantly.  In general, given a symplectic manifold $(M, \omega)$ on which a Lie group $G$ acts by symplectomorphisms, the moment map $\mu: M \rightarrow \fg^*$ is defined by 
\[
d(\mu(\xi))  = \iota_{X_{\xi}} \omega,
\]    
where $\iota_{X_{\xi}} \omega$ is the contraction of $\omega$ by the vector field $X_{\xi}$.  In order for the mapping $\mu(\xi): M \rightarrow \R$ to exist and for it to have nice properties with respect to a change of $\xi \in \fg$, the action of $G$ on $M$ must be \emph{Hamiltonian}.

If the moment map exists, then for a coadjoint $G$-orbit $\cO$, the symplectic quotient $\inv{\mu}(\cO)/G$ described in Theorem \ref{hamquotient} exists under the conditions that the points of $\cO$ are all regular values and the induced action of $G$ on $\inv{\mu}(\cO)$ is free and proper.  For further details concerning this construction see Chapter 5 in \cite{McDS2017}.  
\end{rmk}

\subsection{Hamiltonian reduction}
Let $G$ be a connected reductive group acting on a smooth affine variety $X$.  The action of $G$ lifts to an action on $T^*X$.  This defines an action of $G$ on the algebra of regular functions $\cO(T^*X)$.  We can write this as a morphism $G \rightarrow \Aut(\cO(T^*X))$.  Differentiating this morphism gives us a homomorphism of Lie algebras $\fg \rightarrow \textrm{Der}(\cO(T^*X)) = \cT(T^*X)$ into vector fields on $T^*X$, which defines an action of the Lie algebra $\fg$ of $G$ by derivations on $\cO(T^*X)$.  

That is, this morphism assigns to each $\xi \in \fg$ a vector field $X_{\xi}$ on $T^*X$.  Thus, we can define the moment map $\mu: T^*X \rightarrow \fg^*$ as before by
\[
\mu(q,p)(\xi) := p(X_{\xi}(q)) = \Theta(X_{\xi})(q,p), 
\] 
where $(q,p) \in T^*X$, $\xi \in \fg$, and $X_{\xi}(q)$ is the vector field corresponding to the infinitesimal action of $\xi$ on $X$ evaluated at $q$.  

The moment map is equivariant with respect to the induced $G$ action on $T^*X$ and the coadjoint action of $G$ on $\fg^*$ (see Lemma 1.4.2 in \cite{CG1997}).  This implies that for any $\lambda \in \fg^*$ fixed by the coadjoint action of $G$ the subvariety $\inv{\mu}(\lambda)$ is $G$-invariant.  We may therefore consider its affine and Proj GIT quotients with respect to the action of $G$.  

Unfortunately, the resulting varieties need not always be symplectic.  However, if we make additional assumptions about the action of $G$, we can obtain a result similar to Theorem \ref{hamquotient} in the algebraic setting.  In order to do this, we recall another definition from symplectic geometry.

\begin{defn}
\label{poissondef}
A \textbf{Poisson variety} consists of a pair $(X, \{-,-\})$ such that $X$ is an algebraic variety and $\{-,-\}: \cO_X \times \cO_X \rightarrow \cO_X$ is a $\C$-bilinear morphism defined on the corresponding sheaf of regular functions satisfying
\begin{enumerate}
\item $\{f_1,f_2\} = - \{f_2,f_1\}$,
\item $\{f_1, \{f_2,f_3\}\} + \{f_2, \{f_3,f_1\}\} + \{f_3, \{f_1,f_2\}\} = 0$,
\item $\{f_1, f_2f_3\} = f_3 \{f_1, f_2\}+ f_2\{f_1, f_3\}$, 
\end{enumerate}
for all $f_1, f_2, f_3 \in \cO_X$.  In other words, the sheaf regular functions $\cO(X)$ is a sheaf of \textbf{Poisson algebras}.
\end{defn}

Given a (smooth) affine symplectic variety $(X, \omega)$, then $X$ has a natural structure of a smooth affine Poisson variety.  Indeed, the pairing given by the form $\omega$ defines a canonical isomorphism $TX \cong T^*X$ between the tangent and cotangent bundles of $X$.  This means for any $f \in \cO(X)$ corresponds to a vector field $X_f \in \cT(X)$ (called a \emph{Hamiltonian vector field}) via the pairing $\omega(- , X_f) = df$.  This defines a bilinear bracket $\{-,-\}: \cO(X) \times \cO(X) \rightarrow \cO(X)$ by $\{f_1,f_2\} = \omega(X_{f_1}, X_{f_2})$.  It can be shown (\cite{CG1997} 1.2.7) that this bracket satisfies the conditions of Definition \ref{poissondef} and makes $X$ into a Poisson variety.

Note that the moment map can be seen as a mapping $\fg \rightarrow \cO(T^*X)$ that sends $\xi$ to the function $H_\xi(q,p) := \mu(q,p)(\xi) = \Theta(X_{\xi})(q,p)$ called the \emph{Hamiltonian} associated to $\xi$.  This mapping is a homomorphism of Lie algebras, where the $\fg$ is equipped with the standard Lie bracket and $\cO(T^*X)$ is equipped with the Poisson bracket.  Furthermore, there is a commutative diagram of Lie algebra homomorphisms 

\begin{equation*}
\begin{tikzcd}[row sep=10pt, column sep=10pt, every label/.append style={font=\small}]
 & & & \fg \arrow[ddddlll,""] \arrow[ddddrrr,""]  & & & \\
& & & \xi \arrow[ddl, mapsto] \arrow[ddr, mapsto] & & &\\
& & &  &\\
& & H_{\xi} \arrow[rr, mapsto] & & X_{H_{\xi}} = X_{\xi} & &\\
\cO(T^*M) \arrow[rrrrrr, ""] & & & & & & \cT(T^*X)
\end{tikzcd}
\end{equation*}
This implies $\{H_{\xi}, f\} = X_{\xi}f$ for any $f \in \cO(T^*X)$ (see Section 1.4 in \cite{CG1997} for details).  We will need the next lemma in order to define a symplectic structure on the affine GIT quotient $\inv{\mu}(\lambda)//G$.  

\begin{lmm}
\label{momentimlmm}
Let $X$ be a smooth affine variety with the action of an affine algebraic group $G$.  Let  $\gamma = (q,p) \in T^*X$ and let $d_{\gamma}\mu: T_{\gamma}(T^*X) \rightarrow \fg^*$ be the differential of the moment map.  

Let $\fg_{\gamma}$ be the Lie algebra of the stabilizer $G_{\gamma}$.  The image of $d_{\gamma}\mu(T_{\gamma}(T^*X))$ is the annihilator $\fg_{\gamma}^{\perp}$ of $\fg_{\gamma}$.
\end{lmm}
\begin{proof}
Let $v \in T_{\gamma}(T^*X)$ and $\xi \in \fg$, then considering $\xi$ as an element of $(\fg^{*})^*$, we can compute the following based on the definition of the moment map
\[
(d_{\gamma}\mu(v))(\xi) = d_{\gamma}(\xi \circ \mu)(v) = \omega(X_{\xi}(\gamma), v)
\]
If $\xi \in \fg_\gamma$, then $X_{\xi}(\gamma) = 0$, so $d_{\gamma}\mu(v) \in \fg_{\gamma}^{\perp}$.  

Furthermore, we see that $\ker d_{\gamma} \mu$ is the orthogonal complement of $T_{\gamma}(G \cdot \gamma)$ with respect to the symplectic form $\omega$.  Therefore, we can compute the dimension
\[
\dim \Im d_{\gamma} \mu = \dim T_{\gamma}(T^*X) - \dim \ker d_{\gamma} \mu = \dim T_{\gamma}(G \cdot \gamma) = \dim \fg - \dim \fg_{\gamma} = \dim \fg_{\gamma}^{\perp}.
\]
It follows that $\Im d_{\gamma}\mu =  \fg_{\gamma}^{\perp}$.
\end{proof}

We now have the following theorem.	

\begin{thm}
\label{affinehamred}
Let $G$ be a connected reductive group acting on a smooth affine variety $X$.  Let $\mu: T^*X \rightarrow \fg^*$ be the corresponding moment map, and let $\lambda \in \fg^*$ be a fixed point of the coadjoint action of $G$.  The following is true:
\begin{enumerate}
\item The affine GIT quotient $\inv{\mu}(\lambda)//G$ has a natural structure of a Poisson variety induced by the Poisson structure on $T^*X$.
\item If the action of $G$ on $\inv{\mu}(\lambda)$ is free, then $\inv{\mu}(\lambda)//G$ is a smooth variety with a symplectic structure induced by the symplectic form on $T^*X$.  In this case, we have
\[
\dim \inv{\mu}(\lambda)//G = 2 \dim X - 2 \dim G.
\]
\end{enumerate}
\end{thm}
\begin{proof}
Let $I \subset \cO(T^*X)$ be the ideal that cuts out $\inv{\mu}(\lambda)$ as a closed subvariety of $T^*X$.  Let us define a Poisson bracket on $\cO(\inv{\mu}(\lambda))^G = (\cO(T^*X)/I)^G$ using the Poisson bracket on $\cO(T^*X)$ by $\{[f], [h]\} = [\{f,h\}] \in \cO(T^*X)/I$, where $[f]$ indicates the element represented by $f$ in the quotient.  To show this is well-defined we must demonstrate that $\{f, I\} \subset I$ for $f \in \cO(T^*X)$ such that $[f] \in \cO(T^*X)/I$ is $G$-invariant, and that the bracket is $G$-invariant.

Note that by definition $I$ is generated by regular functions $H_{\xi} - \lambda \Id$ where $\xi \in \fg$.  We can also compute $\{f, H_{\xi} - \lambda \Id\} = \{f, H_{\xi}\} = X_{\xi}f \in I$, since $[f]$ is $G$-invariant.  Now, we have $g \cdot\{[f], [h]\} = [\{g\cdot f,g \cdot h\}] = [\{f,h\}]$, since $[f], [h]$ are $G$-invariant.  Therefore, the bracket is well-defined.  It is not hard to see it satisfies the properties of a Poisson bracket, so $\inv{\mu}(\lambda)//G$ has the structure of a Poisson variety induced by that on $T^*X$.

Let $\gamma = (q,p) \in \inv{\mu}(\lambda)$, and consider the differential morphism $d_{\gamma}\mu: T_{\gamma}(T^*X) \rightarrow \fg^*$.  Note that $\gamma$ is a nonsingular point of $\inv{\mu}(\lambda)$ if and only if  $d_{\gamma}\mu$ is surjective.  By Lemma \ref{momentimlmm}, the latter occurs if and only if $\fg_{\gamma} = 0$ .  Since the action of $G$ is free, the point $\gamma$ is nonsingular.  Thus, by Theorem \ref{projbun} we have that $\inv{\mu}(\lambda)//G$ is smooth and $\inv{\mu}(\lambda) \rightarrow \inv{\mu}(\lambda)//G$ is a principal $G$-bundle in the \'{e}tale topology.  It follows that $(T_{[\gamma]}\inv{\mu}(\lambda))//G$ is isomorphic to the quotient space $(T_{\gamma}\inv{\mu}(\lambda))/\fg$, where $[\gamma]$ is the image of $\gamma \in \inv{\mu}(\lambda)$ in $\inv{\mu}(\lambda)//G$ and the tangent space of the orbit $G \cdot \gamma$ is identified with $\fg$ .

We can define a $2$-form $\hat{\omega}$ on $\inv{\mu}(\lambda)//G$ by 
\[
\hat{\omega}_{[\gamma]}([v], [w]) = \omega_{\gamma} (v, w),
\]
for $[v], [w] \in T_{[\gamma]}(\inv{\mu}(\lambda)//G) \cong (T_{\gamma}\inv{\mu}(\lambda))/\fg$.

We now check that this form is well-defined.  Indeed, let $v\in T_{\gamma}\inv{\mu}(\lambda)$ and $\xi \in \fg$. By the proof of Lemma \ref{momentimlmm} we see that 
\[
\omega (v, X_{\xi}) = (d_{\gamma}\mu(v))(\xi).
\]
We have that $\ker d_{\gamma} \mu$ is precisely the orthogonal complement $\fg^{\omega}$ of $\fg$ in $T_{\gamma}(T^*X)$ with respect to the symplectic form (see proof of Lemma \ref{momentimlmm}).  Furthermore,  it is clear that $\ker d_{\gamma} \mu = T_{\gamma}\inv{\mu}(\lambda)$.  Therefore, it follows that $T_{\gamma}\inv{\mu}(\lambda) = \fg^{\omega}$.  

Thus, for $v, w \in T_{\gamma}\inv{\mu}(\lambda)$ and $\xi, \rho \in \fg$ we have
\begin{align*}
& \omega_{\gamma}(v + X_{\xi}(\gamma), w + X_{\rho}(\gamma)) = \omega_{\gamma}(v, w) + (d_{\gamma}\mu(w))(\xi)\\ &+ (d_{\gamma}\mu(v))(\rho) + \omega(X_{\xi}(\gamma), X_{\rho}(\gamma)) = \omega_{\gamma}(v, w).
\end{align*}
Furthermore, since the action of $G$ preserves the symplectic form $\omega$, then $\hat{\omega}$ is well-defined.  It is also closed because it is induced by $\omega$, which is closed.

To show $\hat{\omega}$ is nondegenerate suppose there is a $w \in T_{\gamma}\inv{\mu}(\lambda)$ such that $\omega_{\gamma}(v, w) = 0$ for all $v \in T_{[\gamma]}\inv{\mu}(\lambda)//G$.  As demonstrated above $T_{\gamma}\inv{\mu}(\lambda)$ and $\fg$ are orthogonal complements with respect to $\omega$, so $w \in \fg$ and consequently $[w] = 0$.  Consequently, the $2$-form $\hat{\omega}$ defines a symplectic structure on $\inv{\mu}(\lambda)//G$.  Since $\dim T^*X = 2 \dim X$, $\ker d_{\gamma} \mu = T_{\gamma}\inv{\mu}(\lambda)$, and $d_{\gamma} \mu$ is surjective, then we have 
\[
\dim \inv{\mu}(\lambda)//G = \dim T^*X - \dim \Im d_{\gamma} \mu - \dim G = 2 \dim X - 2 \dim G.
\]
\end{proof}

Note that the proof of Theorem \ref{projgood} given in \cite{MFK1994} (see also Theorem 3.14 in \cite{New1978}) implies that $\inv{\mu}(\lambda)$ can be covered by invariant open affine subvarieties $\{U_i\}$ such that $\inv{\mu}(\lambda)//_{\chi}G$ is covered by affine subvarieties $\{U_i//_{\chi} G\}$ .  Restricting the Poisson bracket (resp. symplectic form), means Theorem \ref{affinehamred} holds for each $U_i$, and the corresponding Poisson brackets (resp. symplectic forms) agree on the intersections.
Gluing over the open cover $\{U_i \rightarrow U\}$ leads to the following theorem:

\begin{thm}
\label{alghamred}
Let $G$ be a connected reductive group acting on a smooth affine variety $X$.  Let $\mu: T^*X \rightarrow \fg^*$ be the corresponding moment map.  Let $\lambda \in \fg^*$ be a fixed point of the coadjoint action of $G$, and let $\chi: G \rightarrow \C^{\times}$ be an algebraic group character.  The following is true:
\begin{enumerate}
\item The Proj GIT quotient $\inv{\mu}(\lambda)//_{\chi}G$ has a structure of a Poisson variety induced by the Poisson structure on $T^*X$.
\item The geometric quotient $(\inv{\mu}(\lambda))^s_{\chi}/G$ is smooth with symplectic structure induced by the symplectic form on $T^*X$, and it has dimension $2 \dim X - 2 \dim G$.
\end{enumerate}
\end{thm}

\subsection{Double quivers and symplectic geometry}

Let us consider the constructions from the previous two sections in the context of quiver representations.  Let $Q$ be quiver and let $\opcat{Q}$ be the quiver with the same vertex set but with the source and target of each of the arrows reversed.  The quiver $\overline{Q}$ with vertex set $I_Q$ and arrow set $A_{\overline{Q}} = A_Q \sqcup A_{\opcat{Q}}$ is called the \emph{double} of $Q$.  It is easy to see that $\Rep(\overline{Q}, \alpha) \cong \Rep(Q, \alpha) \times  \Rep(\opcat{Q}, \alpha)$.  Furthermore, the trace of the product of matrices defines a perfect pairing 
\begin{align*}
& \tr: \Hom(\C^n, \C^m) \times \Hom(\C^m, \C^n) \rightarrow \C\\
& \tr(A,B) = \tr(AB) = \tr(BA).
\end{align*}
Thus there is a canonical isomorphism $\Rep(Q, \alpha) \cong \Rep(\opcat{Q}, \alpha)$, and consequently
\[
T^*\Rep(Q, \alpha) \cong \Rep(Q, \alpha) \times  \Rep(Q, \alpha)^{\vee} \cong \Rep(\overline{Q}, \alpha).
\]

For any arrow $a \in A_Q$ denote by $a^* \in A_{\opcat{Q}}$ the corresponding arrow in the opposite direction.  Following Example \ref{vspacesymex} in the holomorphic (or algebraic) setting, we see that $T^*\Rep(Q, \alpha)$ has a symplectic structure defined by the $2$-form 
\[
\omega (V_1, V_2) = \sum_{a \in A_Q} ((f_{a}^2)^*(f_{a}^1) - (f_{a}^1)^*(f_{a}^2)) = \sum_{a \in A_Q} (\tr(f_{a}^1f_{a^*}^2) - \tr(f_{a}^2f_{a^*}^1)),
\]
where $V_1 = (\{V_i^1\}, \{f_a^1\} \sqcup \{(f_a^1)^*\}), V_2 = (\{V_i^2\}, \{f_a^2\} \sqcup \{(f_a^2)^*\})$ are in $T^*\Rep(Q, \alpha)$ and $(f_a^l)^*$ is identified with $(f_{a^*}^l)$ for $l = 1,2$ via the trace pairing.

The standard action of the group $G(\alpha)$ gives us an induced action of $G(\alpha)$ on $T^*\Rep(Q,\alpha)$.  After the identification with $\Rep(\overline{Q}, \alpha)$ this $G(\alpha)$ can be seen to be the usual action by change of basis at each vertex.  

The Lie group $G(\alpha) = \prod_{i \in I_Q} \GL(\alpha_i, \C)$ has Lie algebra $\fg(\alpha) = \prod_{i \in I_Q} \gl_{\alpha_i}(\C)$.   Using the trace pairing once again we can identify $\fg(\alpha)^{*}$ with $\fg(\alpha)$.  This means the computation from Example \ref{vspacesymex} (once again done in the holomorphic or algebraic setting) results in the following moment map:
\begin{align*}
& \mu: \Rep(\overline{Q}, \alpha) \rightarrow \fg(\alpha)\\
& \mu(\overline{V}) = \left(\sum_{\substack{a \in A_Q\\ t(a) = i}} f_{a}f_{a^*} - \sum_{\substack{a \in A_Q\\ s(a) = i}} f_{a^*}f_{a} \right)_{i \in I_Q}.
\end{align*} 

Note that the formula for the moment map implies that if $C := (C_i) \in \fg(\alpha)$, the fiber $\inv{\mu}(C)$ is empty unless $\sum_i\tr (C_i) =  0$ (since the trace of each commutator is $0$).  This means we can think of the moment map as $\mu: \Rep(\overline{Q}, \alpha) \rightarrow \End(\alpha)_0$, where $\End(\alpha)_0 := \{ (C_i)_{i \in I_Q} \in \fg(\alpha)|\sum_i\tr(C_i) = 0 \}$.  Further note that the moment map formula implies that it is a morphism of algebraic varieties.  Consequently, the fibers of $\mu$ are closed subvarieties of $\Rep(\overline{Q}, \alpha)$. 

If $\lambda = (\lambda_i)_{i \in I_Q} \in \C^{I_Q}$, then we can define an element $(\lambda_i \Id)_{i \in I_Q} \in \fg(\alpha)$, which we will also denote by $\lambda$.  We will use the following definition to give a different interpretation to the fiber $\inv{\mu}(\lambda)$. 

\begin{defn}
Let $Q$ be a quiver, and let $\lambda \in \C^{I_Q}$. The quotient of the path algebra of the double $\overline{Q}$ given by $\Pi_{\lambda}(Q) := \C \overline{Q}/ J$, where $J$ is the (two-sided) ideal generated by the elements
\[
\left(\sum_{\substack{a \in A_Q\\ t(a) = i}} e_{a}e_{a^*} - \sum_{\substack{a \in A_Q\\ s(a) = i}} e_{a^*}e_{a} \right) -  \lambda_i e_i,
\]
for $i \in I_Q$, is called the \textbf{deformed preprojective algebra} of $Q$ associated to $\lambda$.  If $\lambda = 0$, then $\Pi_0(Q)$ is called the \textbf{preprojective algebra} of $Q$.
\end{defn}

Let $\Rep(\Pi_{\lambda}(Q), \alpha)$ be representations of $\Pi_{\lambda}(Q)$ with dimension vector $\alpha$ (i.e. these are representations of $\C \overline{Q}$ such that elements from $J$ act trivially). For $\lambda \cdot \alpha = 0$, we get $\inv{\mu}(\lambda)$ is nonempty, and we can see that by definition there is a bijection between $\Rep(\Pi_{\lambda}(Q), \alpha)$ and $\inv{\mu}(\lambda)$.  It is not hard to check this is an isomorphism of algebraic varieties.  Furthermore, since $\inv{\mu}(\lambda)$ is stable under the action of $G(\alpha)$, the variety $\Rep(\Pi_{\lambda}(Q), \alpha)$ has an induced action of $G(\alpha)$ making the isomorphism $G(\alpha)$-equivariant.

\begin{rmk}
\label{preprojrmk}
Since $\Rep(\Pi_{\lambda}(Q), \alpha)$ is isomorphic to a subvariety of $\Rep(\overline{Q}, \alpha)$, then we can define $\theta$-stability for representations of a deformed preprojective algebra.  Namely, given $\theta \in \Z^{I_Q}$ such that $\theta \cdot \alpha = 0$, we say $V \in \Rep(\Pi_{\lambda}(Q), \alpha)$ is $\theta$-semistable (resp. $\theta$-stable) if and only if for any (resp. any nontrivial, proper) subrepresentation $W \subset V$ with dimension vector $\beta$ we have $\theta \cdot \beta \le 0$ (resp. $\theta \cdot \beta < 0$).

Many of the results concerning $\theta$-stability quiver representations hold for representations of $\Pi_{\lambda}(Q)$.  For example, analogues of Theorems \ref{kingstab}, \ref{quiverpoly}, \ref{quivermodulic}, and \ref{quivermodulif} are all true.  In fact, even stronger versions exist for finite-dimensional representations of finite-dimensional algebras (see Theorem 4.1, Proposition 4.2, Proposition 5.2, and Proposition 5.3 in \cite{King1994}). 
\end{rmk}

\begin{exa}
\label{doublejordanex}
Consider the Jordan quiver $J$.  Its double $\overline{J}$ is pictured below.  

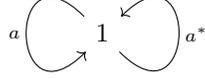
\begin{equation*}
\begin{tikzcd}
1 \arrow[out=135,in=225,loop, swap, "a"] \arrow[out=315,in=45,loop, swap, "a^*"]
\end{tikzcd}
\end{equation*}

Identifying the cotangent bundle $T^*\Rep(J, \alpha)$ with representations of the double $\Rep(\overline{J}, \alpha)$ and $\fg(\alpha)^*$ with $\fg(\alpha)$, we see that the moment map $\mu: \Rep(\overline{J}, \alpha) \rightarrow \fg(\alpha)$ is given by 
\[
\mu(f_a, f_{a^*}) = f_af_{a^*} - f_{a^*}f_a = [f_a, f_{a^*}].
\]
Note that if $\lambda \in \C$, the fiber of the moment map $\inv{\mu}(\lambda)$ is empty unless $\lambda = 0$.  Therefore, the only nonempty fiber $\inv{\mu}(0)$ can be thought of as pairs of commuting $\alpha \times \alpha$ matrices.  Alternatively, we can think of this fiber as $\alpha$-dimensional representations of the preprojective algebra $\Pi_0(J) = \C[e_a,e_{a^*}]$.

Introducing a stability parameter $\theta \in \Z$, we see that it must necessarily be equal to $0$.  Therefore, every representation in $\inv{\mu}(0)$ is $\theta$-semistable, while only $1$-dimensional representations are $\theta$-stable (any pair of commuting square matrices has a common eigenvector).  Since the points of $\inv{\mu}(0)//_{\theta}G(\alpha)$ correspond to $\theta$-polystable representations of $\Pi_0(J)$, we see that it parametrizes pairs of diagonal matrices (up to simultaneous permutation of the diagonal elements). 
\end{exa}
\section{Nakajima quiver varieties}

\subsection{Definition and basic properties}
We now combine framing for quivers with Hamiltonian reduction.  Namely, let $Q$ be a quiver and $Q^{fr}$ its framing.  We can identify $T^*\Rep(Q^{fr}, \alpha, \alpha')$ with representations of the double of the framed quiver $\Rep(\overline{Q^{fr}}, \alpha, \alpha')$ via the trace pairing.  Under this identification, the action of the group $G(\alpha)$ on $\Rep(Q^{fr}, \alpha, \alpha')$ naturally lifts to an action on $\Rep(\overline{Q^{fr}}, \alpha, \alpha')$.  This can be written explicitly as 
\[
g \cdot \overline{V^{fr}} = (\{V_i\}, \{V'_{i}\}, \{g_{t(a)}f_a \inv{g}_{s(a)}\}, \{g_{t(a)}f_{a^*}\inv{g}_{s(a)}\}, \{f_{b_i}\inv{g_i}\},\{g_i f_{b_i^*}\}),
\] 
where $\overline{V^{fr}} = (\{V_i\}, \{V'_{i}\}, \{f_a\}, \{f_{a^*}\}, \{f_{b_i}\}, \{f_{b_i^*}\})$.  The corresponding moment map $\mu: \Rep(\overline{Q^{fr}}, \alpha, \alpha') \rightarrow \fg(\alpha)$ is given by 
\[
\mu \left(\overline{V^{fr}} \right) \mapsto \left(-f_{b_i^*}f_{b_i} +  \sum_{\substack{a \in A_Q\\ t(a) = i}} f_{a}f_{a^*} - \sum_{\substack{a \in A_Q\\ s(a) = i}} f_{a^*}f_{a} \right)_{i \in I_Q}.
\]

We can now state the following definition due to Nakajima (see \cite{Na1994}):
\begin{defn}
Let $\theta \in \Z^{I_Q}$ and $\lambda \in \C^{I_Q}$.  The algebraic variety $\fM_{\theta, \lambda}(\alpha, \alpha') := \inv{\mu}(\lambda)//_{\chi_{\theta}} G(\alpha)$ is called a \textbf{quiver variety}.  We denote by $\fM_{\theta, \lambda}^s(\alpha, \alpha')$ the geometric quotient corresponding to the $\chi_{\theta}$-stable points. 
\end{defn}

We would like to characterize $\chi_{\theta}$-semistable points in $\inv{\mu}(\lambda)$ in a similar fashion to Theorem \ref{framedthm}.  Recall that by Proposition \ref{framedeqprp} there is a $G(\alpha)$-equivariant isomorphism between $\Rep(Q^{fr}, \alpha, \alpha')$ and $\Rep(\hat{Q}, \hat{\alpha})$, which sends each matrix assigned to an arrow going into a framing vertex to its rows.  Let $\widetilde{Q}$ be the double of the framed quiver $\hat{Q}$.  The action of $G(\alpha)$ on $\Rep(\widetilde{Q}, \hat{\alpha})$ lifts from that on $\Rep(\hat{Q}, \hat{\alpha})$.  Therefore, as before, we can construct a $G(\alpha)$-equivariant isomorphism $\Phi: \Rep(\overline{Q^{fr}}, \alpha, \alpha') \rightarrow \Rep(\widetilde{Q}, \hat{\alpha})$.  Namely, if $\overline{V^{fr}} \in \Rep(\overline{Q^{fr}}, \alpha, \alpha')$ is as above, then $\Phi$ maps it to $\widetilde{V} = (\{V_i\}, \{V_{\infty}\}, \{f_a\}, \{f_{a^*}\}, \{r_{ik}\}, \{c_{ik}\})$, where $V_{\infty} = \C$, $r_{ik}$ are the rows of $f_{b_i}$, and $c_{ik}$ are the columns of $f_{b_i^*}$.

The moment map $\hat{\mu}: \Rep(\widetilde{Q}, \hat{\alpha}) \rightarrow \fg(\alpha)$ is given by 
\[
\hat{\mu} \left(\overline{V^{fr}} \right) \mapsto \left(- \sum_k c_{ik}r_{ik} +  \sum_{\substack{a \in A_Q\\ t(a) = i}} f_{a}f_{a^*} - \sum_{\substack{a \in A_Q\\ s(a) = i}} f_{a^*}f_{a} \right)_{i \in I_Q}.
\]
Thus, it is not hard to see that given a $\lambda \in \C^{I_Q}$ there is a $G(\alpha)$-equivariant isomorphism between the varieties $\inv{\mu}(\lambda)$ and $\inv{\hat{\mu}}(\lambda)$ induced by $\Phi$.  In light of Remark \ref{preprojrmk}, we can use this isomorphism to describe semistable elements of $\inv{\mu}(\lambda)$ as follows:

\begin{thm}
\label{nakajimastab}
Let $\theta \in \Z^{I_Q}$ and let $\lambda \in \C^{I_Q}$ be such that $\inv{\mu}(\lambda)$ is nonempty.  The representation $\overline{V^{fr}} = (\{V_i\}, \{V'_{i}\}, \{f_a\}, \{f_{a^*}\}, \{f_{b_i}\}, \{f_{b_i^*}\}) \in \inv{\mu}(\lambda)$ is $\chi_{\theta}$-semistable with respect to the action of $G(\alpha)$ if and only if for any subrepresentation $W \subset \overline{V} = (\{V_i\}, \{f_a\}, \{f_{a^*}\})$ with dimension vector $\beta$ the following conditions hold:
\begin{itemize}[label={}]
\item if $W_i \subset \ker f_{b_i}$ for all $i \in I_Q$, then $\theta \cdot \beta \le 0$,
\item if $W_i \supset \Im f_{b_i^*}$ for all $i \in I_Q$, then $\theta \cdot \beta \le \theta \cdot \alpha$.
\end{itemize} 
The representation $\overline{V^{fr}}$ is $\chi_{\theta}$-stable if the above inequalities are strict with the additional conditions that $W$ is nontrivial in the first case and $W$ is proper in the second case. 
\end{thm}  
\begin{proof}
We proceed as in the proof of Theorem \ref{framedthm}.  Let $\hat{\theta} = (\theta, - \sum_{i \in I_Q} \theta_i \alpha_i)$, and let $\widetilde{V} = (\{V_i\}, \{V_{\infty}\}, \{f_a\}, \{f_{a^*}\}, \{r_{ik}\}, \{c_{ik}\}) = \Phi(\overline{V^{fr}})$.  Note that $\inv{\hat{\mu}}(\lambda)$ is isomorphic to the corresponding variety of representations of the deformed preprojective algebra $\Rep(\Pi_{\lambda}, \hat{\alpha})$.

Therefore, by Remark \ref{preprojrmk}, we have that $\widetilde{V}$ is $\chi_{\hat{\theta}}$-semistable with respect to the action of $G(\hat{\alpha})$ on $\Rep(\Pi_{\lambda}, \hat{\alpha})$ if and only if it is $\hat{\theta}$-semistable.  Furthermore, we have that $\widetilde{V}$ is $\chi_{\hat{\theta}}$-semistable if and only if it is $\chi_{\theta}$-semistable with respect to the induced $G(\alpha)$ action.  

Thus, $\widetilde{V}$ is $\chi_{\hat{\theta}}$-semistable if and only if for every subrepresentation $\widetilde{W} \subset \widetilde{V}$ (with dimension vector $\hat{\beta}$) we have
\[
\hat{\theta} \cdot \hat{\beta}  = \sum_i \theta_i\beta_i - \beta_{\infty}\sum_i \theta_i \alpha_i \le 0.
\]

Let $\widetilde{W} \subset \widetilde{V}$ be a subrepresentation. Note that $\beta_{\infty}$ can only be $0$ or $1$.  If it is $0$, then the subrepresentation $W \subset V$ induced by $\widetilde{V}$ must satisfy $W_i \subset \bigcap_k r_{ik} = \ker f_{b_i}$ for all $i \in I_Q$.  It follows that $\widetilde{V}$ is $\chi_{\hat{\theta}}$-semistable if $\theta \cdot \beta \le 0$.  If $\beta_{\infty} = 1$, then $W$ must satisfy $W_i \supset \sum_k \Im c_{ik} = \Im f_{b_i^*}$, and so $\widetilde{V}$ is $\chi_{\hat{\theta}}$-semistable if $\theta \cdot \beta \le \theta \cdot \alpha$.

Conversely, suppose $\widetilde{V}$ is $\chi_{\hat{\theta}}$-semistable.  If $W \subset V$ is a subrepresentation (with dimension vector $\beta$) such that $W_i \subset \ker f_{b_i}$ for all $i \in I_Q$, then we can complete it to a subrepresentation $\widetilde{W} \subset \widetilde{V}$ by setting $\beta_{\infty} = 0$, $r_{ik} = 0$, and $c_{ik} = 0$.  Semistability implies that $\theta \cdot \beta \le 0$.

If $W_i \supset \Im f_{b_i^*}$ for all $i \in I_Q$, then we can complete $W$ to a subrepresentation $\widetilde{W} \subset \widetilde{V}$ by setting $\beta_{\infty} = 1$, $r_{ik}$ to the rows of $f_{b_i}$ restricted to $W_i$, and $c_{ik}$ to the columns of $f_{b_i^*}$.  Semistability implies that $\theta \cdot \beta \le \theta \cdot \alpha$.  Since the moment maps fibers $\inv{\mu}(\lambda)$ and $\inv{\hat{\mu}}(\lambda)$ are $G(\alpha)$-equivariantly isomorphic, this completes the proof of first part of the theorem.  The proof of the $\chi_{\theta}$-stable is entirely analogous.
\end{proof}

Keeping the same notation as in Theorem \ref{nakajimastab}, we obtain the following easy corollary:
\begin{cor}
\label{nakajimastabcor}
Let $\theta \in \Z_{>0}^{I_Q}$.  We have that $\overline{V^{fr}} \in \inv{\mu}(\lambda)$ is $\chi_{\theta}$-semistable if and only if any subrepresentation $W \subset \overline{V}$ such that $W_i \subset \ker f_{b_i}$ for all $i \in I_Q$ is trivial.  Similarly, if $\theta \in \Z_{<0}^{I_Q}$, then $\overline{V^{fr}}$ is $\chi_{\theta}$-semistable if and only if any subrepresentation $W \subset \overline{V}$ such that $W_i \supset \Im f_{b_i^*}$ for all $i \in I_Q$ is equal to $V$.
\end{cor}

Note that, as in Theorem \ref{framedthm}, choosing $\theta \in \Z_{>0}^{I_Q}$ guarantees that all semistable points in $\inv{\mu}(\lambda)$ are stable, making $\fM_{\theta, \lambda}(\alpha, \alpha') = \fM_{\theta, \lambda}^s(\alpha, \alpha')$.  However, this is true for other choices of the parameters $(\theta, \lambda)$.
  
Namely, let $(\alpha,\beta) := \langle \alpha, \beta \rangle + \langle \beta, \alpha \rangle$ be the symmetrization of the Euler-Ringel form.  Denote $R_{+}  := \{\gamma \in \Z^{I_Q}_{\ge 0}| (\gamma, \gamma) \le 2 \}/ \{0\}$.  For a dimension vector $\alpha$, denote $R_{+}(\alpha)  := \{\gamma \in R_{+}| \gamma_i \le \alpha_i \textrm{ for all } i \in I_Q \}$.  We say $(\theta, \lambda)$ is \emph{generic} with respect to $\alpha$ if $\theta \cdot \gamma \neq 0$ or $\lambda \cdot \gamma \neq 0$ for all $\gamma \in R_{+}(\alpha)$.  The following result is due to Nakajima (see Theorem 2.8 in \cite{Na1994}).

\begin{thm}
\label{nakajimafree}
Let $(\theta, \lambda)$ be generic with respect to $\alpha$. The action of $G(\alpha)$ on the $\chi_{\theta}$-semistable locus $\inv{\mu}(\lambda)_{\chi_{\theta}}^{ss}$ is free.
\end{thm}

Putting the results from the previous few sections together we obtain the following theorem.

\begin{thm}
\label{nakajimavarthm}
Let $\lambda \in \C^{I_Q}$ and $\theta \in \Z^{I_Q}$.  The quiver variety $\fM_{\theta, \lambda}(\alpha, \alpha')$ is a Poisson variety and the natural morphism $\fM_{\theta, \lambda}(\alpha, \alpha') \rightarrow \fM_{0, \lambda}(\alpha, \alpha')$ is a projective morphism of Poisson varieties (i.e. it preserves the bracket).  If $(\lambda, \theta)$ is generic with respect to $\alpha$, then $\fM_{\theta, \lambda}(\alpha, \alpha') = \fM_{\theta, \lambda}^s(\alpha, \alpha')$ and $\fM_{\theta, \lambda}(\alpha, \alpha')$ is a smooth symplectic variety of dimension $2\alpha \cdot \alpha' - 2q(\alpha)$.
\end{thm}
\begin{proof}
By Theorem \ref{alghamred}, the quiver variety $\fM_{\theta, \lambda}(\alpha, \alpha')$ has a Poisson structure.  The morphism $\fM_{\theta, \lambda}(\alpha, \alpha') \rightarrow \fM_{0, \lambda}(\alpha, \alpha')$ is projective by the discussion in Section \ref{projgit}.  It preserves the Poisson structure since both the bracket on $\fM_{\theta, \lambda}(\alpha, \alpha')$ and on $\fM_{0, \lambda}(\alpha, \alpha')$ are induced by that on $T^*\Rep(\overline{Q^{fr}}, \alpha, \alpha')$.  

By Theorem \ref{nakajimafree}, we have that any orbit of a $\chi_\theta$-semistable point of $\inv{\mu}(\lambda)$ has maximal dimension $\dim G(\alpha)$.  Consequently, by Proposition \ref{orbit} no such orbit can contain another such orbit in its closure.  This means, the orbit of every $\chi_{\theta}$-semistable point is closed, and this implies every such point is also $\chi_{\theta}$-stable.  Therefore, we have $\fM_{\theta, \lambda}(\alpha, \alpha') = \fM_{\theta, \lambda}^s(\alpha, \alpha')$  Furthermore, by Theorem \ref{alghamred}, we see that $\fM_{\theta, \lambda}(\alpha, \alpha')$ is a smooth variety.  It has dimension equal to 
\begin{align*}
& 2 \dim \Rep(\overline{Q^{fr}}, \alpha, \alpha') - 2 \dim G(\alpha)\\ &= 2(\alpha \cdot \alpha' + \sum_{a \in A_Q} \alpha_{s(a)}\alpha_{t(a)}) - 2 \sum_{i \in I_Q} \alpha_i^2  = 2\alpha \cdot \alpha' - 2q(\alpha),
\end{align*}
and a symplectic structure defined by the symplectic form on $\Rep(\overline{Q^{fr}}, \alpha, \alpha')$.
\end{proof}

\subsection{Example: The Jordan quiver}
\label{nakjordanex}
We conclude this section by presenting several examples of quiver varieties and discussing their geometric significance. The first of these concerns the quiver we have already seen in Section \ref{sheafp2}. Namely, let us consider the quiver varieties corresponding to the double of the framed Jordan quiver $\overline{J^{fr}}$ pictured below.

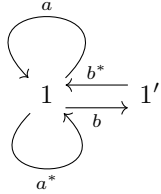
\begin{equation*}
\begin{tikzcd}
1 \arrow[out=45,in=135, loop, swap, "a"] \arrow[out=225,in=315, loop, swap, "a^*"] \arrow[r, shift right = 1.0ex, swap, "b"]  & 1' \arrow[l, shift right = 1.0ex, swap, "b^*"]
\end{tikzcd}
\end{equation*}

Let $\theta \in \Z_{<0}$ and let $\lambda \in \C$.  Using the moment map formula from the previous subsection, we obtain that $\inv{\mu}(\lambda) = \{(V_1, V_1', f_a,f_{a^*}, f_{b}, f_{b^*})| [f_a, f_{a^*}] - f_{b^*}f_{b} = \lambda \Id \}$.  Corollary \ref{nakajimastabcor} implies the $\chi_\theta$-semistable points are those that satisfy the additional condition that there is no proper subrepresentation $W \subset \{(V_1, f_a,f_{a^*}\}$ such that $W_1 \subset \Im f_{b^*}$.  Furthermore, the choice of $\theta$ means that $\chi_\theta$-semistable points are all $\chi_\theta$-stable.  

Setting $f_a = B_1, f_{a^*} = B_2, i = - f_{b^*}, j = f_{b}$ and $(\alpha, \alpha') = (n, r)$, we see that the quiver variety $\fM_{\theta, 0}(\alpha, \alpha')$ is exactly the variety $X$ considered in Theorem \ref{p2moduli}.  Therefore, $\fM_{\theta, 0}(\alpha, \alpha')$ is isomorphic to the moduli space of torsion free-sheaves on $\P^2$ trivial on the line $\l = \{[0:\lambda_1: \lambda_2] \in \P^2 \}$ together with a trivialization.

Let us consider a special case of this quiver variety when $\alpha' = 1$.  The stability condition gives us that $f_{b^*}(1)$ is a cyclic vector for $f_a$ and $f_{a^*}$ (meaning that by acting on $f_{b^*}(1)$ with $f_a$ and $f_{a^*}$ we can generate all of $\C^{\alpha}$).  We will be using the following fact from linear algebra (see Lemma 12.7 of \cite{EG2002}).

\begin{lmm}
\label{linalglmm}
Let $A,B \in \Hom(\C^{\alpha}, \C^{\alpha})$ such that $[A,B]$ has rank 1.  There exists a basis such that $A,B$ in that basis are both upper-triangular.
\end{lmm}
\begin{proof}
We may assume $A$ has a nontrivial kernel (otherwise replace $A$ with $A - \lambda \cdot \Id$, where $\lambda$ is an eigenvalue of $A$).  Let $v \in \ker A$. Since $BAv = ABv + [A,B]v$, if $\ker A \subset \ker [A,B]$, then $BAv = 0$. Therefore, we see $\ker A$ is invariant with respect to $B$.  If $\ker A \nsubseteq \ker [A,B]$, then there exists a nonzero $v' \in \ker A$ such that $ABv' = [A,B]v' \neq 0$.  Since $[A,B]$ has rank 1, this implies $\Im [A,B] \subset \Im A$.  Furthermore, for any $v \in \C^{\alpha}$ we have $BAv = ABv + [A,B]v$, so $\Im A$ is invariant with respect to $B$.

In either of the two cases, there exists a nonzero subspace $W \subset \C^{\alpha}$, which is invariant with respect to both $A$ and $B$.  Therefore, this subspace contains an eigenvector of both $A$ and $B$ simultaneously.  Taking the quotient of $\C^{\alpha}$ by the corresponding $1$-dimensional subspace allows us to obtain the desired basis by induction.
\end{proof}

Since $\lambda = 0$, the moment map equation gives us $[f_a, f_{a^*}] = f_{b^*}f_{b}$.  Let us assume the variables in $\C[x,y]$ are ordered so that powers of $x$ appear before powers of $y$ in each monomial (i.e. $x^{n_1} y^{n_2}$).   For any polynomial $p(x,y) \in \C[x,y]$ we have 
\[
\tr (p(f_a, f_{a^*})[f_a, f_{a^*}]) = \tr(p(f_a, f_{a^*})f_{b^*}f_{b}).
\]
Suppose, $f_b \neq 0$, so that $[f_a, f_{a^*}]$ has rank $1$.  By Lemma \ref{linalglmm},  $f_a$ and $f_{a^*}$ are upper-triangular matrices, while $[f_a, f_{a^*}]$ is a strictly upper triangular. Therefore, the trace in the left-hand side is equal to $0$.  The right-hand side is $\tr(p(f_a, f_{a^*})f_{b^*}f_{b})$ $= \tr(f_{b}p(f_a, f_{a^*})f_{b^*}) = 0$, so $f_{b}p(f_a, f_{a^*})f_{b^*} = 0$.  Since $f_{b^*}(1)$ is a cyclic vector, then this implies $f_{b}(v) = 0$ for any $v \in \C^{\alpha}$. Therefore, $f_b = 0$.  Thus, we can regard representations in $\inv{\mu}(0)^{ss}_{\chi_\theta}$ as finite-dimensional $\C[x,y]$-modules together with a cyclic vector.

Note that each triple $(f_a, f_{a^*}, f_{b^*})$ coming from an element of $\Rep(\overline{J^{fr}}, \alpha, 1)$ defines an ideal $I_{f_a,f_{a^*},f_{b^*}} := \{p \in \C[x,y]|p(f_a,f_{a^*})f_{b^*} = 0\} \subset \C[x,y]$.  

Consider the linear transformation $\C[x,y]/I_{f_a,f_{a^*},f_{b^*}} \rightarrow \C^{\alpha}$ defined by $p(x,y) \mapsto p(f_a,f_{a^*})f_{b^*}(1)$. We see that the ideal $I_{f_a,f_{a^*},f_{b^*}}$ has colength $\alpha$ in $\C[x,y]$ if and only if $f_{b^*}(1)$ is a cyclic vector.

Note that the ideal $I_{f_a,f_{a^*},f_{b^*}}$ is only determined up to $\GL(\alpha)$-action.  This means $(V_1, V_1', f_a,f_{a^*}, f_{b}, f_{b^*}) \mapsto I_{f_a,f_{a^*},f_{b^*}}$ gives a well-defined mapping between $\fM_{\theta, 0}(\alpha, 1)$ and set of ideals $I \subset \C[x,y]$ such that $\dim \C[x,y]/I \ = \alpha$.  The latter has a natural algebraic structure making it into the variety $\textrm{Hilb}^{\alpha}(\A^2)$ which parametrizes collections of $\alpha$ points in the plane (see e.g. \cite{Na1999}). 

Given a colength $\alpha$ ideal $I$ in $\textrm{Hilb}^{\alpha}(\A^2)$, there is a ring homomorphism $\C[x,y]\rightarrow \C[x,y]/I \simeq \C^{\alpha}$, which makes $\C^{\alpha}$ into a $\C[x,y]$-module. The image of $1$ in $\C^{\alpha}$ defines a cyclic vector for this module.  Therefore, we obtain an element of $\inv{\mu}(0)^{ss}_{\chi_\theta}$. Since identifying $\C[x,y]/I \simeq \C^{\alpha}$ requires a choice of basis, this element is determined up to $\GL(\alpha)$-action. This defines a mapping $\textrm{Hilb}^{\alpha}(\A^2) \rightarrow \fM_{\theta, 0}(\alpha, 1)$, which is the inverse of the one given above. Thus, there is a bijection between the sets of points of $\fM_{\theta, 0}(\alpha, 1)$ and of $\textrm{Hilb}^{\alpha}(\A^2)$. In fact, one can show they are isomorphic as varieties.

If $\theta = 0$, then the stability condition is trivial, so $\fM_{0, 0}(\alpha, 1) = \inv{\mu}(0)//\GL(\alpha, \C)$. Consequently, it suffices to describe the closed orbits of $\inv{\mu}(0) = \{(\C^{\alpha}, \C, f_a, f_{a^*}, f_b,$ $f_{b^*} )| [f_a, f_{a^*}] = f_{b^*}f_b\}$ under the action of $\GL(\alpha, \C)$ by change of basis.  

Let $V = (\C^{\alpha}, \C, f_a, f_{a^*}, f_b, f_{b^*})$ be such a representation, belonging to a closed orbit.  Let $W \subset \C^{\alpha}$ be the cyclic $\C[x,y]$-submodule generated by $f_{b^*}(1)$. By the discussion above, we see that $W \subset \ker f_b$, as well as $f_a(W) \subset W$, $f_{a^*}(W) \subset W$, $\Im f_{b^*} \subset W$.  Using a basis compatible with $W$, we can write: 
\[
f_a = \begin{pmatrix}A_{11} & A_{12}\\ 0 & A_{22}\end{pmatrix}, \quad f_{a^*} =  \begin{pmatrix}B_{11} & B_{12}\\ 0 & B_{22}\end{pmatrix}, \quad f_{b} =  \begin{pmatrix}0 & C\end{pmatrix}, \quad f_{b^*} =  \begin{pmatrix}D\\ 0\end{pmatrix}.
\]
Let $g_1(t) := \begin{pmatrix}\Id & 0\\ 0 & \inv{t} \cdot \Id \end{pmatrix}$ be a one-parameter subgroup of $\GL(\alpha, \C)$ expressed in the same basis (i.e. compatible with $W$). The action of $g_1(t) \subset \GL(\alpha, \C)$ on the above quiver representation is given by
\begin{align*}
&g_1(t)f_a \inv{g_1(t)} = \begin{pmatrix}A_{11} & t A_{12}\\ 0 & A_{22}\end{pmatrix}, \quad g_1(t) f_{a^*} \inv{g_1(t)} =  \begin{pmatrix}B_{11} & t B_{12}\\ 0 & B_{22}\end{pmatrix},\\ &
(f_{b})\inv{g_1(t)} =  \begin{pmatrix}0 & t C\end{pmatrix}, \quad g_1(t)(f_{b^*}) =  \begin{pmatrix}D\\ 0\end{pmatrix}.
\end{align*}
Since the orbit $\GL(\alpha, \C)\cdot V$ is closed, it must contain the limit point as $t \to 0$. Thus, it contains the representation given by
\[
f_a = \begin{pmatrix}A_{11} & 0\\ 0 & A_{22}\end{pmatrix}, \quad f_{a^*} =  \begin{pmatrix}B_{11} & 0\\ 0 & B_{22}\end{pmatrix}, \quad f_{b} =  0, \quad f_{b^*} =  \begin{pmatrix}D\\ 0\end{pmatrix}.
\]
Applying a similar argument to this representation using the one-parameter subgroup $g_2(t) = t \cdot \Id$, yields that $\GL(\alpha, \C)\cdot V$ must contain
\[
f_a = \begin{pmatrix}A_{11} & 0\\ 0 & A_{22}\end{pmatrix}, \quad f_{a^*} =  \begin{pmatrix}B_{11} & 0\\ 0 & B_{22}\end{pmatrix}, \quad f_{b} =  0, \quad f_{b^*} =  0.
\]
Since the framing is now $0$, we can use the arguments in Example \ref{doublejordanex} to see that $\fM_{0, 0}(\alpha, 1) = \A^\alpha \times \A^{\alpha}/S_n$.  

In fact, the natural morphism $\fM_{\theta, 0}(\alpha, 1) \rightarrow \fM_{0, 0}(\alpha, 1)$ can be described as the morphism $\textrm{Hilb}^\alpha(\A^2) \rightarrow \A^\alpha \times \A^{\alpha}/S_{\alpha} \cong (\A^2)^\alpha/S_{\alpha}$ that sends each colength $\alpha$ ideal in $\C[x,y]$ to the corresponding collection of $\alpha$ points in $(\A^2)^\alpha/S_{\alpha}$.  This is called the \emph{Hilbert-Chow morphism}, and it is a resolution of singularities for the variety  $(\A^2)^\alpha/S_{\alpha}$ (see Section 1.1 of \cite{Na1999} for details).

\subsection{Example: Quivers of type $A_n$}
\label{nakanex}
Now let us look at the quiver variety $\fM_{\theta, \lambda}(\alpha, \alpha')$ corresponding the double of the framed $A_n$ quiver.  Let $\theta \in \Z_{>0}^{n+1}$ and $\lambda = 0$.  Furthermore, let $\alpha'= (\alpha_0, 0, \cdots, 0)$ and $\alpha = (\alpha_1, \cdots, \alpha_n)$ be such that $\alpha_0 > \alpha_1 > \cdot > \alpha_n > 0$ holds.

\begin{equation*}
\begin{tikzcd}
\C^{\alpha_1} \arrow[r, shift left=1.5ex, "f_{a_1^*}"]  \arrow[d, shift right=2.0ex, swap, "f_{b_1}"] & \C^{\alpha_2} \arrow[r, shift left=1.5ex, "f_{a_2^*}"] \arrow[l, shift left=0.5ex, "f_{a_1}"] \arrow[d, shift right=1.0ex, ""] & \raisebox{+0.5ex}{$\cdots$} \arrow[r, shift left=1.5ex, "f_{a_{n-1}^*}"] \arrow[l, shift left=0.5ex, "f_{a_2}"] & \C^{\alpha_n} \arrow[l, shift left=0.5ex, "f_{a_{n-1}}"] \arrow[d, shift right=1.0ex, ""]\\
\C^{\alpha_0} \arrow[u, swap, "f_{b_1^*}"] & 0 \arrow[u, shift right=1.0ex, swap, ""] & \cdots & 0 \arrow[u, shift right=1.0ex, swap, ""]
\end{tikzcd}
\end{equation*}

  We would like to identify $\fM_{\theta, 0}(\alpha, \alpha')$ with the cotangent bundle of a flag variety.  To do so, consider the flag variety $\Fl := \Fl(\alpha_0, \cdots, \alpha_n, 0)$.  Recall that this variety is isomorphic to the quotient $\GL(\alpha_0, \C)/P$, where $P$ is the subgroup of block upper triangular matrices stabilizing the standard flag $0 \subset \C^{\alpha_n} \subset \cdots \subset \C^{\alpha_0}$.  By Lemma 1.4.9 in \cite{CG1997} we have that $T^*\Fl \simeq \GL(\alpha_0, \C) \times_{P} \fp^{\perp}$, where $\fp^{\perp}$ is the annihilator of the Lie algebra  $\fp$ of $P$ in $\fgl_{\alpha_0}^*(\C)$ ($P$ acts on $\GL(\alpha, \C)$ on the right and on $\fp^{\perp}$ by the coadjoint action).  

After identifying $\fgl_{\alpha_0}^*$ with $\fgl_{\alpha_0}$ using the trace pairing, we see that $\fp^{\perp}$ corresponds to the subspace $\fp_0 \subset \fgl_{\alpha_0}$ consisting of strictly block upper triangular matrices preserving the standard flag.  That is, if $f \in \fp_0$ we have $f(\C^{\alpha_i}) \subset \C^{\alpha_{i+1}}$ for $0 \le i \le n$ and $\alpha_{n+1} = 0$.  Given a flag $V_{\bullet} \in \Fl$, choose a representative of this flag $g \in \GL(\alpha_0, \C)$ (this is an element which sends the standard flag to $V_{\bullet}$).  Denote $\fp_{V_{\bullet}} := \textrm{Ad}(g)(\fp_0)$.  This is the subspace of strictly block upper triangular matrices that preserve $V_{\bullet}$.  Thus we get the following description of the cotangent bundle 
\[
T^*\Fl = \{(V_{\bullet}, f)| V_{\bullet} \in \Fl \textrm{ and } f \in \fp_{V_{\bullet}}\}. 
\]

Now consider the moment map $\mu: \Rep(\overline{A_n^{fr}}, \alpha, \alpha') \rightarrow \fg(\alpha)$ is given by 
\[
\mu(\overline{V^{fr}}) = (-f_{b_1^*}f_{b_1} + f_{a_1}f_{a_1^*}, f_{a_2}f_{a_2^*} - f_{a_1^*}f_{a_1} , \cdots, -f_{a_{n-1}^*}f_{a_{n-1}}),
\]
 for $\overline{V^{fr}} = (\{V_i\}, V_1', \{f_{a_i}\}, \{f_{a_i^*}\}, f_{b_1}, f_{b_1^*})$.  If $\overline{V^{fr}} \in \inv{\mu}(0)$, then the subspaces $(\ker f_{b_1}, \ker f_{a_1}, \cdots, \ker f_{a_{n-1}})$ define a subrepresentation $W$ of $(\{V_i\}, \{f_{a_i}\}, \{f_{a_i^*}\})$.  Indeed, $(-f_{b_1^*}f_{b_1} + f_{a_1}f_{a_1^*})(\ker f_{b_1}) = f_{a_1}f_{a_1^*}(\ker f_{b_1}) = 0$, so $f_{a_1^*}(\ker f_{b_1}) \subset \ker f_{a_1}$.  Proceeding in the same way we get that $f_{a_i^*}(\ker f_{a_{i-1}}) \subset \ker f_{a_i}$ for $2 \le i \le n-1$, making $W$ a subrepresentation.  By Corollary \ref{nakajimastabcor} an element of $\inv{\mu}(0)$ is semistable (and therefore stable) only if $W = 0$.  Consequently, $f_{b_1}, f_{a_1}, \cdots, f_{a_{n-1}}$ must all be injective.  Conversely, if these linear transformations are all injective, then any subrepresentation $W \subset (\{V_i\}, \{f_{a_i}\}, \{f_{a_i^*}\})$ such that $W_1 \in \ker f_{b_{1}}$ must be trivial, so $\overline{V^{fr}}$ is semistable (and therefore stable).  

Given $\overline{V^{fr}} \in \inv{\mu}(0)$ we can define a flag $V_{\bullet}$ in $\C^{\alpha_0}$ by  $V_i = \Im f_{b_1}f_{a_1} \cdots f_{a_{i}}$ and $f \in \fp_{V_{\bullet}}$ by $f = f_{b_1}f_{b_1^*}$.  It is not hard to see this extends to a mapping $\Psi: \fM_{\theta, 0}(\alpha, \alpha') \rightarrow T^*\Fl$.  This mapping has an inverse given by $(V_{\bullet}, f) \mapsto \overline{V^{fr}}$, where $f_{b_1}, f_{a_1}, \dots, f_{a_{n-1}}$
are given by the inclusions of the subspace in $V_{\bullet}$ and $f_{b_1^*} = f, f_{a_1^*} = f|_{V_{1}}, \dots, f_{a_{n-1}^*} = f|_{V_{n-1}}$.  It can be shown (see Theorem 7.3 in \cite{Na1994}) this bijection is an isomorphism of varieties.

Note that there is a different way to look at elements of $\fp_{V_{\bullet}}$.  Let $\hat{\nu} = (\hat{\nu}_i) = (\alpha_0-\alpha_1,  \dots, \alpha_{n-1} - \alpha_{n}, \alpha_n)$ and let $\nu_i = |\{\hat{\nu_l}| \hat{\nu}_l \ge i \}|$.  Note that the tuple $\nu = (\nu_i)$ is a partition of $\alpha_0$.  Therefore, it defines a conjugacy class $C_{\nu}$ of nilpotent matrices in $\fgl_{\alpha_0}(\C)$ (each part $\nu_i$ is the size of a Jordan block).  The following result (see e.g. Theorem 2.1 in \cite{CB2004} for proof) describes $\fp_{V_{\bullet}}$ for any flag $V_{\bullet} \in \Fl$. 

\begin{prp}
Let $\pi: T^*\Fl \rightarrow \fgl_{\alpha_0}(\C)$ be the projection onto the second component.  We have $\pi(T^*\Fl) = \overline{C_{\nu}}$.
\end{prp}

In the case that $\hat{\nu}$ is already a partition of $\alpha_0$, we can state a similar result due to Kraft and Procesi (see Section 3.3 in \cite{KP1979}), which gives a description of $\fM_{0, 0}(\alpha, \alpha')$.

\begin{thm}[Kraft and Procesi]
If the dimension vectors $\alpha$ and $\alpha'$ are such that $\alpha_0 - \alpha_1 \ge \cdots \ge \alpha_{n-1} - \alpha_n \ge \alpha_n$, then the affine GIT quotient $\fM_{0, 0}(\alpha, \alpha')$ is isomorphic to $\overline{C_{\nu}}$ and the quotient map is given by $f_{b_1}f_{b_1^*}$.
\end{thm}

If the conditions of the theorem hold, then the natural projective morphism $\fM_{\theta, 0}(\alpha, \alpha') \rightarrow \fM_{0, 0}(\alpha, \alpha')$ is given by $f_{b_1}f_{b_1^*}$, and it is just the projection $\pi: T^*\Fl \rightarrow \overline{C_{\nu}}$ onto the second component.  As in the Jordan quiver example, we get a well-known resolution of singularities.  It is called the \emph{Springer resolution}.

\begin{rmk}
\label{shmelkinrmk}
Note that if we only have $\alpha_0 > \alpha_1 > \cdots > \alpha_n$, rather than $\alpha_0 - \alpha_1 \ge \cdots \ge \alpha_{n-1} - \alpha_n \ge \alpha_n$, then $\fM_{0, 0}(\alpha, \alpha')$ need no longer be isomorphic to $\overline{C_{\nu}}$ and the morphism $\fM_{\theta, 0}(\alpha, \alpha') \rightarrow \fM_{0, 0}(\alpha, \alpha')$ need not be surjective.  While this morphism is still given by $f_{b_1}f_{b_1^*}$, the description of $\fM_{0, 0}(\alpha, \alpha')$ is more complicated.  See \cite{Shmel2012} for details.
\end{rmk}

\begin{rmk}
\label{eigenvaluesrmk}
Note that the above construction can be generalized to conjugacy classes with arbitrary eigenvalues. Indeed, let $\zeta  = (\zeta_i) \in \C^n$.  Let $\theta \in \Z_{>0}^{n+1}$ and $\lambda = (-\zeta_1, \zeta_1 - \zeta_2, \dots, \zeta_{n-1} - \zeta_n)$.  Let $\alpha$ and $\alpha'$ be such that $\alpha_0 > \alpha_1 > \cdots > \alpha_n$.  If $\lambda \cdot \alpha = 0$, then $\fM_{\theta, \lambda}(\alpha, \alpha')$ is nonempty and we can describe it as follows:
\[
\fM_{\theta, \lambda}(\alpha, \alpha') \cong T^{*, \zeta}\Fl := \{(V_{\bullet},f)|V_{\bullet}\in \Fl \textrm{ and } (f - \zeta_{i}\Id)(V_i) \subset V_{i+1} \}.
\]
This is an affine bundle modeled on the cotangent bundle $T^*\Fl$ called a \emph{twisted cotangent bundle}.  As before, we can identify the projection $\pi: T^{*, \zeta}\Fl \rightarrow \fgl_{\alpha_0}(\C)$ onto the second component with the closure of a conjugacy class $\overline{C_{\nu}}$ with eigenvalues in $\zeta$.  Assuming that $\alpha_0 - \alpha_1 \ge \cdots \ge \alpha_{n-1} - \alpha_n \ge \alpha_n$, we get an upgraded version of the Kraft and Procesi theorem (see Lemma 9.1 in \cite{CB2003b}), which says $\fM_{0, \lambda}(\alpha, \alpha') \cong \overline{C_{\nu}}$.  The natural morphism $\fM_{\theta, \lambda}(\alpha, \alpha') \rightarrow \fM_{0, \lambda}(\alpha, \alpha')$ is once again just the projection $\pi$.
\end{rmk}

Note that in the two examples above the natural projective morphism $\pi: \fM_{\theta, \lambda}(\alpha, \alpha') \rightarrow \fM_{0, \lambda}(\alpha, \alpha')$ was a resolution of singularities.  Nakajima proves this is always the case (see Theorem 4.1 in \cite{Na1994} or Proposition 3.24 in \cite{Na1998}) under the conditions that $(\theta, \lambda)$ is generic and the subset of elements in $\fM_{0, \lambda}(\alpha, \alpha')$ coming from representations with trivial stabilizer group is nonempty.  In fact, in this case, the morphism $\pi$ is a \emph{symplectic resolution}.




\section{Conclusion}

We would like to conclude by briefly discussing some continuations for the topics we touched upon earlier.

\subsection{Quiver realizations of moduli spaces of sheaves}

The Beilinson spectral sequence and the subsequent quiver constructions of moduli space of sheaves on $\P^1$ and $\P^2$ are part of a general framework used to describe the derived category $D^b(X)$ of coherent sheaves on an algebraic variety $X$.  The existence of objects $E_1, \dots , E_m$ forming a \emph{strong exceptional collection} allows (see e.g. Section 8.3 in \cite{Huy2006}) for an equivalence between $D^b(X)$ and the derived category of modules over an algebra that can be realized as a path algebra of some quiver, possibly with additional relations.  Choosing the correct rigidity condition (such as global generation in Section 4 or triviality at infinity in Section 5), leads to a correspondence between objects rather than complexes of objects.  This can identify moduli spaces of rigidified sheaves on $X$ with a moduli space of quiver representations (possibly with relations).  One can see this approach in \cite{Ma2017}, where moduli spaces of Gieseker stable coherent sheaves on $\P^1$, $\P^2$, and $\P^1 \times \P^1$ are identified with moduli spaces of stable representations of certain quivers.

\subsection{Quiver realizations of moduli spaces of connections}

Consider the following generalization of the flag variety associated to the $A_n$ quiver in Example \ref{framedanex} and Section \ref{nakanex}.  Let $E$ be a rank $\alpha_0$ vector bundle on a smooth projective curve $X$.  Let $D = x_1 + \cdots x_k \subset X$ be a (reduced) divisor and let $0  = E_{in_i+1} \subset E_{in_i} \subset \cdots \subset E_{i1} \subset E_0 = E_{x_i}$ be a collection of flags in the fibers of $E$ over $x_i$.  The data $\mathbf{E} = (E,E_{ij})$ is called a \emph{parabolic bundle} on $X$.  Note that if $i = 1$, $X = \P^1$, and $E$ is trivial (with a chosen trivialization), then we can view each parabolic bundle $(E, E_j)$ as a flag in $\Fl(\alpha_0, \cdots, \alpha_n, 0)$, where $\alpha_j = \dim E_j$.

Similarly, we can generalize the twisted cotangent bundles of Section \ref{nakanex} as follows.  Let $E$ and $D$ be as above.  A \emph{logarithmic connection} (also called a connection with \emph{regular singularities}) on $X$ is a $\C$-linear morphism $\nabla: E \rightarrow E \otimes \Omega^1_{X}(D)$ satisfying $\nabla(fs) = s \otimes df + f \otimes \nabla(s)$ for $s \in E$ and $f \in \cO_{X}$.  Note this is an algebraic connection with poles of order at most $1$ at the points $x_i$.  It follows we can define the residue endomorphism $\Res_{x_i} \nabla \in \End(E_{x_i})$ at each point $x_i$.  If $\mathbf{E} = (E, E_{ij})$ is a parabolic bundle on $\P^1$ associated to $D$, then a \emph{$\zeta$-parabolic connection} for a complex vector $\zeta = (\zeta_{ij})$ is a logarithmic connection such that $(\Res_{x_i} \nabla - \zeta_{ij}\Id) E_{ij} \subset E_{ij+1}$.  Once again, setting $i = 1$ and considering only trivial (and trivialized) vector bundles identifies pairs $(\mathbf{E}, \nabla)$ of parabolic bundle on $\P^1$ with $\zeta$-parabolic connection and the elements of $T^{*, \zeta}\Fl$ (see Example \ref{eigenvaluesrmk}).

It is possible to identify parabolic bundles on $\P^1$ with finite-dimensional representations of a quotient of a certain quiver path algebra called a \emph{squid} (see e.g. \cite{CB2004}).  This leads to descriptions for moduli of pairs $(\mathbf{E}, \nabla)$ on $\P^1$ as quiver varieties with additional relations (see \cite{ASo2016}).  Furthermore, we may consider connections with poles of order greater than $1$.  These are said to have \emph{irregular singularities}.  In this case, there are also quiver variety constructions in \cite{Bo2008}, \cite{HY2014}, \cite{Hi2017} for moduli spaces of pairs on $\P^1$ as long as the underlying vector bundle is trivial.

\subsection{Multiplicative quiver varieties}

Rather than looking at connections $\nabla$ on $X$, we can instead consider the local systems of horizontal sections of $\nabla$.  The \emph{Riemann-Hilbert correspondence} defines an equivalence between the category of vector bundles with logarithmic connections on $X$ with the category of local systems on $X$.  Since the latter category is equivalent to the category of representations of the fundamental group of the punctured Riemann surface $X - D$, it is possible to define an (analytic) isomorphism between the moduli space $M_{DR}$ of pairs $(E, \nabla)$ and the corresponding variety $M_B$ of representations of $\pi_1(X - D)$.  

As with $M_{DR}$, it is possible to realize $M_B$ for $X = \P^1$ in terms of a kind of quiver variety.  However, one must work with $\GL(\alpha_0, \C)$, rather than $\fg_{\alpha_O}(\C)$, so a multiplicative version of the moment map construction is necessary.  This is given by the \emph{group-valued moment map} defined in \cite{AMM1998}.  In the case of $M_B$, the moment map fiber may be described in terms of representations of a \emph{multiplicative preprojective algebra} (see \cite{CBS2006}), and taking the quotient results in a \emph{multiplicative quiver variety} studied in \cite{Ya2008}.

The Riemann-Hilbert correspondence generalizes to vector bundles with connections that have higher order poles.  In this case, the varieties parametrizing the corresponding monodromy representations together with the \emph{Stokes data} that encodes behavior around the irregular singularities have been studied extensively by Boalch (see e.g. \cite{Bo2015} and \cite{Bo2017}).

\subsection{Kac-Moody algebras}

The matrix associated to the symmetrized Euler form $(- , -)$ for a quiver $Q$ with no loops is a generalized Cartan matrix, which in turn can be used to define a \emph{Kac-Moody Lie algebra $\fg_Q$} (see \cite{Kac1990} for details).  There is a close relationship between $\fg_Q$ and the representations of $Q$.  For example, a famous theorem of Kac (\cite{Kac1980} and \cite{Kac1982}) tells us that an indecomposable representation of $Q$ with dimension vector $\alpha$ exists if and only if $\alpha$ is a root of $\fg_Q$.  There are geometric implications as well.  Crawley-Boevey showed in \cite{CB2001} that if $\alpha$ is a positive root satisfying certain additional numerical conditions, then the fibers $\inv{\mu}(\lambda)$ of the moment map $\mu: \Rep(\overline{Q}, \alpha) \rightarrow \fg$ are nonempty irreducible complete intersections with simple general element.  Thus, we can see that the geometry of quiver varieties can be related to the algebra $\fg_Q$.

A converse relationship may be observed in \cite{Na1994} and \cite{Na1998}. Nakajima showed that the \emph{modified enveloping algebra} $\tilde{U}(\fg_Q)$ of $\fg_Q$ has a representation in the Borel-Moore homology of the fiber $\inv{\pi}(0)$ of the natural morphism $\pi: \fM_{\theta, 0}(\alpha, \alpha') \rightarrow \fM_{0, 0}(\alpha, \alpha')$.  A further construction for representations of the (modified) \emph{quantized enveloping algebras} of a Kac-Moody algebra, using equivariant K-groups of quiver varieties in place of Borel-Moore homology, may be seen in \cite{Na2001}. 

\subsection{Moduli spaces of instantons}

One of Nakajima's original motivations in constructing quiver varieties was to study the solutions of the \emph{Yang-Mills equations}.  These equations arise in physics as a generalization of Maxwell's equations of electromagnetism, and certain minimizing solutions are called \emph{instantons}.  Mathematically, these may be understood as self-dual or anti-self dual connections on $G$-bundles on a 4-dimensional Riemannian manifold $M$ (see \cite{At1978} for details).  In \cite{ADHM1978} the moduli space of instantons on $M = \R^4$ (or its compactification $S^4 = \R^4 \cup \{\infty\}$) is described in terms of linear algebraic data.  After identifying $\R^4$ with $\C^2$, instantons on $S^4$ together with a framing at $\infty$ may be interpreted in terms of holomorphic vector bundles on $\P^2$ framed on the line at infinity $\ell_{\infty}$ (i.e. there is a bijection between the moduli spaces; see \cite{Don1984} for details).  The ADHM construction results in the variety $\fM_{\theta, 0}(\alpha, \alpha')$ corresponding to the Jordan quiver $J$ we have already seen in Sections 5 and 11.  Thus the moduli space of framed instantons on $S^4$ may be interpreted as a quiver variety.

From the perspective of complex geometry, quiver varieties are hyperk\"{a}hler quotients (see \cite{Na1994}).  Following this interpretation, the quiver varieties corresponding to affine ADE Dynkin graphs are a type of hyperk\"{a}hler manifold called an \emph{ALE space} (see e.g. \cite{Kro1989} or \cite{Na1994}).  Similarly to the ADHM description, the moduli spaces of framed instantons on an ALE space may be interpreted as a quiver variety (see Proposition 2.15 in \cite{Na1994}).  There have been multiple generalizations of the ADHM results, as well as many other quiver variety constructions for moduli spaces of instantons.  See, for example, \cite{Ne2017} for a modern perspective.

\bibliographystyle{alpha}
\bibliography{ds}
\end{document}